\documentclass{article}

\usepackage[utf8]{inputenc}
\usepackage{amsmath,amsthm}
\usepackage{amssymb}
\usepackage{mathrsfs}
\usepackage{listings}
\usepackage{graphicx,multirow}
\usepackage{color,soul}
\usepackage{esint}
\usepackage{algorithm}
\usepackage{algpseudocode}
\usepackage[margin=0.98in]{geometry}
\usepackage{hyperref}
\usepackage{stmaryrd} 
\usepackage{xspace}
\usepackage{xr}
\usepackage{titling}

\hypersetup{hidelinks,colorlinks,citecolor=magenta,linkcolor=blue}
\usepackage{biblatex}
\usepackage{etoolbox}
\makeatletter
\patchcmd{\thebibliography}{%
   \section*{\refname}\@mkboth{\MakeUppercase\refname}{\MakeUppercase\refname}%
}{}{}{}
\makeatother
\usepackage[font=small,labelfont=bf]{caption}
\usepackage{authblk}
\usepackage{fancyhdr}
\usepackage{subfigure,bm,bbm}
\usepackage[normalem]{ulem}
\DeclareMathOperator\sign{sign}
\DeclareMathOperator\diag{diag}

\newcommand{\net}{DSP-WENO\xspace}

\newcommand{\con}{{\bm{u}}}
\newcommand{\f}{{\bm{f}}}
\newcommand{\ent}{{\bm{v}}}
\newcommand{\z}{{\bm{z}}}
\newcommand{\iph}{{i + \frac{1}{2}}}
\newcommand{\imh}{{i - \frac{1}{2}}}
\newcommand{\velvec}{\textbf{u}}
\newcommand{\vel}{\text{u}}
\newcommand{\temp}{\mathcal{T}}
\newcommand{\N}{{\bm{\mathcal{N}}}}
\newcommand{\Ro}{{\mathbb{R}}}
\newcommand{\bPhi}{{\boldsymbol{\Phi}}}
\newcommand{\Ff}{{\bm{\mathcal{F}}}}
\newcommand{\Ss}{{\bm{\mathcal{S}}}}
\newcommand{\V}{{\bm{\mathcal{V}}}}
\newcommand{\argmin}[1]{\underset{#1}{\arg\min}}
\newcommand{\chf}[1]{ { \mathcal{X} }_{#1}}

\newtheorem{lem}{Lemma}

\newtheorem{defn}{Definition}
\newtheorem{remark}{Remark}

\usepackage{tikz}
\usetikzlibrary{decorations.pathreplacing,calc}
\newcommand{\tikzmark}[1]{\tikz[overlay,remember picture] \node (#1) {};}

\newcommand*{\AddNote}[4]{%
    \begin{tikzpicture}[overlay, remember picture]
        \draw [decoration={brace,amplitude=0.5em},decorate,ultra thick]
            ($(#3)!(#1.north)!($(#3)-(0,1)$)$) --  
            ($(#3)!(#2.south)!($(#3)-(0,1)$)$)
                node [align=center, text width=2.5cm, pos=0.5, anchor=west] {#4};
    \end{tikzpicture}
}%

\usepackage{amsfonts}
\usepackage{graphicx}
\usepackage{epstopdf}
\ifpdf
  \DeclareGraphicsExtensions{.eps,.pdf,.png,.jpg}
\else
  \DeclareGraphicsExtensions{.eps}
\fi

\begin{document}

\title{Learning WENO for entropy stable schemes to solve conservation laws}
\author[1]{Philip Charles \thanks{charlesp@umd.edu}}
\author[1]{Deep Ray \thanks{deepray@umd.edu}}
\affil[1]{Department of Mathematics, University of Maryland at College Park}
\date{}

\maketitle

\begin{abstract}
Entropy conditions play a crucial role in the extraction of a physically relevant solution for systems of conservation laws, thus motivating the construction of entropy stable schemes that satisfy a discrete analogue of such conditions. TeCNO schemes (Fjordholm et al. 2012) \cite{fjordholm2012arbitrarily} form a class of arbitrary high-order entropy stable finite difference solvers, which require specialized reconstruction algorithms satisfying the \textit{sign property} at each cell interface. Third-order weighted essentially non-oscillatory (WENO) schemes called SP-WENO (Fjordholm and Ray, 2016) \cite{fjordholm2016sign} and SP-WENOc (Ray, 2018) \cite{ray2018third} have been designed to satisfy the sign property. However, these WENO algorithms can perform poorly near shocks, with the numerical solutions exhibiting large spurious oscillations. In the present work, we propose a variant of the SP-WENO, termed as Deep Sign-Preserving WENO (\net), where a neural network is trained to learn the WENO weighting strategy. The sign property and third-order accuracy are strongly imposed in the algorithm, which constrains the WENO weight selection region to a convex polygon. Thereafter, a neural network is trained to select the WENO weights from this convex region with the goal of improving the shock-capturing capabilities without sacrificing the rate of convergence in smooth regions. The proposed synergistic approach retains the mathematical framework of the TeCNO scheme while integrating deep learning to remedy the computational issues of the WENO-based reconstruction. We present several numerical experiments to demonstrate the significant improvement with \net over the existing variants of WENO satisfying the sign property.
\end{abstract}

\section{Introduction}

Hyperbolic systems of conservation laws dictate the behavior of quantities that are conserved in time as the system evolves. These systems of partial differential equations (PDEs) are ubiquitous across many disciplines. A generic system of conservation laws in one spatial dimension can be mathematically expressed as

\begin{equation}\label{eqn:cons_law}
\begin{aligned}
\frac{\partial}{\partial t}\con(x,t) + \frac{\partial}{\partial x}\f(\con(x,t)) &= \mathbf{0} \quad &&\forall \ x \in \Omega \subset \mathbb{R}, \ t \in \mathbb{R}^+ \\
\con(x,0) &= \con_0(x) \quad &&\forall \ x \in \Omega,
\end{aligned}
\end{equation}

\noindent where $\con$: $\Omega\times\mathbb{R}^+ \xrightarrow{} \mathbb{R}^d$ represents a vector of $d$ conserved variables and $\f$: $\mathbb{R}^d \xrightarrow{} \mathbb{R}^d$ is a smooth flux, representing the flow of the conserved quantities.

However, it is well-known that solutions to nonlinear conservation laws can develop discontinuities in finite time even with smooth initial conditions. Thus, solutions to \eqref{eqn:cons_law} must be understood in the weak (distributional) sense. Moreover, since weak solutions are not unique in general, additional constraints in terms of the \textit{entropy conditions} must be imposed such that a physically relevant weak solution is extracted. Assume that for \eqref{eqn:cons_law}, $\eta(\con)$ is a convex entropy function and $q(\con)$ is the entropy flux function satisfying the compatibility condition $\nabla_{\con} q(\con)^\top = \nabla_{\con} \eta(\con)^\top \nabla_\con \f(\con)$.
The solution $\con$ of \eqref{eqn:cons_law} is said to be an entropy solution if it satisfies the following inequality for all admissible entropy pairs $(\eta(\con),q(\con))$ associated with \eqref{eqn:cons_law}: 
\begin{equation}\label{eqn:ent_ineq}
\frac{\partial}{\partial t}\eta(\con) + \frac{\partial}{\partial x}q(\con) \leq 0,
\end{equation}
which is understood in the weak sense. The entropy is conserved for smooth solutions, and thus \eqref{eqn:ent_ineq} is taken to be an equality, i.e., $\eta$ is the solution of a conservation law. In contrast, entropy is dissipated near shocks, thus the inequality in \eqref{eqn:ent_ineq} is strict near such discontinuities. Scalar conservation laws have unique entropy solutions \cite{Kruzkov}, but the uniqueness 
is not guaranteed for general systems of conservation laws \cite{deLellis2010,Chiodaroli}. However, the entropy conditions provide the only non-linear estimates currently available for generic systems of conservation laws \cite{dafermos2005hyperbolic}, and are thus essential. In practice, the entropy condition \eqref{eqn:ent_ineq} is considered with respect to a specific choice of $(\eta(\con),q(\con))$ for a general system \eqref{eqn:cons_law}.

It is meaningful to develop \textit{entropy stable schemes}, i.e., schemes satisfying a discrete version of \eqref{eqn:ent_ineq}, to solve systems of conservation laws. Further, it is essential for such methods to be high-order accurate while being capable of capturing discontinuities without spurious oscillations. 
In finite difference methods, the computational domain is partitioned into cells, on which a discrete version of the conservation law is formulated. Point values of the solution in the cells are evolved in time using time integration techniques, such as strong stability preserving Runge-Kutta methods \cite{gottlieb2001strong}. In \cite{tadmor2003entropy}, a novel approach to constructing entropy stable finite difference schemes was proposed, which comprises two steps: (1) beginning with a second-order \textit{entropy conservative} flux such that entropy is conserved locally and (2) adding an artificial dissipation term to ensure entropy stability.

TeCNO schemes \cite{fjordholm2012arbitrarily} are a popular class of arbitrary high-order entropy stable finite difference schemes which build on the formulation considered in \cite{tadmor2003entropy}. These schemes consist of augmenting high-order entropy conservative fluxes with high-order numerical diffusion, which rely on specialized polynomial reconstruction algorithms. Crucially, these reconstructions must satisfy a \textit{sign property} (see Lemma \ref{lem:es_condition}) at each cell interface so that entropy stability is ensured. 
Although any reconstruction method with the sign property can be used to formulate entropy stable TeCNO schemes, only a handful of such algorithms are currently available \cite{fjordholm2012arbitrarily,cheng2016third,cheng2019fourth,fjordholm2016sign,ray2018third}. 
The essentially non-oscillatory (ENO) reconstruction method, which adaptively chooses the smoothest stencil for reconstruction, satisfies the sign property as demonstrated in \cite{fjordholm2013eno}. To overcome some of the computational challenges encountered when using ENO, third-order weighted ENO (WENO) reconstructions were proposed in \cite{fjordholm2016sign,ray2018third} which guarantee the sign property. However, these WENO schemes (termed as SP-WENO), while ensuring entropy stability in the TeCNO framework, are susceptible to spurious oscillations near discontinuities. We will highlight this issue in greater detail in the present work. We also mention here an alternate strategy \cite{pandey2023sign} based on TeCNO schemes, where the sign property of reconstructions is not required. Instead, a specialized diffusion operator is constructed to ensure entropy stability, which can be used with any high-order reconstruction algorithm. 

The last few years have witnessed a surge in the use of deep learning-based strategies to solve problems in scientific computing. Key examples includes building surrogates to solve PDEs \cite{raissi2019,pinns_review,deeponet,li2020fourier,varmion}, learning parameter-to-observable maps \cite{lye2020,lye2021iterative,huang2024operator}, solving Bayesian inference problems \cite{goh2022,patel2022solution,pigans,dasgupta2023conditional,raycgan23}, and learning closure models \cite{singh2016,beck2019,maulik2020}. Among these is the class of deep learning (DL) approaches which aim to enhance existing numerical solvers, instead of replacing them. The philosophy here is to first identify computational bottlenecks in the solver, and then use domain knowledge to train specialized neural networks to replace the bottleneck, keeping the rest of the solver intact. Thus, this synergistic approach leverages the ``best of both worlds", leading to a better numerical solver aided by DL. For instance, in \cite{ray2018artificial,ray2019detecting}, deep-learned (universal) troubled-cell indicators were designed for discontinuous Galerkin (DG) schemes solving conservation laws, to identify (classify) elements containing discontinuities. Similar DL-based shock capturing strategies have also been explored in \cite{discacciati2020controlling,beck2020,zeifang2021,schwander2021,bruno2022,veiga2023,sun2020convolution,yu2022multi}.

DL techniques have also been used to learn ENO/WENO-type reconstructions. In \cite{DeRyck2022}, it was shown that ENO (and some of its variants) are equivalent to deep ReLU neural networks.
A six-point ENO-type reconstruction was proposed in \cite{li2022six} where the stencil selection is performed using a neural network. However, when used to solve conservation laws, some numerical experiments were accompanied by a degeneracy in the expected order of accuracy. In \cite{stevens2020}, a neural network was trained to develop an optimal variant of the classical WENO-JS \cite{jiangshu96}. 
Despite leading to sharper shock profiles in numerical solutions, 
the expected accuracy was not always observed with smooth solution. A DL-based fifth-order WENO called WENO-DS was proposed in \cite{kossaczka2021enhanced}
with improved shock-capturing capabilities while retaining the formal order of accuracy. However WENO-DS is not model agnostic, and needs to be retrained if the conservation law is changed. An extension to the 2D Euler equations was proposed in \cite{kossaczka2024deep}.

In the present work, we are interested in using DL to recover reconstruction algorithms constrained to satisfy useful physical properties. In particular, we construct a variant of SP-WENO, which we call \net, where a neural network predicts the weights associated with the reconstruction, while ensuring that the sign property (along with other crucial constraints) are satisfied. These constraints are \textit{strongly imposed} which results in a convex polygonal region for the WENO weight selection. Then, a neural network is trained on suitably generated training samples to learn the weight selection algorithm of the \net for appropriate behavior near discontinuities and smooth regions. The proposed \net is constructed to be third-order accurate, which is also demonstrated numerically. When used in the TeCNO framework, \net overcomes the computational issues faced by ENO and the existing SP-WENO strategies. We re-iterate that the \net only replaces the reconstruction needed in the high-order diffusion term, keeping the rest of the TeCNO framework intact. Thus the proposed method can be seen as a \textit{deep learning-based enhancement} of an existing numerical algorithm. We remark here that a single \net network is trained (offline) to be used with any conservation laws, i.e., the proposed algorithm is agnostic to the specific PDE model being solved.

The rest of the paper is organized as follows. In Section \ref{sec:fin_dif_schemes}, we describe the framework of high-order entropy stable finite difference schemes. In Section \ref{sec:sign_pre_weno}, we introduce the formulation and properties of existing sign-preserving WENO reconstructions. The construction of our DL-based method is explained in Section \ref{sec:SP-WENO-DL}. The results of various numerical tests including one-dimensional and two-dimensional scalar and systems of conservation laws are presented in Section \ref{sec:numerical}. Final conclusions and future directions are discussed in Section \ref{sec:conclusion}.

\section{Finite Difference Schemes/Entropy Stable Schemes}\label{sec:fin_dif_schemes}

We consider a one-dimensional finite difference formulation for ease of discussion, which can easily be extended to higher dimensions. We partition the spatial domain $\Omega = [a,b]$ into $N$ disjoint cells $I_i=[x_{i-\frac{1}{2}},x_{i+\frac{1}{2}})$ of uniform length $h=(b-a)/N$ with cell center $x_i$ where
\begin{equation*}
x_{i+\frac{1}{2}} = a + ih, \quad \forall \ 0 \leq i \leq N \ \ \  \text{and} \ \ \  x_i = \frac{x_{i-\frac{1}{2}}+x_{i+\frac{1}{2}}}{2} \quad \forall \  1 \leq i \leq N.
\end{equation*} Keeping time continuous, the semi-discrete finite difference scheme for \eqref{eqn:cons_law} is expressed as
\begin{equation}\label{semi_discrete}
\frac{d\con_i(t)}{dt} + \frac{1}{h}(\f_{i+\frac{1}{2}}-\f_{i-\frac{1}{2}}) = 0.
\end{equation}
Here, $\con_i(t)$ approximates the point values of the solution to \eqref{eqn:cons_law} at cell center $x_i$ and time $t$, while $\f_{i+\frac{1}{2}}$ is a consistent, conservative numerical approximation of the flux $\f$ at the cell interface $x_{i+\frac{1}{2}}$. We seek entropy stable schemes to approximate \eqref{eqn:cons_law} which satisfy a discrete version of the entropy condition \eqref{eqn:ent_ineq}. As described in \cite{tadmor2003entropy}, we begin by constructing an entropy conservative scheme which satisfies the discrete entropy equality
\begin{equation}\label{disc_ent_eq}
\frac{d\eta(\con_i)}{dt} + \frac{1}{h}\left(\Tilde{q}_\iph-\Tilde{q}_\imh\right) = 0,
\end{equation}
where $\Tilde{q}_\iph$ is a numerical entropy flux consistent with entropy flux $q$ in \eqref{eqn:ent_ineq}.

We denote the undivided jump and average across the interface $x_{i+\frac{1}{2}}$ by $\Delta(\cdot)_{i+\frac{1}{2}} = (\cdot)_{i+1} - (\cdot)_i$ and $\overline{(\cdot)}_{i+\frac{1}{2}} =\big((\cdot)_{i+1} + (\cdot)_i\big)/2$, respectively. 

Additionally, we introduce the \textit{entropy potential} $\Psi(\con) := \ent(\con)^\top\f(\con) - q(\con)$, where $\ent(\con) = \nabla_\con \eta(\con)$ is the vector of \textit{entropy variables}. A sufficient condition \cite{tadmor2003entropy} for the two-point flux $\Tilde{\f}_{i+\frac{1}{2}}=\Tilde{\f}_{i+\frac{1}{2}}(\con_i,\con_{i+1})$ to be entropy conservative is given by 
\begin{equation}\label{ent_cons_condition}
\big(\Delta \ent_{i+\frac{1}{2}}\big)^\top\Tilde{\f}_{i+\frac{1}{2}} = \Delta \Psi_{i+\frac{1}{2}}.
\end{equation}
The expression \eqref{ent_cons_condition} provides a recipe to construct second-order entropy conservative fluxes, and leads to a unique numerical flux for scalar conservation laws (given $\eta)$. For systems of conservation laws, it is possible to carefully construct numerical fluxes that satisfy \eqref{ent_cons_condition} with additional desirable properties \cite{ismail2009,chandrashekar2013kinetic,winter2016,gouasmi2019}. Arbitrary high-order entropy conservative fluxes $\Tilde{\f}^p$ (of order $p$) can be constructed using linear combinations of second-order two-point entropy conservative fluxes \cite{lefloch2002fully}. For example, the fourth-order $(p=4)$ entropy conservative flux is given by
\begin{equation}\label{4th_order_flux}
\Tilde{\f}_{i+\frac{1}{2}}^{4} = \frac{4}{3}\Tilde{\f}(\con_{i},\con_{i+1}) - \frac{1}{6}\left(\Tilde{\f}(\con_{i-1},\con_{i+1}) + \Tilde{\f}(\con_{i},\con_{i+2})\right).
\end{equation}

Entropy is conserved for smooth solutions and thus using an entropy conservative scheme is meaningful. However, entropy is dissipated near discontinuities in accordance to \eqref{eqn:ent_ineq}. Hence, an entropy variable-based numerical dissipation term is added to the entropy conservative numerical flux 
\begin{equation}\label{ent_stable_flux}
\f_{i+\frac{1}{2}} = \Tilde{\f}^{p}_{i+\frac{1}{2}} - \frac{1}{2}\mathbf{D}_{i+\frac{1}{2}}\Delta\ent_{i+\frac{1}{2}},
\end{equation}
where $\mathbf{D}_{i+\frac{1}{2}}\succeq0$ (a positive semi-definite matrix) is evaluated at some suitable averaged states. The numerical flux \eqref{ent_stable_flux} leads to the satisfaction of the following discrete entropy inequality \cite{fjordholm2012arbitrarily}
\begin{equation}\label{eqn:disc_ent_stab}
\frac{d\eta(\con_i)}{dt} + \frac{1}{h}\left(q_\iph-q_\imh\right) \leq 0,
\end{equation}
where $q_\iph$ is a consistent numerical entropy flux.
While any positive semi-definite matrix ensures entropy stability in the sense of \eqref{eqn:disc_ent_stab}, we choose the form $\mathbf{D}_{i+\frac{1}{2}} = \mathbf{R}_{i+\frac{1}{2}}\Lambda_{i+\frac{1}{2}}\mathbf{R}^{\top}_{i+\frac{1}{2}}$. Here, $\mathbf{R}$ is a matrix consisting of the right eigenvectors of the flux Jacobian and $\mathbf{\Lambda}$ is a nonnegative diagonal matrix that depends on the eigenvalues of the flux Jacobian. Specifically, we choose the \textit{Roe-type} diffusion matrix with $\mathbf{\Lambda}=\diag{(|\lambda^1|,...,|\lambda^d|)}$. See \cite{tadmor2003entropy} for other choices for the diffusion matrix.

Note that the term $\Delta \ent_{i+\frac{1}{2}}$ in \eqref{ent_stable_flux} is $\mathcal{O}(h)$. Therefore, the numerical scheme that results from using \eqref{ent_stable_flux} is only first-order accurate regardless of the accuracy of the entropy conservative flux. Since the accuracy of the scheme is limited by the diffusion term, we construct a higher-order diffusion term by suitably reconstructing the jump in entropy variables at the cell interfaces.

We consider the interface at $x_{i+\frac{1}{2}}$ between the cells $I_i$ and $I_{i+1}$ and reconstruct from the left and right of this interface. We define the (locally) scaled entropy variables $\z=\mathbf{R}^{\top}_\iph\ent$ corresponding to this interface. The flux \eqref{ent_stable_flux} can thus be expressed as
\begin{equation}\label{ent_stable_flux_scaled}
\f_{i+\frac{1}{2}} = \Tilde{\f}^{p}_{i+\frac{1}{2}} - \frac{1}{2}\mathbf{R}_{i+\frac{1}{2}}\Lambda_{i+\frac{1}{2}}\Delta\z_{i+\frac{1}{2}}.
\end{equation}
Let $\z_i(x)$ and $\z_{i+1}(x)$ be polynomial reconstructions of the scaled entropy variables in $I_i$ and $I_{i+1}$, respectively. We denote the reconstructed values and jump at the cell interface by
\begin{equation*}
\z_{i+\frac{1}{2}}^- = \z_i(x_{i+\frac{1}{2}}), \hspace{2mm} \z_{i+\frac{1}{2}}^+ = \z_{i+1}(x_{i+\frac{1}{2}}), \hspace{2mm} \llbracket\z\rrbracket_{i+\frac{1}{2}} = \z_{i+\frac{1}{2}}^+ - \z_{i+\frac{1}{2}}^-.
\end{equation*}
Replacing the original jump $\Delta\z_{i+\frac{1}{2}}$ in \eqref{ent_stable_flux_scaled} by the reconstructed jump $\llbracket\z\rrbracket_{i+\frac{1}{2}}$  will lead to a higher-order accurate scheme. However, this scheme is not guaranteed to be entropy stable. The following lemma provides a sufficient condition on the reconstruction algorithm that ensures entropy stability.

\begin{lem}\label{lem:es_condition}
\textnormal{(\cite{fjordholm2012arbitrarily})} For each interface $x_{i+\frac{1}{2}}$, if the reconstruction satisfies the sign property
\begin{equation}\label{sign_prop}
\sign{(\llbracket\z\rrbracket_{i+\frac{1}{2}})} = \sign{(\Delta\z_{i+\frac{1}{2}})},
\end{equation}
then the scheme with the following numerical flux is entropy stable
\begin{equation}\label{high_order_ent_stable_flux}
\f_{i+\frac{1}{2}} = \Tilde{\f}^{p}_{i+\frac{1}{2}} - \frac{1}{2}\mathbf{R}_{i+\frac{1}{2}}\Lambda_{i+\frac{1}{2}}\llbracket\z\rrbracket_{i+\frac{1}{2}}.
\end{equation}
\end{lem}

High-order entropy stable schemes described in Lemma \ref{lem:es_condition} are called \textit{TeCNO} schemes, which rely on reconstructions satisfying the sign property \eqref{sign_prop}. However, only a handful of reconstructions are known to satisfy the sign property. In \cite{fjordholm2013eno}, it was proven that ENO reconstructions satisfy this critical condition. In \cite{fjordholm2016sign}, third-order WENO schemes (known as SP-WENO) were designed to possess the sign property, which serves as the starting point for the novel reconstruction proposed in the present work. We discuss the SP-WENO formulation in Section \ref{sec:sp_weno}.

\section{Sign-Preserving WENO Reconstructions}\label{sec:sign_pre_weno}

A typical WENO reconstruction \cite{jiangshu96} produces a $(2k-1)$th-order accurate reconstruction by taking a convex combination of the $2k-1$ candidate polynomials considered in the $k$th-order ENO reconstruction. A third-order WENO has the following form for the left and right reconstructions of the variable $z$ at the interface $x_{i+\frac{1}{2}}$:

\begin{equation}\label{eqn:weno_recon}
\begin{aligned}
z_{i+\frac{1}{2}}^- {=}  \frac{w_{0}(z_i {+} z_{i+1})+ w_{1}(3z_i {-}z_{i-1})}{2}, \ 
z_{i+\frac{1}{2}}^+ {=}  \frac{\Tilde{w}_{0}(3z_{i+1} {-} z_{i+2}) + \Tilde{w}_{1}(z_i {+} z_{i+1})}{2},
\end{aligned}
\end{equation}
where the weights $w_0,w_1,\Tilde{w}_0,\Tilde{w}_1$ (at each interface) must be chosen to: (i) ensure third-order accuracy of the reconstructions $z_{i+\frac{1}{2}}^{\pm}$ for smooth solutions and (ii) provide minimal weight to linear polynomials on stencils containing discontinuities. In general, WENO reconstructions do not satisfy the sign property, and their use in the TeCNO framework described by \eqref{high_order_ent_stable_flux} does not yield entropy stable schemes. Thus, a WENO reconstruction must be specifically designed to be sign-preserving, i.e., it must satisfy the sign property \eqref{sign_prop}.

\subsection{SP-WENO \cite{fjordholm2016sign}}\label{sec:sp_weno}

SP-WENO was the first variant of WENO schemes designed to satisfy the sign property. In this framework, the weights in \eqref{eqn:weno_recon} are given as
\begin{equation}\label{eqn:spweno_wts}
   w_0=\frac{3}{4}+2C_1, \quad  \Tilde{w}_0=\frac{1}{4}-2C_2, \quad w_1=1-w_0, \quad \Tilde{w}_1=1-\Tilde{w}_0.
\end{equation}
We define the jump ratios at the interface $x_{i+\frac{1}{2}}$ as $\theta_i^-:=\Delta z_{i+\frac{1}{2}}/\Delta z_{i-\frac{1}{2}}$, $\theta_i^+ := 1/\theta_i^-$ and the additional terms
\begin{equation}\label{eqn:psi_pm}
\psi_{i+\frac{1}{2}}^+ := \frac{1-\theta_{i+1}^-}{1-\theta_{i}^+}, \quad \psi_{i+\frac{1}{2}}^- = \frac{1}{\psi_{i+\frac{1}{2}}^+}. 
\end{equation}
Then the perturbations $C_1$, $C_2$ in SP-WENO are determined as
\begin{flalign*}
C_1(\theta_i^+,\theta_{i+1}^-) = \begin{cases}
\frac{1}{8}\left(\frac{\kappa^+}{(\kappa^+)^2+(\kappa^-)^2}\right) &\text{ if } \theta_i^+\neq1, \psi^+<0, \psi^+\neq-1 \\
0 &\text{ if } \theta_i^+\neq1, \psi^+ = -1 \\
-\frac{3}{8} &\text{ if } \theta_i^+=1 \text{ or } \psi^+ \geq 0, |\theta_i^+|\leq1 \\
\frac{1}{8} &\text{ if } \psi^+ \geq 0, |\theta_i^+| > 1
\end{cases},
\end{flalign*}
and $C_2(\theta_i^+,\theta_{i+1}^-) = C_1(\theta_{i+1}^-,\theta_{i}^+)$, where
\begin{flalign*}
&\kappa^+(\theta_i^+,\theta_{i+1}^-) := \begin{cases}
\frac{1}{1+\psi^+} &\text{ if } \theta_i^+\neq1, \psi^+\neq-1\\
1 &\text{ otherwise}
\end{cases}, \hspace{5mm} \kappa^-(\theta_i^+,\theta_{i+1}^-) := \kappa^+(\theta_{i+1}^-,\theta_{i}^+).
\end{flalign*}

We summarize the key properties that SP-WENO satisfies:
\begin{enumerate}
    \item \textbf{Consistency}: The weights obey $0\leq w_0,w_1,\Tilde{w}_0,\Tilde{w}_1\leq1$, or equivalently $-3/8\leq C_1,C_2\leq1/8$.
    \item \textbf{Sign Property}: The reconstructed values \eqref{eqn:weno_recon} satisfy the condition \eqref{sign_prop}. Further, we can show that the reconstructed jump can be re-written as
    \begin{equation}\label{recon_jump}
    \llbracket z\rrbracket_{i+\frac{1}{2}} = \frac{1}{2}[\Tilde{w}_0(1-\theta_{i+1}^-)+w_1(1-\theta_i^+)]\Delta z_{i+\frac{1}{2}},
    \end{equation}
    which leads to the following constraint (equivalent to the sign property)
    \begin{equation}\label{eqn:sign_prop_constraint}
    [\Tilde{w}_0(1-\theta_{i+1}^-)+w_1(1-\theta_i^+)]\geq0.
    \end{equation}
\item \textbf{Negation Symmetry}: The weights remain unchanged under the transformation $z\mapsto -z$. 
A sufficient condition to ensure this property is to choose $C_1,C_2$ to be functions of features invariant to this negation transformation. For example,
\begin{equation}\label{neg_sym}
C_k = C_k(\theta_i^+,\theta_{i+1}^-,|\Delta z_{i+\frac{1}{2}}|, (|z_i| + |z_{i+1}|)) \hspace{3mm} k = 1,2.
\end{equation}
\item \textbf{Mirror Property}: Mirroring the solution about interface $x_{i+\frac{1}{2}}$ should also mirror the weights about the interface. Assuming negation symmetry holds and $C_1,C_2$ are of the form \eqref{neg_sym}, the mirror property is ensured if
\begin{equation}\label{mirror_prop}
C_1(a,b,c,d)=C_2(b,a,c,d) \hspace{2mm} \forall a,b,c,d\in\mathbb{R}.
\end{equation}
\item \textbf{Inner Jump Condition}: This condition ensures that the reconstruction is locally monotonicity-preserving. For each cell $i$:
\begin{equation}\label{inner_jump_cond}
\sign{\left(z_{i+\frac{1}{2}}^--z_{i-\frac{1}{2}}^+\right)} = \sign{\left(\Delta z_{i+\frac{1}{2}}\right)} = \sign{\left(\Delta z_{i-\frac{1}{2}}\right)},
\end{equation}
whenever the second equality holds. With consistent weights, this property is automatically satisfied.
\item \textbf{Bound on jumps:} The reconstructed jump has the bound
\begin{equation}\label{jump_bound}
|\llbracket z\rrbracket_{i+\frac{1}{2}}| \leq 2|\Delta z_{i+\frac{1}{2}}|.    
\end{equation}
\end{enumerate}

\subsection{SP-WENOc \cite{ray2018third}}\label{sec:sp_wenoc}
While SP-WENO leads to an entropy stable scheme, numerical results presented in \cite{ray2018third} (also see Section \ref{sec:numerical}) demonstrate its poor performance near discontinuities in the TeCNO framework. As discussed in \cite{ray2018third}, this issue can be attributed to the fact that the reconstructed jump $\llbracket z\rrbracket_{i+\frac{1}{2}}$ is zero in most cases, resulting in zero numerical diffusion (see \eqref{high_order_ent_stable_flux}) in these regions. While the absence of numerical diffusion may be acceptable in smooth regions, this can cause Gibbs oscillations near discontinuities. The most important problematic cases lie in the so-called \textit{C-region} \cite{ray2018third}, where the jump ratios are either $\theta_{i}^+<1$, $\theta_{i+1}^->1$ or $\theta_{i}^+>1$, $\theta_{i+1}^-<1$.

SP-WENOc seeks to remedy this by introducing a small perturbation $\mathcal{G}$ to ensure that the reconstructed jump is nonzero in the C-region. The modified jump is taken to be of the form
\begin{equation*}
\llbracket z\rrbracket_{i+\frac{1}{2}} = \frac{1}{2}\left[\Tilde{w}_0(1-\theta_{i+1}^-)+w_1(1-\theta_i^+) + \mathcal{G}\right]\Delta z_{i+\frac{1}{2}}.
\end{equation*}
where $\mathcal{G}$ is chosen as
\begin{equation}
\mathcal{G}=\left(\min{\left(\frac{\left|\Delta z_{i+\frac{1}{2}}\right|}{0.5(|z_i|+|z_{i+1}|)},\left|\Delta z_{i+\frac{1}{2}}\right|\right)}\right)^3.
\end{equation}
The perturbed jump can be realized by using the following modifications to $C_1,C_2$ 
\begin{equation}
\overline{C}_1=C_1-\frac{1}{4}\frac{\mathcal{G}}{(1-\theta_i^+)}, \hspace{3mm} \overline{C}_2=C_2-\frac{1}{4}\frac{\mathcal{G}}{(1-\theta_{i+1}^-)}.
\end{equation}

SP-WENOc retains all of the properties enumerated in Section \ref{sec:sp_weno} with the exception of the bound on the reconstructed jumps. SP-WENOc certainly injects more diffusion near discontinuities and mitigates the oscillatory behavior to some extent. However, the overshoots present in solutions, while reduced, are still significant in a multitude of test cases when compared to the performance of ENO reconstructions in the TeCNO framework (see Section \ref{sec:numerical}).

\subsection{Feasible Region for SP-WENO}\label{sec:feasible_region}

The SP-WENO and SP-WENOc formulations are just two possible weight selection strategies satisfying the constraints that guarantee consistency, sign-preservation, and third-order accuracy. Note that none of these constraints describe the behavior of solutions near shocks, thus SP-WENO and SP-WENOc are not necessarily constructed with precise shock-capturing in mind. That is to say, there are potentially other \textit{SP-WENO variants} that perform better near discontinuities while satisfying the constraints. 

As noted in \cite{fjordholm2016sign}, the reconstruction problem at an interface can be broken down in several cases depending on the jump ratios $\theta_i^+,\theta_{i+1}^-$. Each case has its own constraints on the perturbations $(C_1,C_2)$ such that consistency, sign-preservation, and third-order accuracy are guaranteed. We define the notion of a \textit{feasible region} based on the values of $\theta_i^+,\theta_{i+1}^-$.
\begin{defn}[Feasible Region]\label{def:feasible}
  Let $\Omega_\Theta \subset \Ro^2$ such that $(\theta_i^+,\theta_{i+1}^-) \in \Omega_\Theta$. Corresponding to $\Omega_\Theta$, we define a feasible region $\Omega_C \subset \Ro^2$ such that any $(C_1,C_2)\in\Omega_C$ satisfies
  \begin{enumerate}
      \item Consistency: $-\frac{3}{8}\leq C_1,C_2\leq\frac{1}{8}$.
      \item Sign Property: The choice of $(C_1,C_2)$ leads to weights that satisfy the constraint \eqref{eqn:sign_prop_constraint}.
      \item Accuracy: Additional order constraints (if necessary) on $(C_1,C_2)$ to ensure that third-order accuracy is achieved near smooth regions.
  \end{enumerate}
\end{defn}

Thus, for any given values of $\theta_i^+,\theta_{i+1}^-$, a feasible region is defined. See Section \ref{sec:theta_cases} in the Supplementary Material (SM) for further discussion on the cases and their associated feasible regions.

\section{\net}\label{sec:SP-WENO-DL}

The typical goal of deep learning is to approximate an unknown function
\begin{equation}
\N: \mathbb{R}^{N_I} \mapsto \mathbb{R}^{N_O}, \text{ given the dataset }
\mathbb{T} = \{(\bm{X}^{(k)},\bm{Y}^{(k)})\}_{k=1}^K.
\end{equation}
To this end, an artificial neural network $\N_{\bPhi}: \mathbb{R}^{N_I} \mapsto \mathbb{R}^{N_O}$ with trainable parameters $\boldsymbol{\Phi}$ (known as weights and biases) is trained on a suitable dataset such that an objective/loss function $\Pi(\boldsymbol{\Phi})$ is optimized. 

In the present work, the unknown function $\N$ serves to select the perturbations $C_1,C_2$ based on local features of the solution. As discussed previously, SP-WENO and SP-WENOc stipulate possible functions that $\N$ can be, though they may not be optimal. Our idea is to learn a weight selection strategy by approximating $\N$ with $\N_{\bPhi}$ so that shock-capturing behavior is improved. Using the training dataset $\mathbb{T}$ consisting of local features and the actual interface values, we hope to train a neural network such that a suitable variant of the SP-WENO reconstruction is learned.

For the $k$-th sample in $\mathbb{T}$, $\bm{X}^{(k)} \in \Ro^4$ consists of the four cell center values of some function $z$ in the local stencil centered about (some) $x_{i+\frac{1}{2}}$: $[z_{i-1},z_i,z_{i+1},z_{i+2}]$, while $\bm{Y}^{(k)} \in \Ro^2$ is the true cell interface values $z^{\pm}(x_{i+\frac{1}{2}})$. Note that the interface values from the left (-) and right (+) will be identical for a smooth function, but will be different if there is a discontinuity at the interface.

We want to ensure that all the inputs to the neural network are of the same order of magnitude (essential for generalization). Thus, we first scale the cell center values as
\begin{equation}\label{eqn:scaling}
z_{i+j}^* = \frac{z_{i+j}}{\max{(1,(\max\limits_{-1\leq r\leq 2}{(|z_{i+r}|)}))}}, \hspace{3mm} j=-1,0,1,2.
\end{equation}
A similar input scaling was also considered in \cite{ray2018artificial,ray2019detecting}. Using the scaled cell center values $[z_{i-1}^*,z_i^*,z_{i+1}^*,z_{i+2}^*]$, we compute $\theta_{i+1}^-,\theta_i^+,\left|\Delta z_{i-\frac{1}{2}}^*\right|, \left|\Delta z_{i+\frac{1}{2}}^*\right|, \left|\Delta z_{i+\frac{3}{2}}^*\right|$. We note that the jump ratios $\theta_{i+1}^-,\theta_i^+$ are invariant to the scaling \eqref{eqn:scaling}. Thus we define the map $\Ss: \Ro^4 \rightarrow \Ro^5$ which transforms $\bm{X}^{(k)}$ in the training set to the jump ratios and scaled absolute jumps mentioned above, i.e.,
\begin{equation}\label{eqn:S}
\Ss([z_{i-1},z_i,z_{i+1},z_{i+2}]) = \left[\theta_{i+1}^-,\theta_i^+,\left|\Delta z_{i-\frac{1}{2}}^*\right|, \left|\Delta z_{i+\frac{1}{2}}^*\right|, \left|\Delta z_{i+\frac{3}{2}}^*\right|\right].
\end{equation}
Since $\theta_{i+1}^-,\theta_i^+$ can assume values with significantly large magnitudes, we transform $\theta_{i+1}^-,\theta_i^+$ using the hyperbolic tangent in order to bound the input to the network. The input vector to the neural network is $\left[\tanh{(\theta_{i+1}^-)},\tanh{(\theta_i^+)},\left|\Delta z_{i-\frac{1}{2}}^*\right|, \left|\Delta z_{i+\frac{1}{2}}^*\right|, \left|\Delta z_{i+\frac{3}{2}}^*\right|\right] \in \Ro^5$.
We define the map $\Ff: \Ro^4 \rightarrow \Ro^5$ which transforms $\bm{X}^{(k)}$ in the training set to the input of the multi-layer perceptron (MLP), i.e.,
\begin{equation}
\Ff([z_{i-1},z_i,z_{i+1},z_{i+2}]) = \left[\tanh{(\theta_{i+1}^-)},\tanh{(\theta_i^+)},\left|\Delta z_{i-\frac{1}{2}}^*\right|, \left|\Delta z_{i+\frac{1}{2}}^*\right|, \left|\Delta z_{i+\frac{3}{2}}^*\right|\right]. 
\end{equation}

The output of the network $\N_\bPhi$ is the vector $\bm{\alpha} \in \Ro^5$ whose components satisfy
\begin{equation}
    \bm{\alpha}^{(k)} = \N_\bPhi ( \Ff(\bm{X}^{(k)})) \quad \text{with } \sum_{s=1}^5 \alpha^{(k)}_s = 1, \quad 0 \leq \alpha^{(k)}_s \leq 1.
\end{equation}
Let us define a function $\V:\Ro^5 \rightarrow \Ro^{2 \times 5}$ that maps $\Ss(\bm{X}^{(k)})$ to 5 vertices in $\Ro^2$, i.e.,
\begin{equation}\label{eqn:vertex_map}
    \V(\Ss(\bm{X}^{(k)})) = \bm{\nu}^{(k)} := \left[\bm{\nu}^{(k)}_1, \bm{\nu}^{(k)}_2,\bm{\nu}^{(k)}_3, \bm{\nu}^{(k)}_4, \bm{\nu}^{(k)}_5\right] \qquad \bm{\nu}^{(k)}_s \in \Ro^2 \quad 1\leq s\leq5.
\end{equation}
In \eqref{eqn:vertex_map}, the $\bm{\nu}^{(k)}_s$ corresponds to the vertices of a convex polygon (defining the feasible region), thus ensuring that any convex combination of these vertices will result in a vector in $\Ro^2$ that lies in the (closure) of this convex polygon. These convex regions are formed by at most five vertices. In situations with less than five vertices, some of the $\bm{\nu}^{(k)}_s$ in \eqref{eqn:vertex_map} are replaced by a repeated interior node of the convex region, typically the centroid. Further details about $\V$ can be found in Section \ref{sec:vertex_alg}. The output of $\N_\bPhi$ is combined with these vertices to obtain the \net weight perturbations $C_1, C_2$, given by
\begin{equation}\label{eqn:C1_C2_calculation}
\begin{bmatrix}C_1 \\ C_2 \end{bmatrix} = \V(\Ss(\bm{X}^{(k)})) \N_\bPhi ( \Ff(\bm{X}^{(k)})) = \bm{\nu}^{(k)} \bm{\alpha}^{(k)},
\end{equation}
where it is understood that $C_1, C_2$ will also be indexed by the sample index $k$.

Using \eqref{eqn:weno_recon} and \eqref{eqn:spweno_wts}, we define the reconstruction function $\mathcal{P}: \Ro^6 \rightarrow \Ro^2$, that takes as input the cell center values, $[z_{i-1},z_i,z_{i+1},z_{i+2}]$ and the WENO weight perturbations $C_1,C_2$ to give the reconstructed values $z^{\pm}_\iph$ at the interface $x_\iph$
\begin{equation}\label{eqn:net_recon}
    \widehat{\boldsymbol{Y}}^{(k)} : = [z^{-}_\iph, z^{+}_\iph] = \mathcal{P}(\bm{X}^{(k)}, [C_1, C_2]).
\end{equation}

We finally define the objective function in terms of the trainable parameters $\bPhi$ of the network
using a mean squared error (MSE) loss
\begin{align} \label{eqn:loss_func}
\Pi(\boldsymbol{\Phi}) &= \frac{1}{K}\sum_{k=1}^K \Pi_k(\bPhi), \quad \Pi_k(\bPhi) = \left\|\widehat{\boldsymbol{Y}}^{(k)} - \boldsymbol{Y}^{(k)} \right\|_2
\end{align}
where the $\widehat{\boldsymbol{Y}}^{(k)}$ encapsulate the dependence on $\boldsymbol{\Phi}$ and $\|.\|_2$ denotes the Euclidean 2-norm. Thus, the network is trained by solving the minimization problem
\[
\bPhi^\dagger  = \argmin{\bPhi} \Pi(\boldsymbol{\Phi}).
\]
Figure \ref{fig:SP-WENO-DL} summarizes how \net computes the reconstructed interface values. Note that the only learnable component of the algorithm is $\mathcal{N}_\bPhi$.

\begin{remark}
We have tried the training process with other loss functions, including the mean absolute error (MAE) loss, but ultimately using the loss function \eqref{eqn:loss_func} provides the best performance in our experiments. The choice of the loss function is a hyperparameter, so we do not claim that one loss function is superior to another. Moreover, we do not claim that we have recovered the best \net network as it might be possible to recover an even better network by further tweaking the hyperparameters.
\end{remark}

\begin{remark}
If $\Delta z_{i+\frac{1}{2}}=0$, $\theta_{i+1}^-$ and $\theta_i^+$ are undefined. In this case, we simply set $z_\iph^- = z_i, z_\iph^+ = z_{i+1}$ leading to the reconstruction jump $\llbracket z\rrbracket_{i+\frac{1}{2}}=0$. Otherwise if $\Delta z_{i+\frac{1}{2}}\neq0$, we evaluate the reconstructions using \net as described above. This is also the strategy followed for SP-WENO and SP-WENOc.
\end{remark}

\begin{figure}[!htbp]
    \centering
    \includegraphics[width=5in]{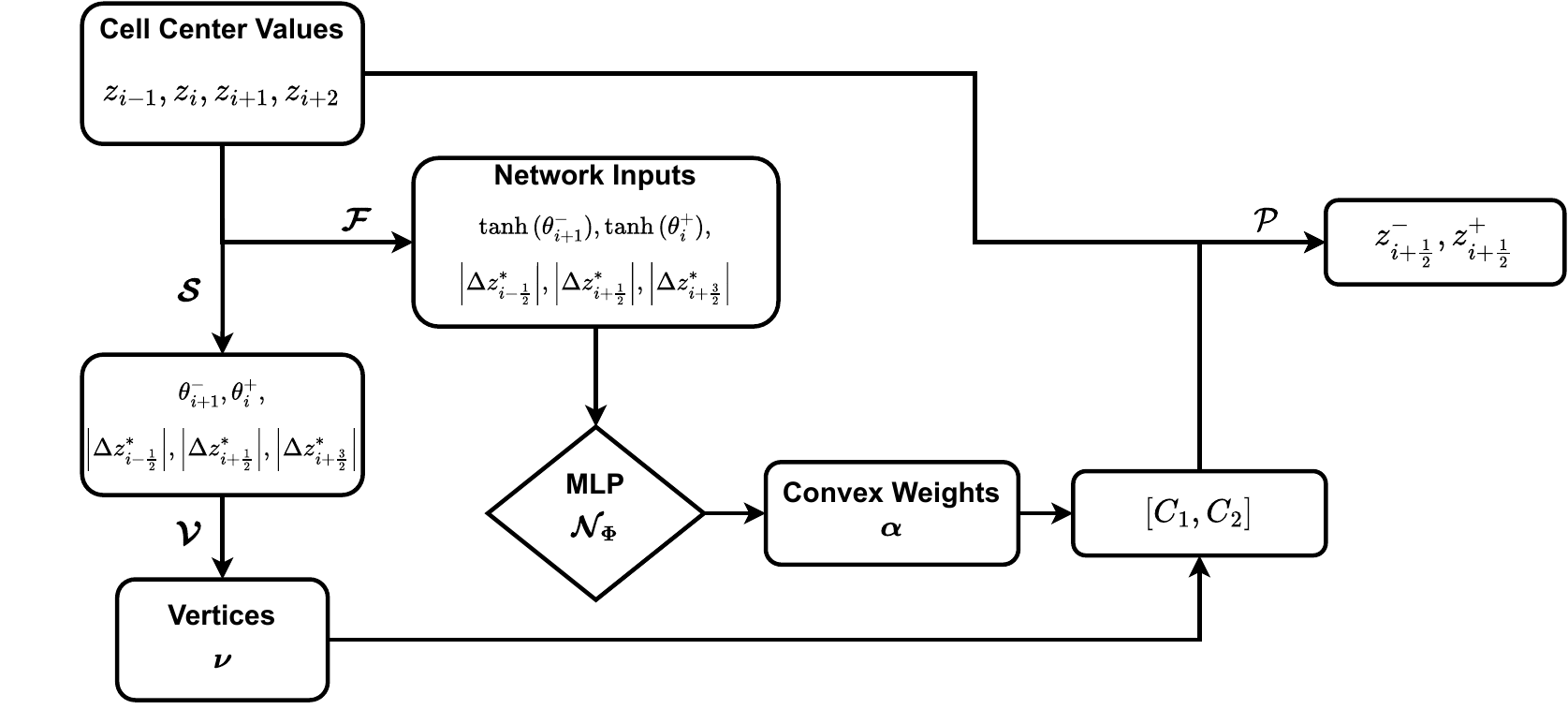}
    \caption{Schematic for the \net reconstruction algorithm.}
    \label{fig:SP-WENO-DL}
\end{figure}

\subsection{Vertex Selection Algorithm}\label{sec:vertex_alg}
In this section, we provide a brief overview of the vertex selection algorithm that defines the map $\V$ in \eqref{eqn:vertex_map}. For full details, refer to Section \ref{sec:SM_vertex} in the SM. Based on the values of $\theta_{i}^+,\theta_{i+1}^-$, we consider 6 disjoint subdomains, each denoted as $\Omega_\Theta$, which partition $\Ro^2$. For each $\Omega_\Theta$, we construct a feasible region $\Omega_C$ (see Definition \ref{def:feasible}). Recall that for any choice of $(C_1,C_2)$ in $\Omega_C$ leads to consistent WENO weights, a guarantee of the sign property and third-order accuracy in smooth regions. Further, it is possible to choose $\Omega_C$ (for each case) to form convex polygon in $\mathbb{R}^2$ with at most 5 vertices. In fact, given any set of cell center values $[z_{i-1},z_i,z_{i+1},z_{i+2}]$, we provide an explicit algorithm to obtain the 5 vertices (taking into account repeated interior nodes when needed) of the convex $\Omega_C$. The pseudocode for the algorithm can be found in Algorithm \ref{alg:vertex_selection} in the SM. The objective with the neural network approach is to learn how to explore these $\Omega_C$ to better select $(C_1,C_2)$. In comparison, SP-WENO and SP-WENOc have handcrafted weight selection strategies that suffer from significant overshoots near discontinuities in the TeCNO framework. \net aspires to mitigate this behavior through a data-driven search in the feasible regions.

\subsection{Data Selection}\label{sec:data_selection}
To generate training data for the neural network, we sample stencils consisting of four cells from parameterized smooth and discontinuous functions of the form listed in Table \ref{tab:sampled_functions} sampled on meshes with mesh size $h=1/40, 1/100$, or $1/200$. The dataset has a 50/50 balanced split of smooth and discontinuous data. The section of the dataset consisting of smooth data consists of an equal representation of the three smooth types listed in Table \ref{tab:sampled_functions}. For discontinuous data, the jump will either be between the first and second cells, the second and third cells, or the third and fourth cells. In any case, using the four cell center values, we compute the jump ratios $\theta_{i+1}^-$ and $\theta_i^+$ and the absolute jumps across the three cell interfaces in the stencil, via $\Ss$ defined in \eqref{eqn:S}. Further, we compute a set of vertices using $\V$ in Section \ref{sec:vertex_alg}. Finally, the true values at the cell interface $x_{i+\frac{1}{2}}$ are recorded to serve as target values that \net will aim to reconstruct. Overall, the dataset consists of 100,000 samples, roughly corresponding to 50,000 discontinuous samples and 50,000 smooth samples.

\begin{table}
\centering
\begin{tabular}{ |c|c|c|c| }
\hline
No. & Type & $u(x)$ & Parameters \\
\hline
1 & Smooth & $ax^3+bx^2+cx+d$ & $a,b,c,d\in\mathbb{U}[-10,10]$ \\
2 & Smooth & $(x-a)(x-b)(x-c)+d$ & $a,b,c,d\in\mathbb{U}[-2,2]$ \\
3 & Smooth & $\sin{(a\pi x + b)}$ & $a,b\in\mathbb{U}[-2,2]$  \\
4 & Discontinuous & $\begin{cases}
ax + b \text{ if } x\leq0.5\\
cx + d \text{ if } x>0.5
\end{cases}$ & $a,b,c,d\in\mathbb{U}[-5,5]$ \\
\hline
\end{tabular}
\caption{Functions used in constructing the training dataset where $\mathbb{U}[p,q]$ denotes the uniform distribution on the interval $[p,q].$ 
}
\label{tab:sampled_functions}
\end{table}

\begin{remark}
 We do not use solutions to conservation laws to train the network. We instead generate functions of varying regularity that canonically represent the local solution features we can expect to observe in solutions to conservation laws. Thus, there is an insignificant cost in generating the training data. A similar strategy was also considered in \cite{ray2018artificial,ray2019detecting}.
\end{remark}

\subsection{Network Architecture and Training}

Given the small input and output dimensions ($N_I = N_O = 5$) for the network, we use a fairly simple network architecture. In particular, we employ an MLP that consists of three hidden layers with five neurons each, leading to an MLP with $\# \bPhi = 120$ trainable parameters. The ReLU activation function is used in the hidden layers. Following the output layer, we use the softmax output function to ensure that the network predicts a vector of convex vertex weights $\bm{\alpha}$. The predicted $\bm{\alpha}$ are used to obtain the perturbation according to \eqref{eqn:C1_C2_calculation}, which are then used to reconstruct the (left and right) cell interface values according to \eqref{eqn:net_recon}.

To train the neural network, we use the Adam optimizer \cite{kingma2014adam} with optimizer parameters $\beta_1=0.5$, $\beta_2=0.9$, a learning rate of $10^{-3}$, and a weight decay regularization of $10^{-5}$. We perform a training/validation/test split in a 0.6/0.2/0.2 proportion where we seek to minimize the reconstruction error \eqref{eqn:loss_func}. We use mini-batches of size 500 and reshuffle the training set every epoch. The network is trained using PyTorch for 50 epochs on a 2.60 GHz Intel Core i7-10750H CPU. The training time for training one instance of this MLP is less than 30 seconds. We perform five training runs with different random initializations of the network weight and biases and select the network that performs the best on the test set. We remark that the true test of the network is based on its performance within the TeCNO schemes to solve conservation laws (see Section \ref{sec:numerical}).

\begin{remark}
    Other variations of the MLP's architecture (including activation functions) were considered, but ultimately this architecture yielded the best results.
\end{remark}

\subsection{Properties of \net}

\net by construction satisfies all properties of SP-WENO with the exception of the mirror property. This property is lost due to the nature of the vertex selection algorithm and computation of the convex weights by the neural network. While it is possible to have pursued a construction that preserves the mirror property, we choose to omit it so as not to overly constrain the network. Additionally, \net satisfies the following stability estimate for the reconstructed jump:
\begin{equation}\label{eqn:stab}
\left|\llbracket z\rrbracket_{i+\frac{1}{2}}\right| \leq \frac{1}{2}\left|\Delta z_{i-\frac{1}{2}}\right| + \left|\Delta z_{i+\frac{1}{2}}\right| + \frac{1}{2}\left|\Delta z_{i+\frac{3}{2}}\right| \hspace{3mm} \forall\hspace{1mm} i\in\mathbb{Z}
\end{equation}
which provides an explicit bound on the size of the reconstructed jump based on the size of
the original jumps of the point values in the local stencil. Similar upper bounds have been obtained for ENO reconstruction \cite{fjordholm2013eno}, SP-WENO \cite{fjordholm2016sign}, and SP-WENOc \cite{ray2018third}.
A proof of the bound \eqref{eqn:stab} can be found in Section \ref{sec:sm_stability}.

\begin{remark}\label{rem:pos_pre}
    We note that the \net reconstruction (as with any SP-WENO variation) is not positivity-preserving for the system of Euler equations, as there is no constraint present that enforces this. In the present work, we have demonstrated that it is possible to learn reconstruction methods by building in constraints of interest, such as the sign property. However, the same ideas can be used to also build in other desirable properties, such as positivity preservation for models such as the Euler system. While is was not the focus of the present work, these avenues will be explored in future work.
\end{remark}

\section{Numerical Results}\label{sec:numerical}
In this section, we present several numerical results to demonstrate the efficacy of the proposed \net approach. We consider both 1D and 2D test cases for the evolution of scalar and systems of conservation laws. 
In 2D, we discretize the domain $[a,b] \times [c,d]$ using a uniform mesh (in each dimension) with 
$h_x=(b-a)/N_x$ and $h_y=(d-c)/N_y$. 
Ghost cells are introduced to extend the mesh as needed while imposing either periodic or Neumann boundary conditions. In all evolution cases, we use the TeCNO4 numerical flux, consisting of the fourth-order accurate entropy conservative flux \eqref{4th_order_flux} and various sign-preserving reconstruction methods for the diffusion operator. SSP-RK3 (see \cite{gottlieb2001strong}) is used for time-marching with the CFL varying across problems. The code used to train \net and the C++ code to run the TeCNO schemes can be found at the repository: \url{https://github.com/pmcharles/DSP-WENO}

\subsection{Reconstruction Accuracy}
Before proceeding to solve conservation laws, we demonstrate the reconstruction accuracy of \net. We consider the reconstruction of a smooth inclined sine wave $u(x) = \sin{(10\pi x)} + x$ on $[0,1]$ as the mesh size $h$ is varied. We reconstruct the values at the cell interfaces, and evaluate the (averaged) error at the interfaces (from the left and right) as
\begin{equation*}
\mathcal{E}_h = \frac{1}{N}\sum_{i=1}^{N} |u^-_{i+\frac{1}{2}}-u(x_{i+\frac{1}{2}})| + \frac{1}{N}\sum_{i=1}^{N} |u^+_{i-\frac{1}{2}}-u(x_{i-\frac{1}{2}})|.
\end{equation*}

\begin{table}[]
    \centering
    \begin{tabular}{|c|c c|c c|c c|c c|}
     \hline
     N & \multicolumn{2}{c|}{ENO3} & \multicolumn{2}{c|}{SP-WENO} & \multicolumn{2}{c|}{SP-WENOc} & \multicolumn{2}{c|}{\net} \\
         & $\mathcal{E}_h$ & Rate & $\mathcal{E}_h$ & Rate & $\mathcal{E}_h$ & Rate & $\mathcal{E}_h$ & Rate \\
     \hline
     40  & 3.47e-2 & -    & 7.27e-2 & -    & 7.41e-2 & -    & 1.65e-1 & - \\
     80  & 4.54e-3 & 2.93 & 5.85e-3 & 3.64 & 6.37e-3 & 3.54 & 3.01e-2 & 2.45 \\
     160  & 5.84e-4 & 2.96 & 4.45e-4 & 3.72 & 4.71e-4 & 3.76 & 2.83e-3 & 3.41 \\
     320  & 7.42e-5 & 2.98 & 3.29e-5 & 3.76 & 3.43e-5 & 3.78 & 2.14e-4 & 3.73 \\
     640  & 9.38e-6 & 2.98 & 2.37e-6 & 3.79 & 2.46e-6 & 3.80 & 1.55e-5 & 3.78 \\
     1280 & 1.17e-6 & 3.00 & 1.68e-7 & 3.82 & 1.74e-7 & 3.82 & 1.22e-6 & 3.67 \\
     \hline
    \end{tabular}
    \caption{Reconstruction accuracy for inclined sine wave.}
    \label{tab:recon}
\end{table}

Table \ref{tab:recon} shows the errors and corresponding convergence rates for ENO3, SP-WENO, SP-WENOc, and \net. All reconstruction methods achieve the expected third-order convergence with the SP-WENO variants exhibiting super-convergence. Note that the errors of \net are about an order of magnitude larger than those of SP-WENO and SP-WENOc while being similar to ENO3 on the finer meshes. This leads to a larger reconstruction jump at interfaces where a discontinuity is present (see Section \ref{sec:rec_jump} in the SM for additional details), which plays a critical role in ensuring that the dissipation term in the TeCNO scheme injects sufficient viscosity near shocks to suppress spurious oscillations. We demonstrate this in the experiments that follow.


\subsection{Scalar conservation laws} 
For scalar conservation laws, we choose the squared entropy $\eta(u)=u^2/2$ with the corresponding entropy variable $v=\eta'(u)=u$. Also, since the scaled entropy variable are the same as the entropy variable, the finite difference TeCNO4 numerical flux \eqref{high_order_ent_stable_flux} is given by
\begin{equation}\label{eqn:scalar_tecno}
f_{i+\frac{1}{2}} = \tilde{f}^{4}_\iph - a_{i+\frac{1}{2}}\llbracket v\rrbracket_\iph, \quad a_\iph = \frac{|f^\prime(u_i)| + |f^\prime(u_{i+1})|}{2}.
\end{equation}

\subsubsection{Linear Advection}\label{sec:linadv}
Taking the flux to be $f(u) = cu$, the second-order entropy conservative flux used to construct $\Tilde{f}^4$ in \eqref{4th_order_flux} using $
\tilde{f}_\iph = c(u_i+u_{i+1})/2$. For all experiments, we set the convective velocity $c=1$.\\

\textbf{Smooth Initial Data:} We consider two test cases with smooth initial conditions on $[-\pi,\pi]$ and periodic boundary conditions
\begin{itemize} 
\item \textbf{Test 1:} $u_0(x) = \sin(x)$, simulated until final time $T=0.5$ with CFL = 0.4.
\item \textbf{Test 2:} $u_0(x) = \sin^4(x)$, simulated until final time $T=0.5$ with CFL = 0.5.
\end{itemize}

Table \ref{tab:lin_adv} shows the errors (measured in the discrete $L^1$ norm) with various reconstructions for both of these test cases (see Section \ref{sec:sm_LinAdv} for the $L^2$ and $L^{\infty}$ errors). ENO3 works well for Test 1 and achieves third-order convergence with marginally smaller absolute errors when compared to the results of SP-WENO, SP-WENOc, and \net. However, the convergence rate with ENO3 deteriorates in Test 2. It has been observed in \cite{rogerson1990numerical} that ENO3 performs poorly for this test case with the MUSCL scheme due to linear instabilities, and we observe the same issues with ENO3 in the TeCNO framework (also seen in \cite{fjordholm2016sign}). The other reconstruction methods are stable, as SP-WENO, SP-WENOc, and \net do not suffer from this issue. All three reconstruction methods achieve third-order convergence, though the errors of \net are about an order of magnitude larger than those of SP-WENO and SP-WENOc. Moreover, SP-WENO and SP-WENOc exhibit super-convergence. This can be attributed to the reconstructed jump often being zero, which effectively disables the local diffusion and results in the local numerical flux entirely consisting of the fourth-order entropy conservative flux.

\net is much more dissipative, so in these test cases and in general for smooth regions, we observe larger errors (though the third-order convergence is still maintained). However, this increase in diffusion leads to improved performance over SP-WENO and SP-WENOc when the solution is not smooth, as is demonstrated in the next test case.\\

\begin{table}[]
    \centering
    \begin{tabular}{|c|c|c c|c c|c c|c c|}
     \hline
     &N & \multicolumn{2}{c|}{ENO3} & \multicolumn{2}{c|}{SP-WENO} & \multicolumn{2}{c|}{SP-WENOc} & \multicolumn{2}{c|}{\net} \\
          & & Error & Rate & Error & Rate & Error & Rate & Error & Rate \\
     \hline
     \parbox[t]{2mm}{\multirow{6}{*}{\rotatebox[origin=c]{90}{Test 1}}} &100  & 3.23e-5 & -    & 6.88e-5 & -    & 6.76e-5 & -    & 9.09e-5 & - \\
     &200  & 4.04e-6 & 3.00 & 7.60e-6 & 3.18 & 7.46e-6 & 3.18 & 8.70e-6 & 3.39 \\
     &400  & 5.05e-7 & 3.00 & 8.27e-7 & 3.20 & 8.17e-7 & 3.19 & 9.11e-7 & 3.25 \\
     &600  & 1.50e-7 & 3.00 & 2.26e-7 & 3.20 & 2.27e-7 & 3.16 & 2.52e-7 & 3.17 \\
     &800  & 6.31e-8 & 3.00 & 8.73e-8 & 3.30 & 8.71e-8 & 3.32 & 1.01e-7 & 3.17 \\
     &1000 & 3.23e-8 & 3.00 & 4.22e-8 & 3.26 & 4.23e-8 & 3.24 & 5.02e-8 & 3.15 \\
     \hline
     \parbox[t]{2mm}{\multirow{6}{*}{\rotatebox[origin=c]{90}{Test 2}}} &100  & 1.48e-3 & -    & 1.47e-3 & -    & 1.45e-3 & -    & 2.07e-3 & - \\
     &200  & 1.98e-4 & 2.91 & 1.62e-4 & 3.17 & 1.60e-4 & 3.18 & 2.30e-4 & 3.18 \\
     &400  & 2.58e-5 & 2.94 & 1.75e-5 & 3.21 & 1.74e-5 & 3.20 & 2.69e-5 & 3.09 \\
     &600  & 8.25e-6 & 2.81 & 4.71e-6 & 3.24 & 4.71e-6 & 3.22 & 7.73e-6 & 3.07 \\
     &800  & 4.65e-6 & 1.99 & 1.84e-6 & 3.27 & 1.82e-6 & 3.30 & 3.27e-6 & 2.99 \\
     &1000 & 3.47e-6 & 1.32 & 8.94e-7 & 3.23 & 8.90e-7 & 3.22 & 1.66e-6 & 3.04 \\
     \hline
    \end{tabular}
    \caption{$L^1$ errors and convergence rates for linear advection smooth tests 1 and 2.}
    \label{tab:lin_adv}
\end{table}

\textbf{Moving Shapes:} The domain is $[0,1.4]$, final time is $T=1.4$, and CFL = 0.2, with initial profile given by
\begin{flalign*}
    u_0(x) =& 10(x-0.2) \chf{(0.2,0.3]}(x) + 10(0.4 - x) \chf{(0.3,0.4]}(x) \\
    &+ \chf{(0.6,0.8]}(x) + 100 (x-1)(1.2-x) \chf{(1,1.2]}(x)
\end{flalign*}
where $\chf{\mathcal{A}}(x)$ denotes the characteristic function on set $\mathcal{A}$. Note that the initial profile is a composition of shapes with different orders of regularity. The results on a mesh with 100 cells with peridoc boundary conditions for various reconstruction methods are shown in Figure \ref{fig:lin_adv3}.
Overall, ENO3 yields the best performance, as its solution is the closest to the exact solution while avoiding any spurious oscillatory behavior. The solution obtained with \net appears to be marginally more dissipative than ENO3, but without exhibiting the large under/overshoots observed with SP-WENO or SP-WENOc.

\begin{figure}
    \centering
    \includegraphics[width=3.8in]{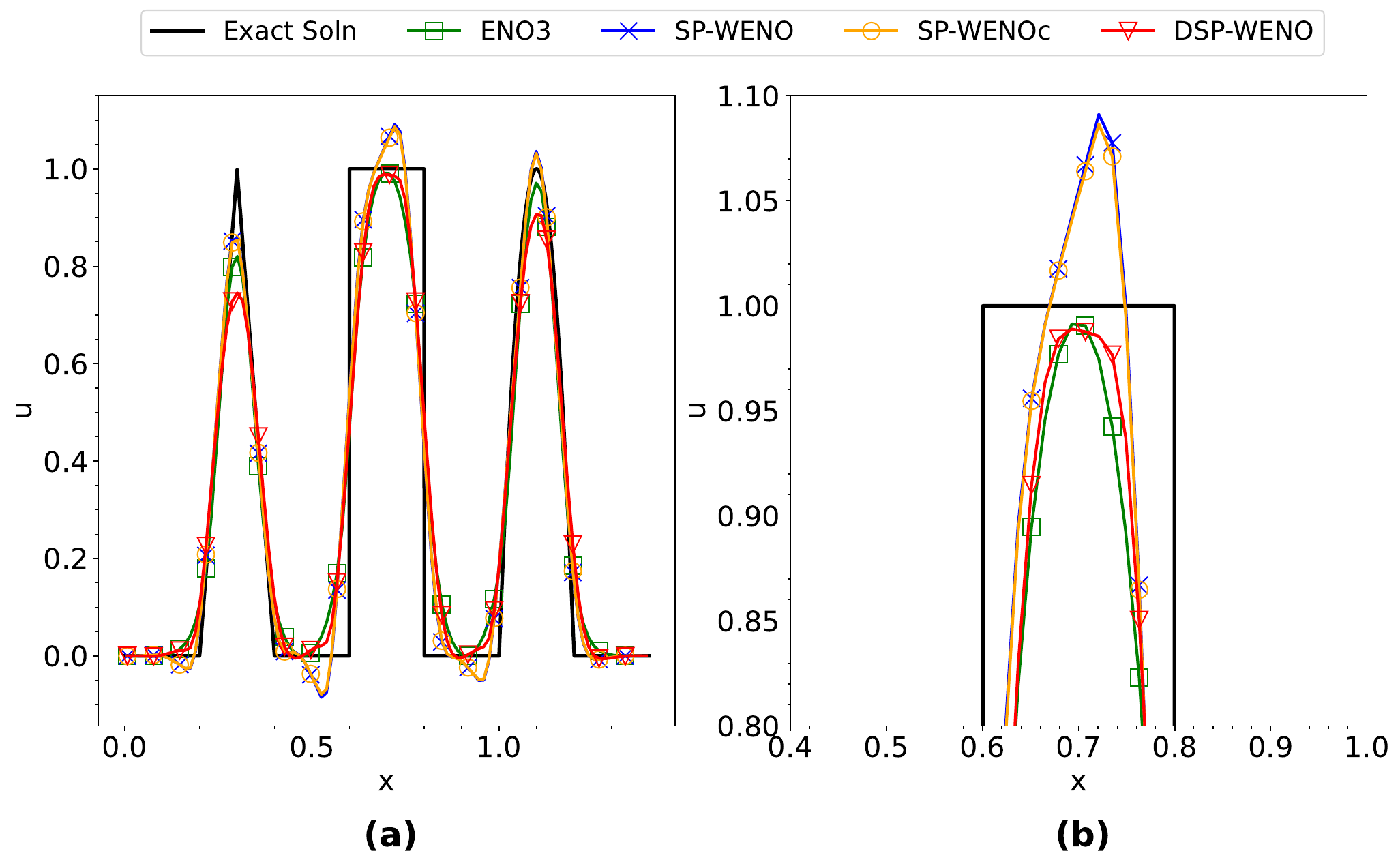}
    \caption{Linear advection of shapes: Solution at time $T=1.4$. Comparison of different reconstruction methods: \textbf{(a)} $x\in[0,1.4]$, \textbf{(b)} $x\in[0.4,1]$.}
    \label{fig:lin_adv3}
\end{figure}

\subsubsection{A Note on the Computational Cost}
In general, \net is a more expensive reconstruction than its SP-WENO and SP-WENOc counterparts. As a first step, we estimate the computational times to reconstruct on 100,000 random samples of local stencil values with the various reconstruction algorithms. These times were 0.852s (\net), 0.178s (SP-WENO), 0.192s (SP-WENOc), and 0.839s (ENO3). It is clear that \net can be up to approximately 5 times slower than SP-WENO/SP-WENOc, while being comparable to ENO3. Furthermore, we observed that the vertex search algorithm is the most expensive part of \net (95\% of the total cost).

Next, we assess the computational runtimes when the reconstructions are used within a TecNO scheme. Figure \ref{fig:comp_cost} shows the runtimes plotted against the $L^1$ errors for the reconstruction methods for both smooth linear advection tests of Section \ref{sec:linadv}. Note that the markers here correspond to the various mesh resolutions considered in Table \ref{tab:lin_adv}. We observe that \net can be be two to three times as expensive than the other methods at finer resolutions. This increased cost, even when compared to ENO3, can be potentially explained by the fact that the current version of the code loops over the cell interfaces in the mesh to compute the reconstructed jump required by the dissipation operator. Thus, the network for \net is queried one input at a time, which is a very inefficient way of using a network. We expect that a batched evaluation of the network on the mesh, along with a more efficient (vectorized) implementation of the vertex algorithm, can significantly lower the overall cost of using \net. However, this is not the primary focus of the present work.

\begin{figure}
    \centering
    \includegraphics[width=3.8in]{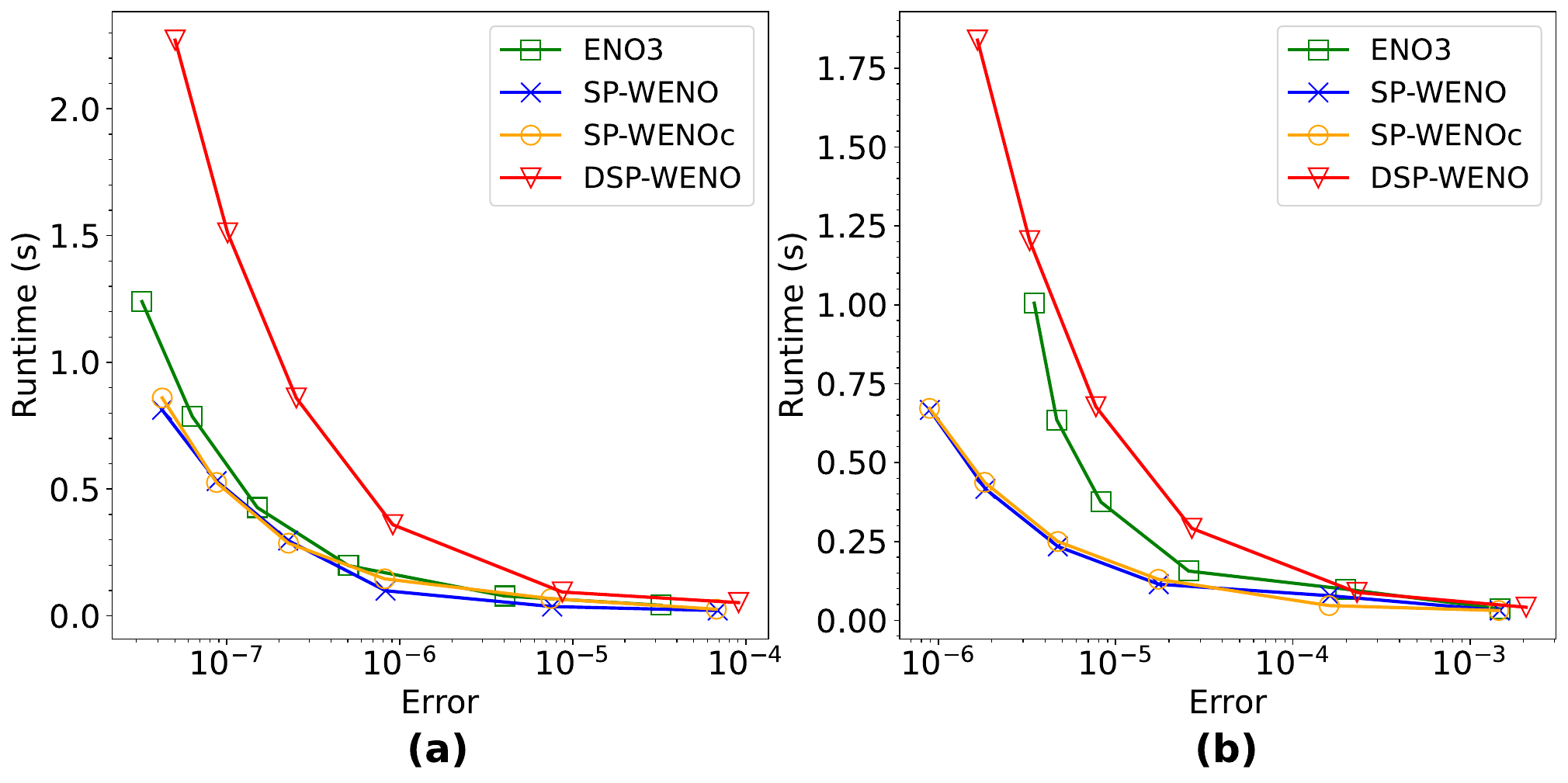}
    \caption{Runtime vs. $L^1$ error for linear advection: \textbf{(a)} Test 1, \textbf{(b)} Test 2.}
    \label{fig:comp_cost}
\end{figure}

\subsubsection{Burgers' Equation}\label{sec:burgers}

We now consider the Burgers' Equation, i.e., $f(u) = u^2/2$. In this case, the second-order entropy conservative flux corresponding to the quadratic entropy can be expressed as $\tilde{f}_\iph = (u_i^2+u_{i+1}^2 + u_iu_{i+1})/6$. \\

\textbf{Test 1:} The domain is $[-1,1]$, final time is $T=0.5$, and CFL = 0.4, with initial condition given by $u_0(x) = 3 \chf{\{x<0\}}(x) - \chf{\{x\geq0\}}(x)$.
The mesh consists of 100 cells with Neumann boundary conditions. Figure \ref{fig:burgers1} shows that all the solutions experience oscillations. This is a fundamental problem with high-order TeCNO schemes, so there is often no way to completely eliminate the overshooting behavior without introducing additional diffusion. However, we observe that \net performs better in mitigating the oscillations leading up to the overshoot.

\begin{remark} 
While omitted from Figure \ref{fig:burgers1}, we found that these oscillations cannot be avoided even with the most restrictive linear reconstruction using the minmod limiter, which also satisfies the sign property \cite{fjordholm2012arbitrarily}.
\end{remark}

\begin{figure}
    \centering
    \includegraphics[width=3.8in]{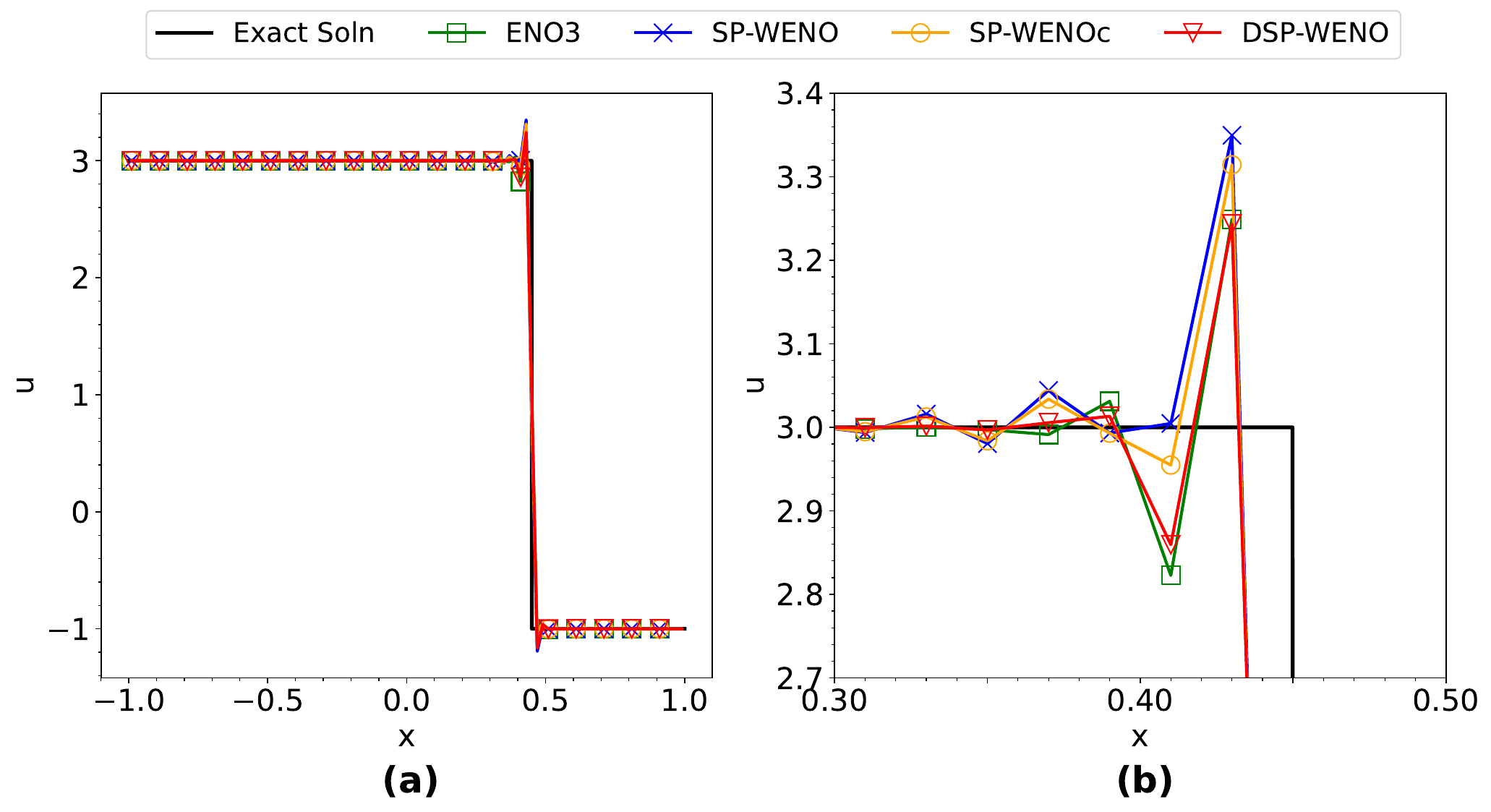}
    \caption{Burgers' Test 1: Solution at time $T=0.5$. Comparison of different reconstruction methods: \textbf{(a)} $x\in[-1,1]$, \textbf{(b)} $x\in[0.3,0.5]$.}
    \label{fig:burgers1}
\end{figure}

\textbf{Test 2:} The domain is $[-4,4]$, final time is $T=0.4$, and CFL = 0.4, with initial profile
\[
u_0(x) = 3 \chf{[-1,-0.5)}(x) + \chf{[-0.5,0)}(x) + 3\chf{[0,0.5)}(x)+ 2\chf{[0.5,1)}(x) + \sin(\pi x) \chf{\{|x|>1\}}(x)
\]
and periodic boundary conditions. This initial condition is taken from \cite{ray2018artificial} and contains both smooth and discontinuous features. When solved on a mesh with 400 cells, Figure \ref{fig:burgers2} shows that SP-WENO and SP-WENOc perform poorly near the shocks, especially at $x=1.3$ where the overshoots are relatively large. While the oscillations are still present with \net, they are significantly mitigated when compared to the other SP-WENO methods.

\begin{figure}
    \centering
    \includegraphics[width=3.8in]{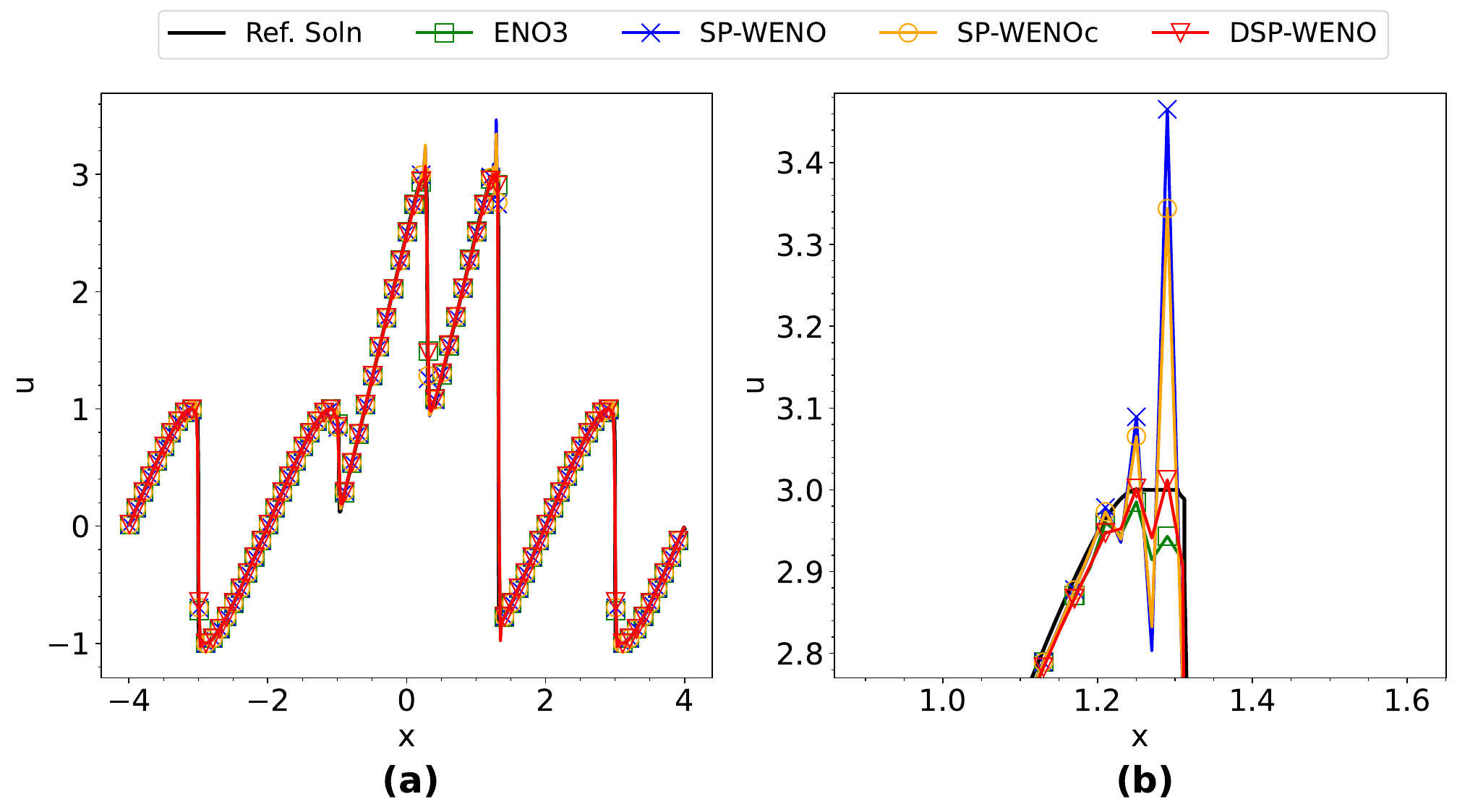}
    \caption{Burgers' Test 2: Solution at time $T=0.4$. Comparison of different reconstruction methods: \textbf{(a)} $x\in[-4,4]$, \textbf{(b)} $x\in[0.86,1.65]$.}
    \label{fig:burgers2}
\end{figure}

\subsection{Euler Equations}\label{sec:euler}

We now consider the Euler equations described by

\begin{flalign*}
\frac{\partial}{\partial t}\begin{pmatrix}\rho\\\rho \velvec\\E\end{pmatrix} + \nabla\cdot\begin{pmatrix}\rho\velvec\\p\mathbf{I} + \rho(\velvec\otimes\velvec)\\(E+p)\velvec\end{pmatrix} = \mathbf{0},
\end{flalign*}
where $\rho$, $\velvec = (\vel_1,\vel_2,\vel_3)^\top$, and $p$ represent the fluid density, velocity, and pressure, respectively. Note that $\velvec$ is distinct from $\con$, with the latter representing the vector of conserved variables. The total energy per unit volume given by 
$ E = \rho |\velvec|^2/2 + p/(\gamma-1)$
where $\gamma = c_p/c_v$ denotes the ratio of specific heats. In all test cases, we set $\gamma=1.4$. 

For the Euler equations, a popular choice for the entropy-entropy flux pair \cite{harten83} is $\eta = -\rho s/(\gamma-1)$, $\bm{q} = [q_1,q_2, q_3] = \eta \velvec^\top$
where $s=\ln{(p)} - \gamma\ln{(\rho)}$ is associated with the thermodynamic specific entropy. The corresponding vector of entropy variables is given by $\ent = [-\beta|\velvec|^2 - (\gamma-s)/(\gamma-1), \ 2\beta\velvec^{\top}, \ -2\beta]^\top$
where $\beta = \rho/(2p)$.

To construct the fourth-order entropy conservative flux \eqref{4th_order_flux} we use the second-order kinetic energy preserving and entropy conservative (KEPEC) flux \cite{chandrashekar2013kinetic}, which in 1D is expressed as:
\begin{equation}
\f = \begin{pmatrix}
F^{\rho}\\F^m\\F^e
\end{pmatrix} = \begin{pmatrix}
\hat{\rho}\overline{\vel}\\\Tilde{p} + \overline{\vel}F^{\rho}\\F^e
\end{pmatrix}, \hspace{3mm} F^e=\left[\frac{1}{2(\gamma-1)\hat{\beta}}-\frac{1}{2}\overline{|\vel|^2}\right]F^{\rho} + \overline{\vel}F^m,
\end{equation}
where $\Tilde{p}=\overline{\rho}/(2\overline{\beta})$ and $\beta=\rho/(2p)$. $\hat{\rho}$, $\hat{\beta}$ represent the logarithmic averages (see \cite{ismail2009,chandrashekar2013kinetic}) of the respective positive quantities.

For the diffusion term in \eqref{high_order_ent_stable_flux}, we use the Roe-type diffusion operator described in Section \ref{sec:fin_dif_schemes} with $\mathbf{\Lambda}_{i+\frac{1}{2}} = \diag{\left(|u-a|,|u|,|u+a|\right)}_\iph$
where $a = \sqrt{\gamma p/\rho}$ is the speed of sound in air. The various terms in the matrices of the diffusion operator are evaluated at some averaged state at the cell interface. For additional details on the diffusion operator, and the flux formulation in higher dimensions, we refer the readers to \cite{chandrashekar2013kinetic,ray2018third}.\\



\textbf{1D Modified Sod Shock Tube:} This is a shock tube problem \cite{toro2013riemann} solved on the domain is $[0,1]$ with initial profile $(\rho,\vel,p) = (1,0.75,1) \chf{\{x<0.3\}}(x) + (0.125,0,0.1) \chf{\{x\geq0.3\}}(x)$.
It is solved on a domain with 400 cells with Neumann boundary conditions, until a final time of $T=0.2$ with CFL=0.4. Figure \ref{fig:modsod} shows that SP-WENO and SP-WENOc exhibit significant overshoots, especially near the shock at $x\approx 0.72$. \net markedly improves on the performance of both methods by featuring only very minor overshoots in comparison. Further, the accuracy in shock-capturing is comparable to ENO3, with a sharper resolution of the contact discontinuity at $x\approx0.57$.\\

\begin{figure}
    \centering
    \includegraphics[width=3.8in]{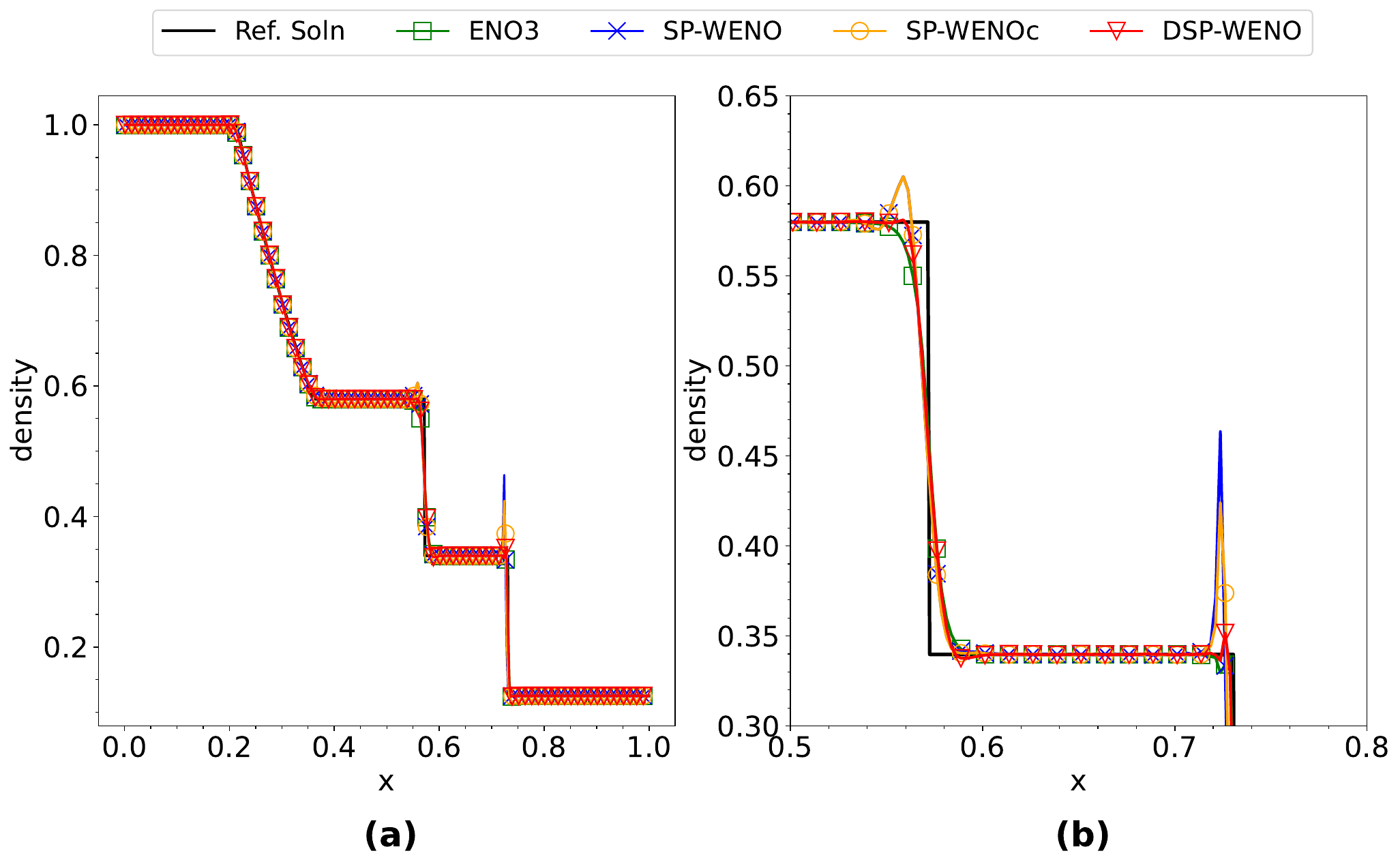}
    \caption{Modified Sod Shock Tube Test. Comparison of density profile at final time $T=0.2$ for different reconstruction methods: \textbf{(a)} $x\in[0,1]$, \textbf{(b)} $x\in[0.5,0.8]$.}
    \label{fig:modsod}
\end{figure}

\textbf{1D Shu-Osher Test:} This test case, proposed in \cite{shu1989efficient}, contains the interaction of an oscillatory smooth wave and a right-moving shock. The domain is $[-5,5]$, final time is $T=1.8$, and CFL = 0.4, with initial profile
\begin{flalign*}
(\rho,\vel,p) =& 
(3.857143, \ 2.629369, \ 10.33333)\chf{\{x<-4\}}(x) \\
&+ 
(1+0.2\sin{(5x)}, \ 0,\ 1) \chf{\{x\geq-4\}}(x),
\end{flalign*}
and Neumann boundary conditions. We solve on a mesh consisting of 400 cells, which is necessary to resolve the high frequency physical oscillations in the solution. This was one of the test cases presented in \cite{ray2018third} where SP-WENOc significantly mitigates the overshoots near the shock as compared to SP-WENO, which is also what we observe in Figure \ref{fig:shuosher}. \net is clearly the most dissipative of the methods (although comparable to ENO3) when focusing in the regions with smooth physical high-frequency oscillations. However, there is essentially no overshoot near the shock with \net, which is a large improvement over SP-WENO and SP-WENOc. \\

\begin{figure}
    \centering
    \includegraphics[width=3.8in]{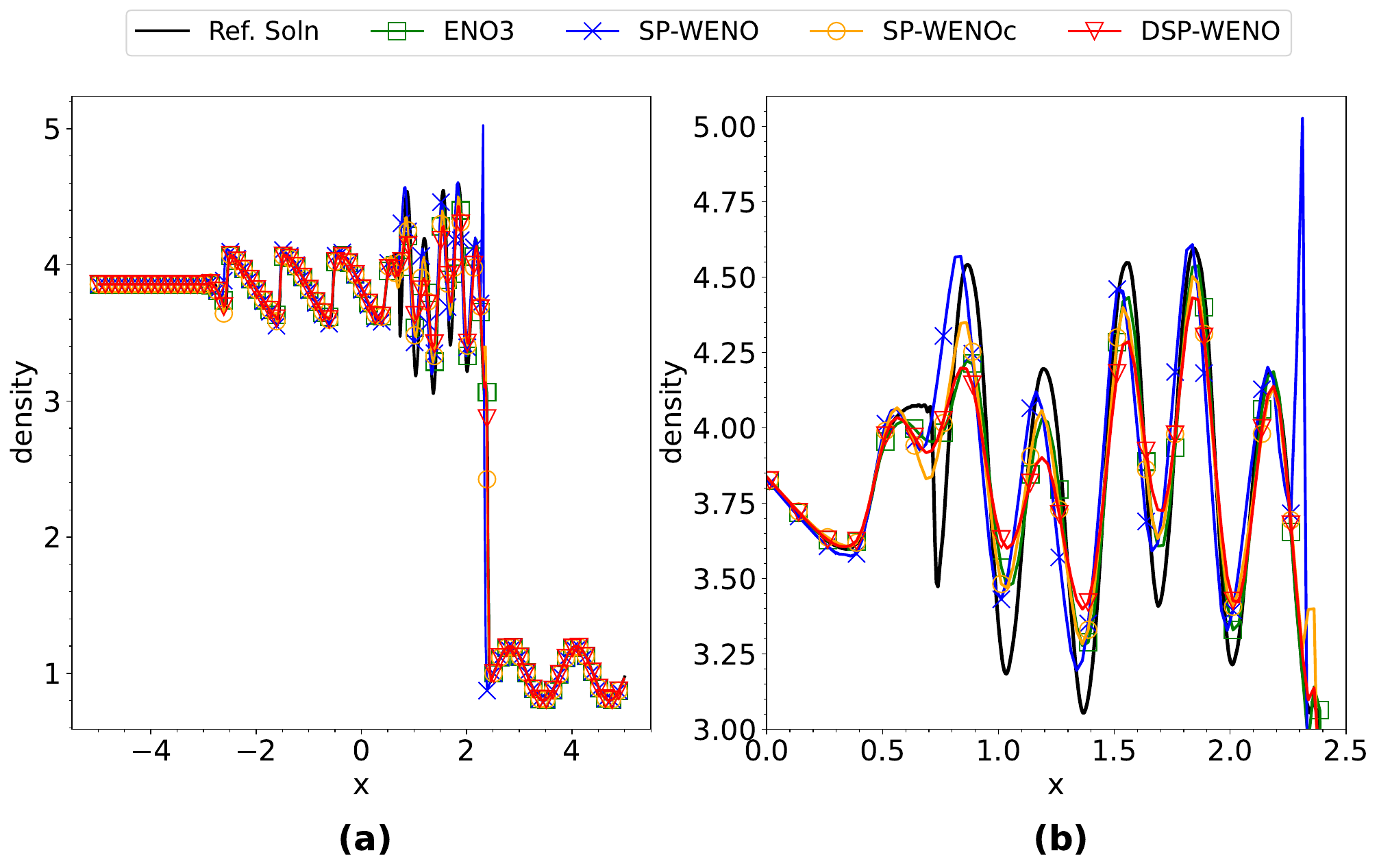}
    \caption{Shu-Osher Test: Comparison of density profile at final time $T=1.8$ for different reconstruction methods: \textbf{(a)} $x\in[-5,5]$, \textbf{(b)} $x\in[0,2.5]$.}
    \label{fig:shuosher}
\end{figure}

\textbf{1D Lax Test:} The Lax shock tube problem \cite{lax1954weak} is described by the initial profile 
\begin{flalign*}
    (\rho,\vel,p) = 
(0.445,\ 0.698,\ 3.528)\chf{\{x<0\}}(x) + 
(0.5, \ 0, \ 0.571) \chf{\{x\geq0\}}(x)
\end{flalign*}
on the domain $[-5,5]$ with Neumann boundary conditions. The problem is solved on a mesh with 200 cells with CFL=0.4 until a final time $T=1.3$. Figure \ref{fig:lax} shows that SP-WENO and SP-WENOc exhibit significant overshoots near the shocks. The solution obtained with \net greatly minimizes the oscillatory behavior, similar to the previous test cases. \\

\begin{figure}
    \centering
    \includegraphics[width=3.8in]{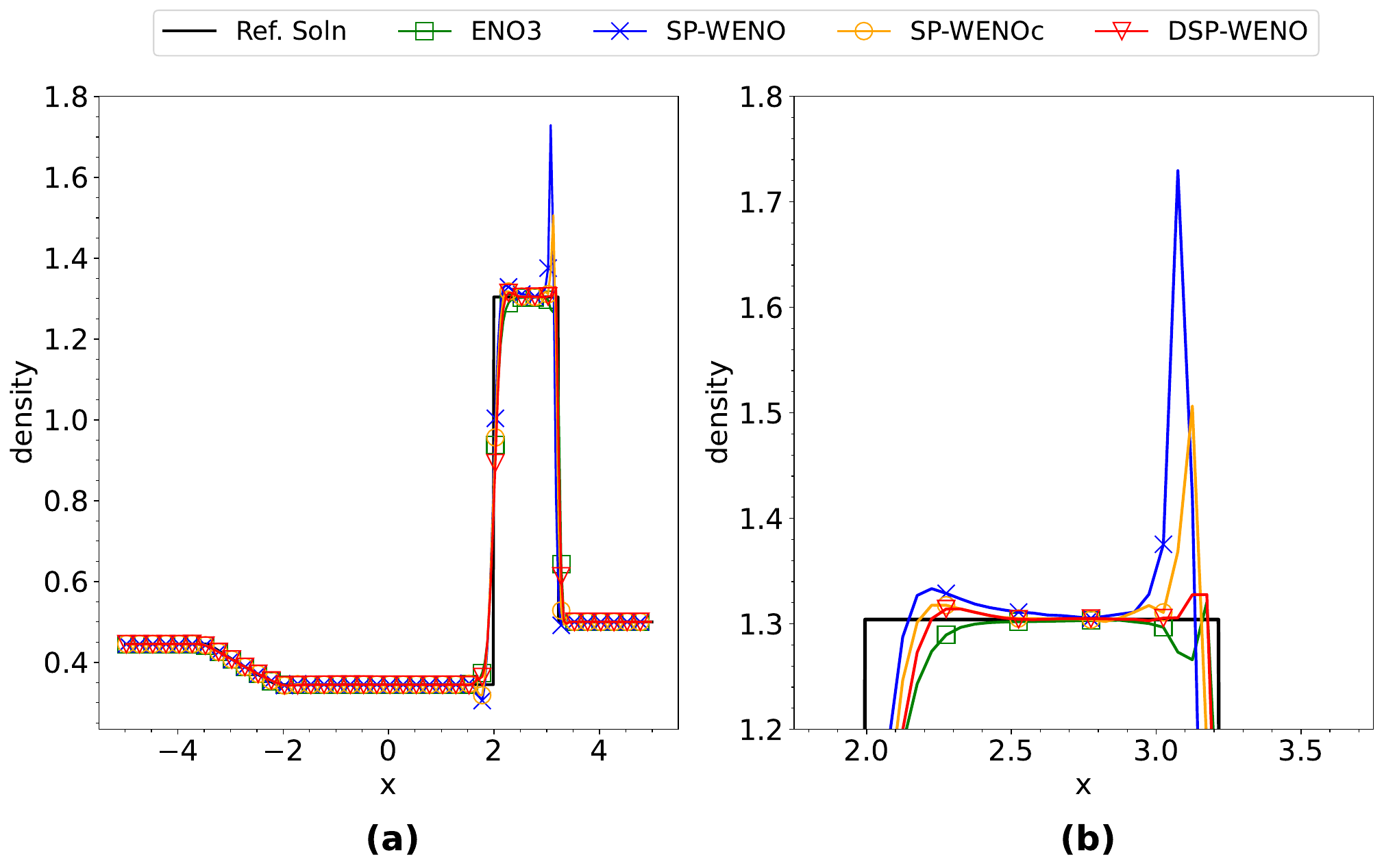}
    \caption{Lax Shock Tube Test. Comparison of density profile at final time $T=1.3$ for different reconstruction methods: \textbf{(a)} $x\in[-5,5]$, \textbf{(b)} $x\in[1.75,3.75]$.}
    \label{fig:lax}
\end{figure}

\textbf{2D Isentropic Vortex:} We consider the advection of a smooth isentropic vortex and perform a mesh-refinement study. The domain is $[-5,5]\times[-5,5]$, final time is $T=10$, and CFL = 0.5. The initial conditions are given by
\begin{flalign*}
\vel_1 {=} 1-\frac{\Gamma y}{2\pi}\exp{\Big(\frac{1-r^2}{2}\Big)},  \ \  \vel_2 {=}  \frac{\Gamma x}{2\pi}\exp{\Big(\frac{1-r^2}{2}\Big)},  \ \  \temp {=} 1 - \frac{(\gamma-1)\Gamma^2}{8\gamma\pi^2}\exp{(1-r^2)},
\end{flalign*}
where $\temp$ is the fluid temperature field, $r^2=x^2+y^2$
and $\Gamma=5$ (the vortex strength). Further, we have $\rho=\temp^{1/(\gamma-1)}$ and $p=\temp^{\gamma/(\gamma-1)}$ due to the isentropic conditions. Assuming periodic boundary condition, the vortex moves horizontal with unit velocity and completes one full periodic cycle at the final time. Table \ref{tab:isent_vortex} shows the discrete $L^1$ errors and convergence rates for the density with various reconstructions. We observe that ENO3 is unable to achieve third-order convergence, which we hypothesize is due to the linear instabilities with ENO \cite{rogerson1990numerical}. The SP-WENO variants lead to third-order accuracy, with the results using \net being more dissipative as compared to SP-WENO and SP-WENOc. We reiterate that this behavior is due to SP-WENO and SP-WENOc having nearly zero reconstructed jumps in larger regions of the domain, which is not preferable in the presence of discontinuities.\\


\begin{table}[!h]
    \centering
    \begin{tabular}{|c|c c|c c|c c|c c|}
     \hline
     N & \multicolumn{2}{c|}{ENO3} & \multicolumn{2}{c|}{SP-WENO} & \multicolumn{2}{c|}{SP-WENOc} & \multicolumn{2}{c|}{\net} \\
     & Error & Rate & Error & Rate & Error & Rate & Error & Rate \\
     \hline
     50  & 1.17e-1 & -    & 7.56e-2 & -    & 7.51e-2 & -    & 2.09e-1 & - \\
     100 & 1.75e-2 & 2.74 & 6.62e-3 & 3.51 & 6.73e-3 & 3.48 & 2.13e-2 & 3.30 \\
     150 & 7.13e-3 & 2.22 & 1.37e-3 & 3.88 & 1.40e-3 & 3.87 & 5.94e-3 & 3.15 \\
     200 & 3.60e-3 & 2.38 & 4.59e-4 & 3.80 & 4.69e-4 & 3.80 & 2.42e-3 & 3.13 \\
     300 & 1.48e-3 & 2.19 & 1.01e-4 & 3.73 & 1.03e-4 & 3.74 & 6.73e-4 & 3.15 \\
     400 & 7.86e-4 & 2.20 & 3.57e-5 & 3.61 & 3.64e-5 & 3.61 & 2.73e-4 & 3.14 \\
     \hline
    \end{tabular}
    \caption{$L^1$ errors in density for advecting isentropic vortex with different reconstructions.}
    \label{tab:isent_vortex}
\end{table}

\textbf{2D Riemann problem (configuration 12):}  We consider a two-dimensional Riemann problem for the Euler equations whose initial conditions on the domain $[0,1]\times[0,1]$ are given by
\begin{flalign*}
(\rho,\vel_1,\vel_2,p) =\ & 
(0.5313,0,0,0.4) \chf{\mathcal{Q}_1}(x,y) +
(1,0.7276,0,1) \chf{\mathcal{Q}_2}(x,y) \\
&+ (0.8,0,0,1) \chf{\mathcal{Q}_3}(x,y) +(1,0,0.7276,1)\chf{\mathcal{Q}_4}(x,y)
\end{flalign*}
In this case, the evolved solution (with Neumann boundaries) comprises two shock waves and two contact discontinuities. The problem is solved on a mesh consisting of 400 $\times$ 400 cells until a final time $T=0.25$ with CFL = 0.5. All methods sharply resolve the shocks and contact waves as shown in Figure \ref{fig:2D_riemann12} (see Section \ref{sec:sm_2DRiemann_conf12} in the SM for the SP-WENO and SP-WENOc density profiles). However, SP-WENO and SP-WENOc lead to significant overshoots, which can once again be seen in the one-dimensional slices shown in Figure \ref{fig:2D_riemann12}. On the other hand, the overshoots are generally smaller with \net and comparable to ENO3. The reference solution is once again generated using ENO3 by solving the problem on $1200 \times 1200$ mesh.\\

\begin{figure}
    \centering
    \subfigure[ENO3]{\includegraphics[width=2.5in]{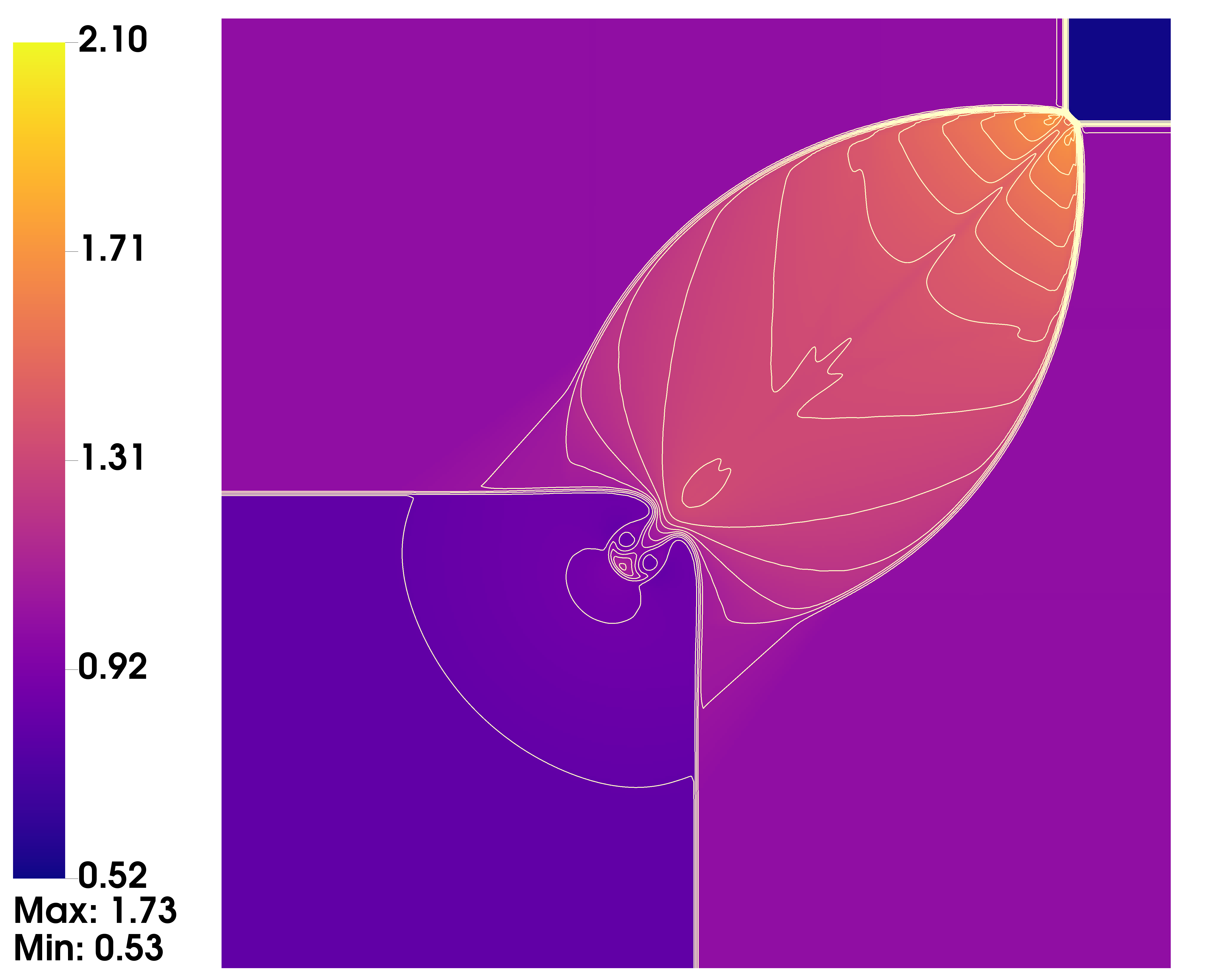}}
    \subfigure[\net]{\includegraphics[width=2.5in]{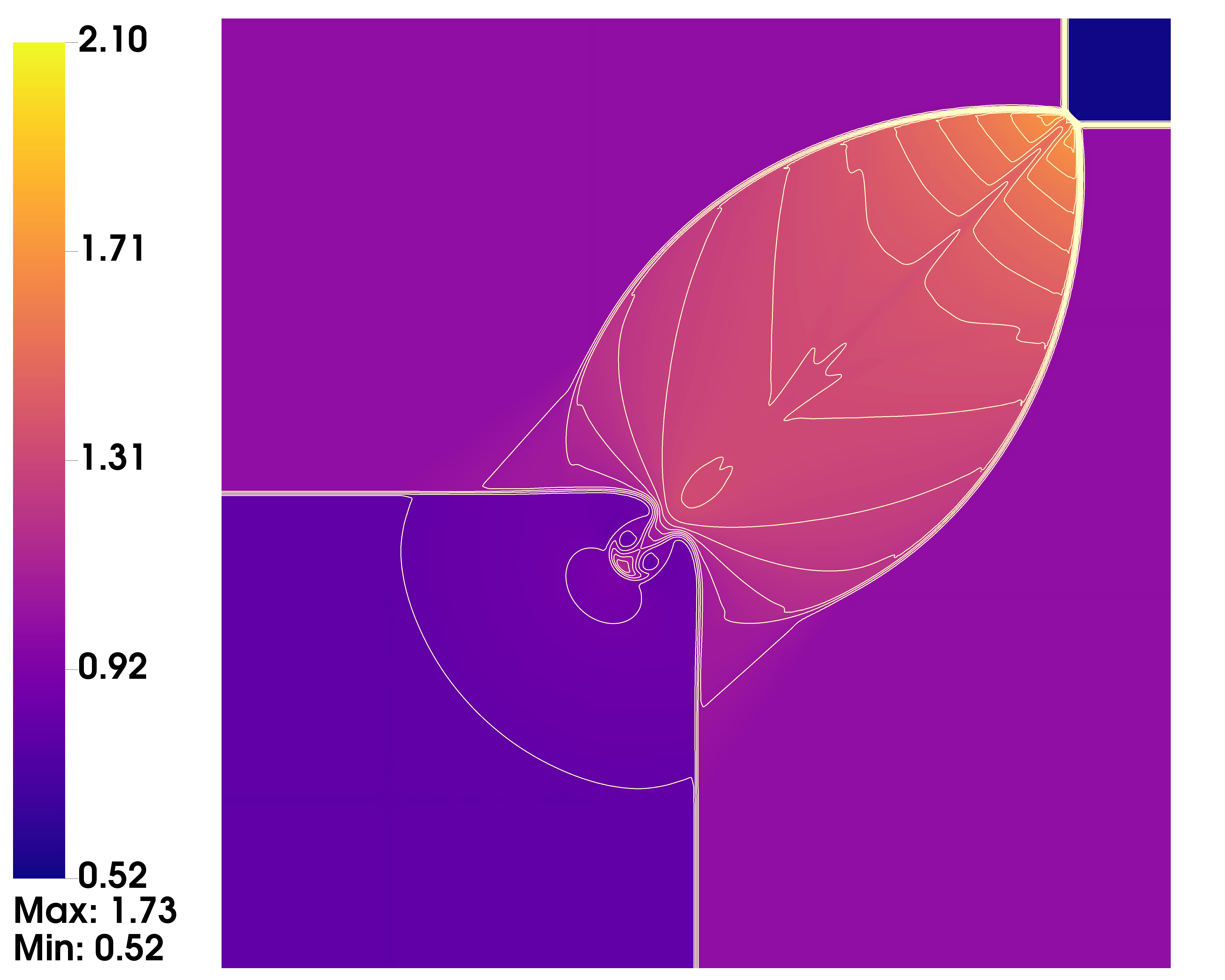}}
    \subfigure{\includegraphics[width=3.8in]{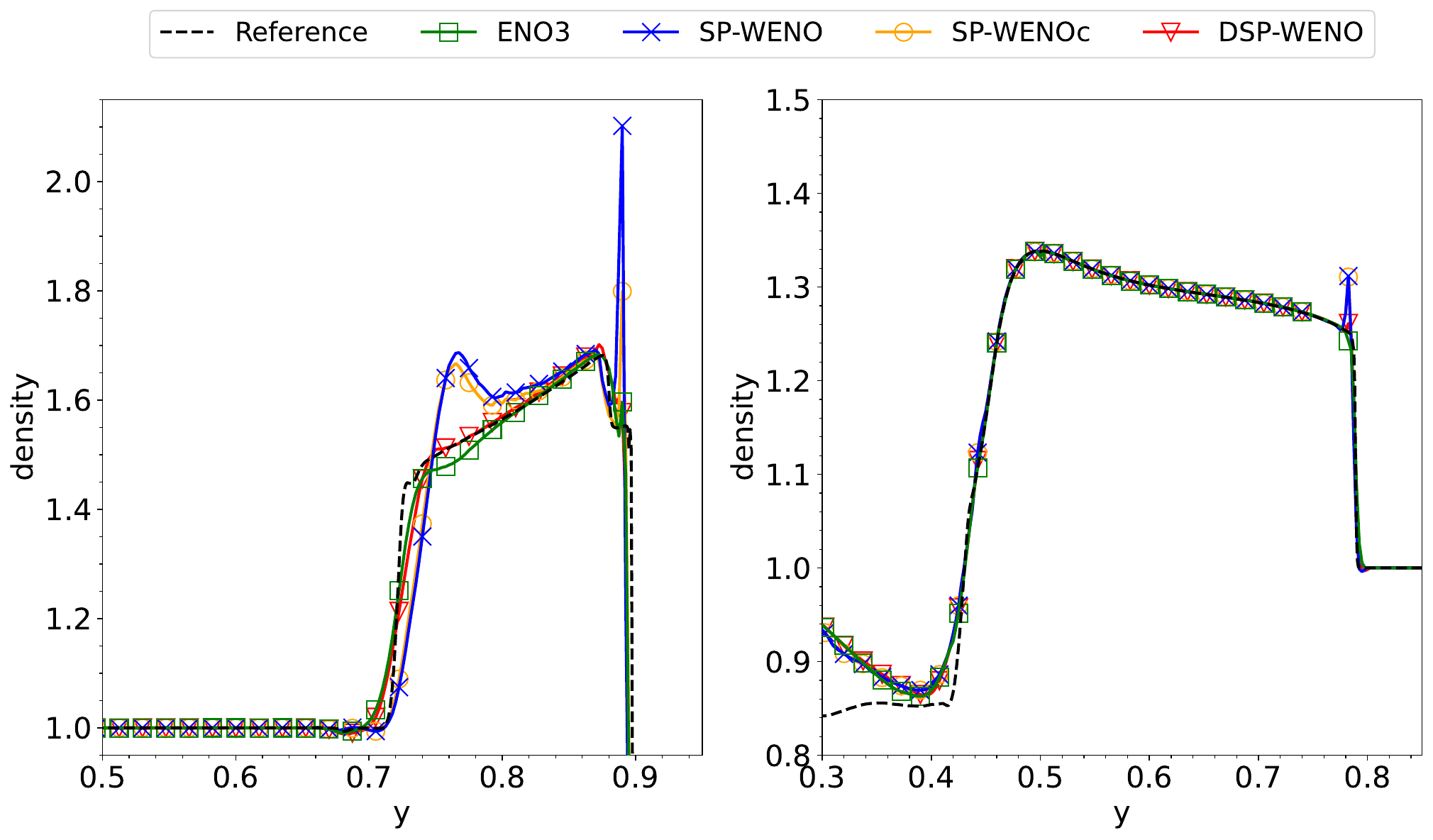}}
    \caption{2D Riemann problem (conf. 12): Density profiles at time $T=0.25$ with 30 contour lines between 0.52 and 2.2. Comparison of different reconstruction methods (top). One-dimensional slices of the solution using each reconstruction methods through $x=0.89$ (bottom left) and $x=0.5$ (bottom right).}
    \label{fig:2D_riemann12}
\end{figure}

\textbf{2D Riemann problem (configuration 3):} We consider another two-dimensional Riemann problem for the Euler equations \cite{kurganov2002solution} whose initial conditions are
\begin{flalign*}
(\rho,\vel_1,\vel_2,p) = & 
(1.5,0,0,1.5) \chf{\mathcal{Q}_1}(x,y) +
(0.5323,1.206,0,0.3) \chf{\mathcal{Q}_2}(x,y) \\
&+ (0.138,1.206,1.206,0.029) \chf{\mathcal{Q}_3}(x,y) +(0.5323,0,1.206,0.3)\chf{\mathcal{Q}_4}(x,y)
\end{flalign*}
where $\mathcal{Q}_1 =\{x>0.5, y>0.5\}$, $\mathcal{Q}_2 =\{x\leq0.5, y>0.5\}$, $\mathcal{Q}_3 =\{x\leq0.5, y\leq0.5\}$, and $\mathcal{Q}_4 =\{x>0.5, y\leq0.5\}$ are the four Cartesian quadrants.
When solved on the domain $[0,1]\times[0,1]$ with Neumann boundary conditions, the solution evolves into four interacting shock waves. We solve the problem on a mesh consisting of 400 $\times$ 400 cells until a final time $T=0.3$ with CFL = 0.4. We observe from the density profiles in Figure \ref{fig:2D_riemann3} that both SP-WENO and \net lead to a protruding artifact in the bottom left region of the domain which resembles carbuncle-like phenomenon 
\cite{quirk1997contribution,ismail2009proposed}. This artifact is significantly more pronounced for SP-WENO.
We remark here that entropy stability does not preclude the appearance of pathological behavior of this form. It was noted in \cite{chandrashekar2013kinetic,powers2015physical} that the carbuncle effect can be mitigated by introducing additional diffusion in the numerical scheme. In particular, we observed that the appearance of the carbuncle artifact is weakened (see Section \ref{sec:sm_2DRiemann_conf3}) by using a Rusanov-type dissipation operator of the form $\mathbf{\Lambda}_{i+\frac{1}{2}} = (|u| + a)_{i+\frac{1}{2}} \mathbf{I}$ in \eqref{high_order_ent_stable_flux}.

Empirical results in literature \cite{chandrashekar2013kinetic} suggest that the carbuncle typically arises in numerical schemes that resolve contact discontinuities well. This merits further investigation to assess the properties of SP-WENO-type reconstruction schemes in this context. However, this lies beyond the scope of the present work and will be explored in the future.\\

\begin{figure}
    \centering
    \subfigure[ENO3]{\includegraphics[width=2.5in]{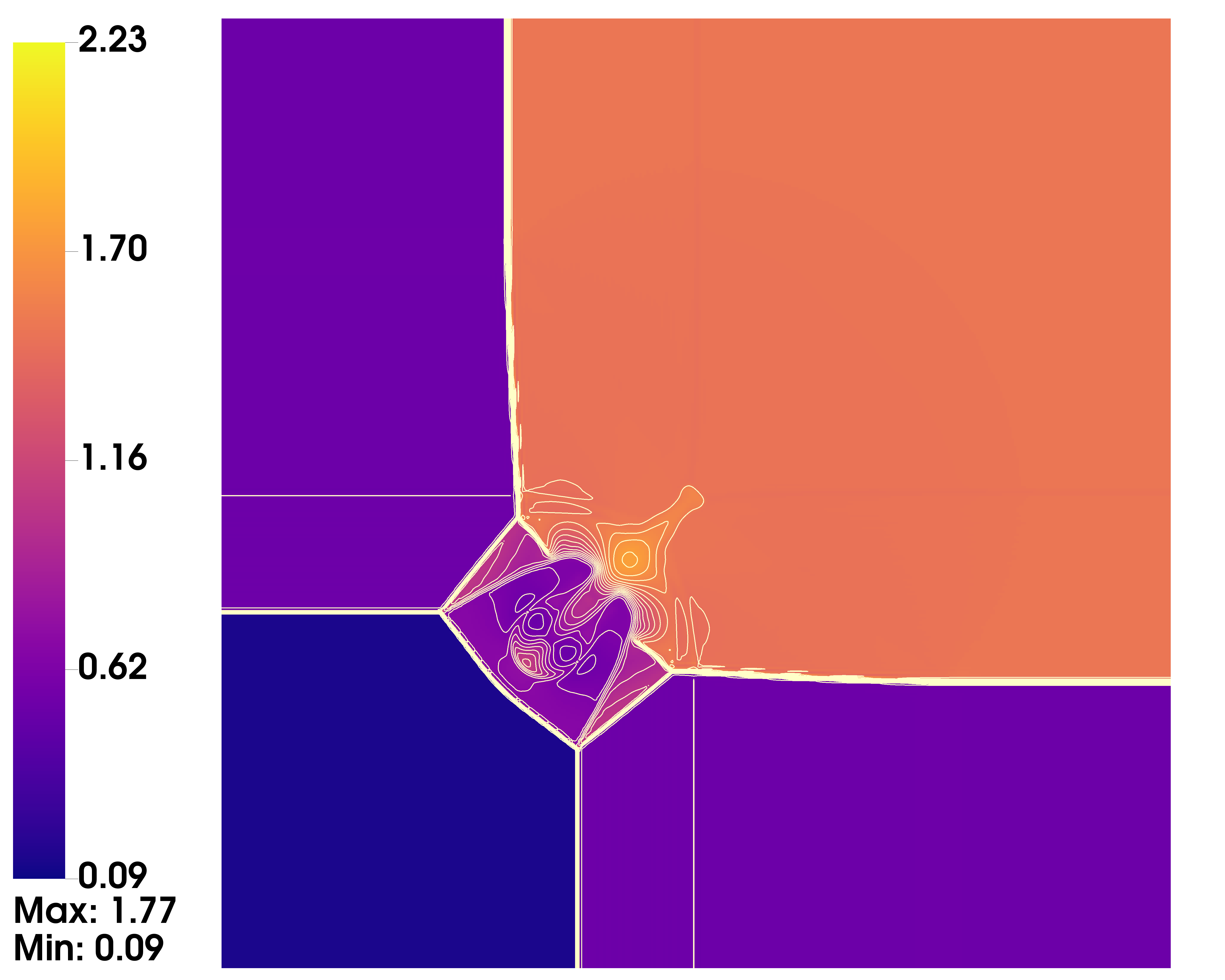}}
    \subfigure[SP-WENO]{\includegraphics[width=2.5in]{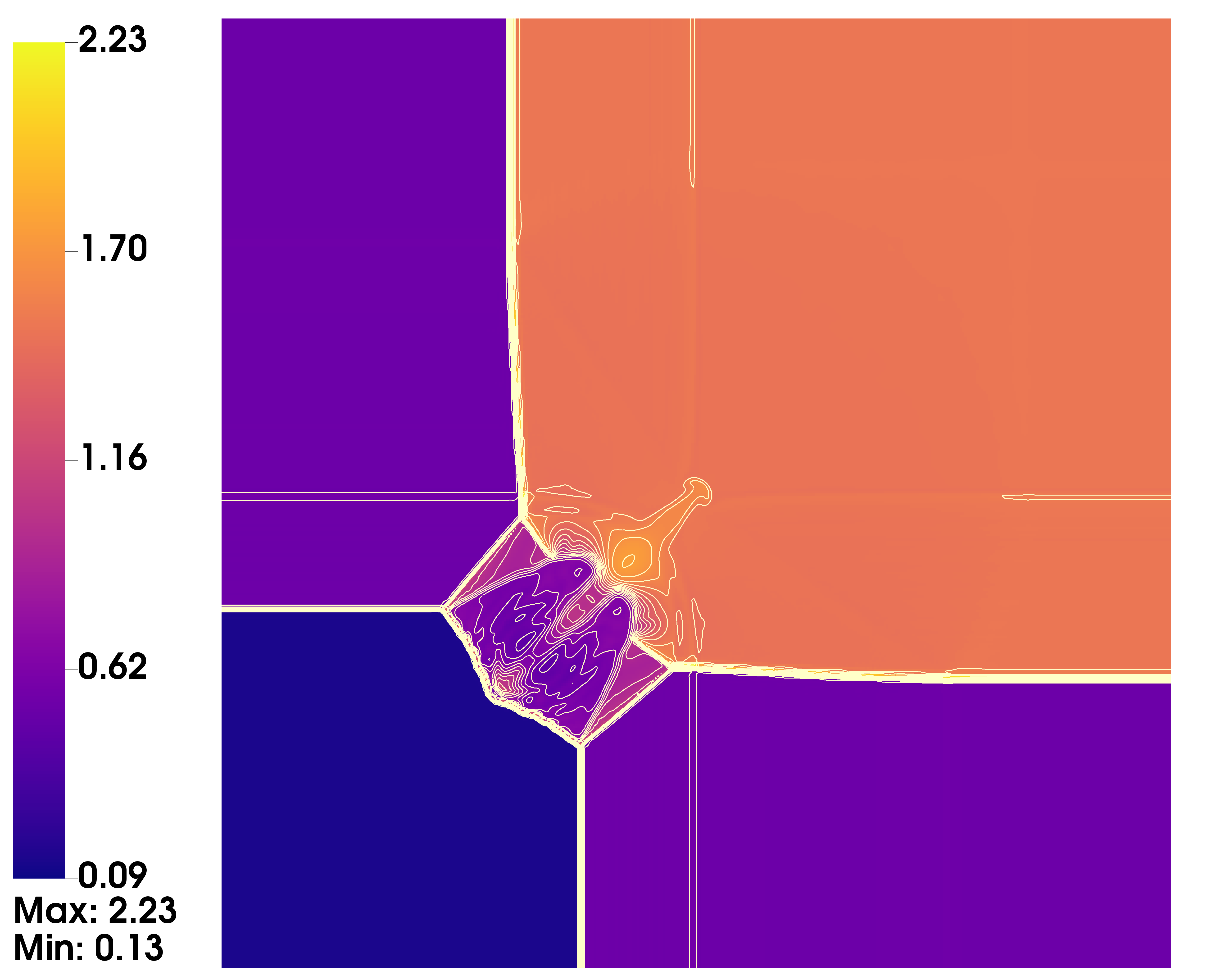}}
    \subfigure[SP-WENOc]{\includegraphics[width=2.5in]{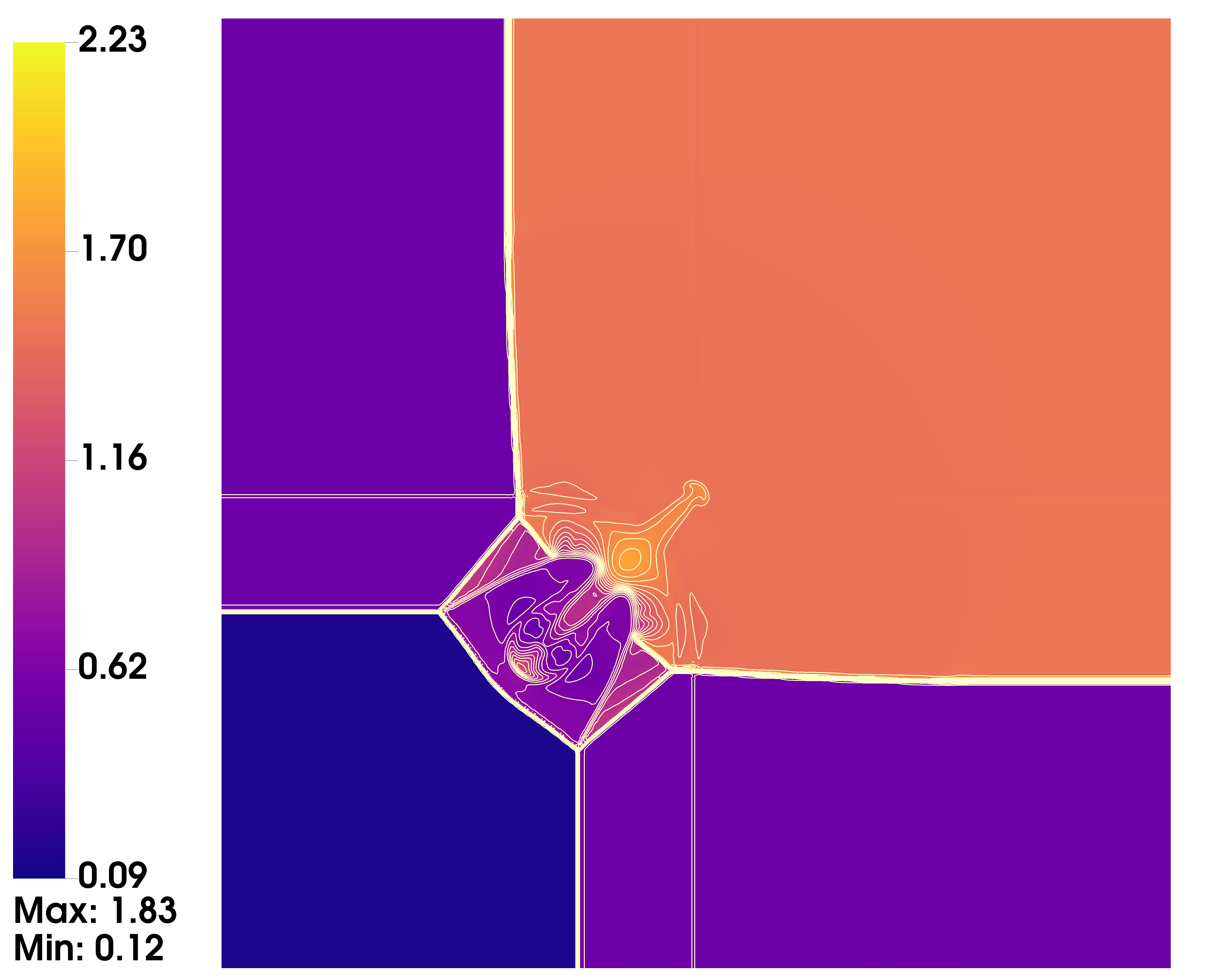}}
    \subfigure[\net]{\includegraphics[width=2.5in]{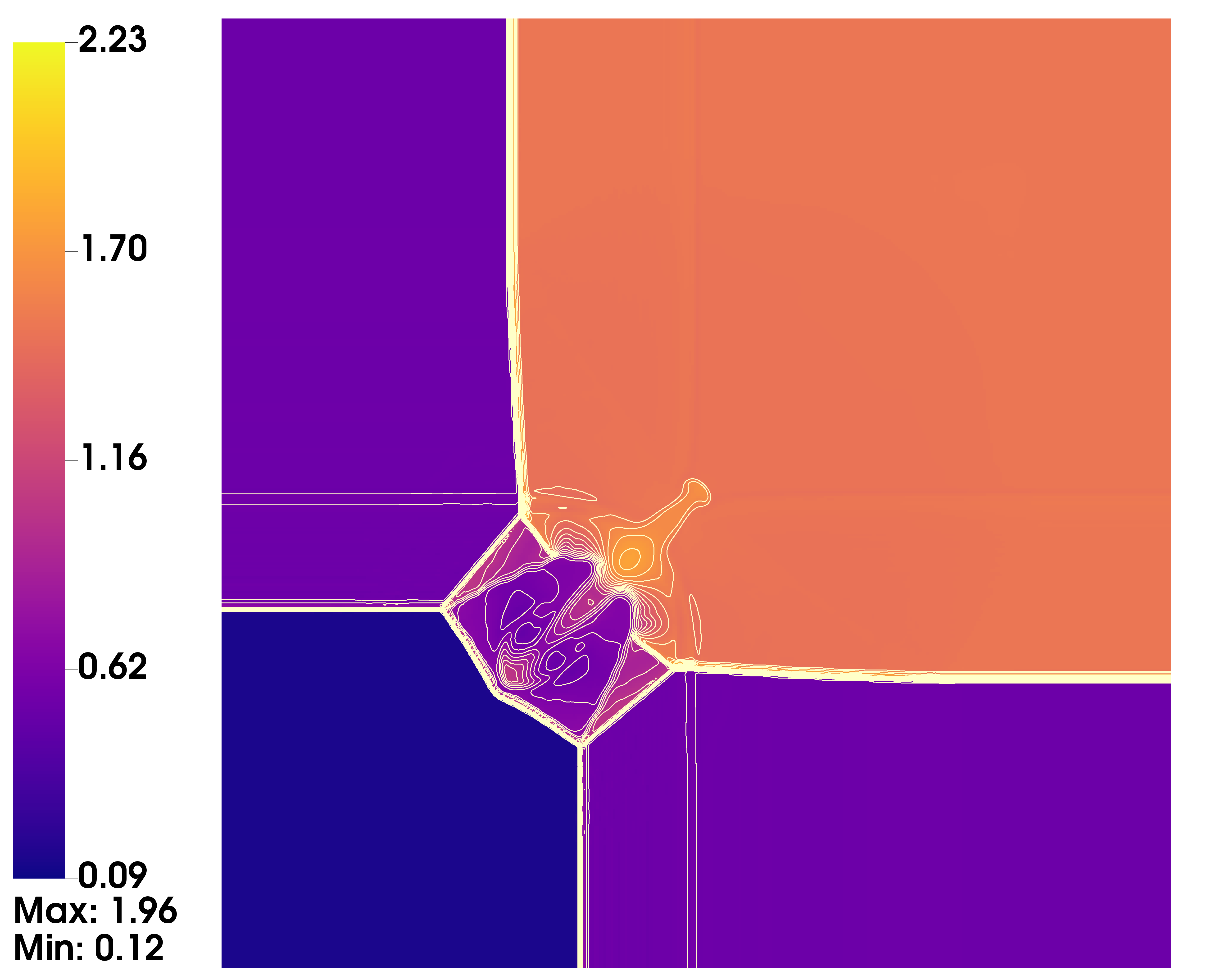}}
    \caption{2D Riemann problem (conf. 3): Density profiles at time $T=0.3$ with 30 contour lines between 0.1 and 2.28. Comparison of different reconstruction methods.}
    \label{fig:2D_riemann3}
\end{figure}


\textbf{Kelvin-Helmholtz Instability:} Finally, we consider a more complex two-dimensional problem for the Euler equations. This is the well-known Kelvin-Helmholtz instability which describes a shear flow separating two fluid regions with differing density, leading to two-dimensional turbulence. The initial conditions \cite{san2015evaluation} on the domain $[-0.5,0.5]\times[-0.5,0.5]$ are given by
 \begin{flalign*}
(\rho,\vel_1,\vel_2,p) = & 
(1,0.5,\epsilon,2.5) \chf{\{y>0.25\}}(x,y) +
(2,-0.5,\epsilon,2.5) \chf{\{|y|\leq 0.25\}}(x,y) \\
&+ (1,0.5,\epsilon,2.5) \chf{\{y< -0.25\}}(x,y)
\end{flalign*}
with periodic boundary conditions. When $\epsilon = 0$, the initial conditions describe a stationary solution of the PDE system. However, when a small perturbation is introduced in the initial condition, which in our case is done by setting the vertical velocity to be $\epsilon = 0.01 \sin(2 \pi x)$, the shearing leads to an instability generating small-scale vortical structures at the interfaces.

 We solve this problem using \net in the TeCNO scheme on meshes consisting of 256 $\times$ 256 cells and 512 $\times$ 512 cells until a final time $T=3$ with CFL = 0.4. For this problem, there is no convergence with mesh refinement \cite{fjordholm2017construction}, with finer structures appearing as the mesh is successively refined, which can be seen in Figure \ref{fig:kelvin-helmholtz}. Simulation with ENO3, SP-WENO and SP-WENOc on the same meshes show a similar behavior (see Section \ref{sec:sm_kelvin-helmholtz}), although the solutions with the different reconstructions on the same mesh look very different. We note that while the mesh convergence for a single deterministic sample does not exist, the convergence of the statistics of random samples (with randomized perturbations) does exist within the framework of measure-valued solutions \cite{fjordholm2017construction} or statistical solutions \cite{fjordholm2020statistical}.

\begin{figure}
    \centering
    \subfigure[$256 \times 256$ mesh]{\includegraphics[width=2.5in]{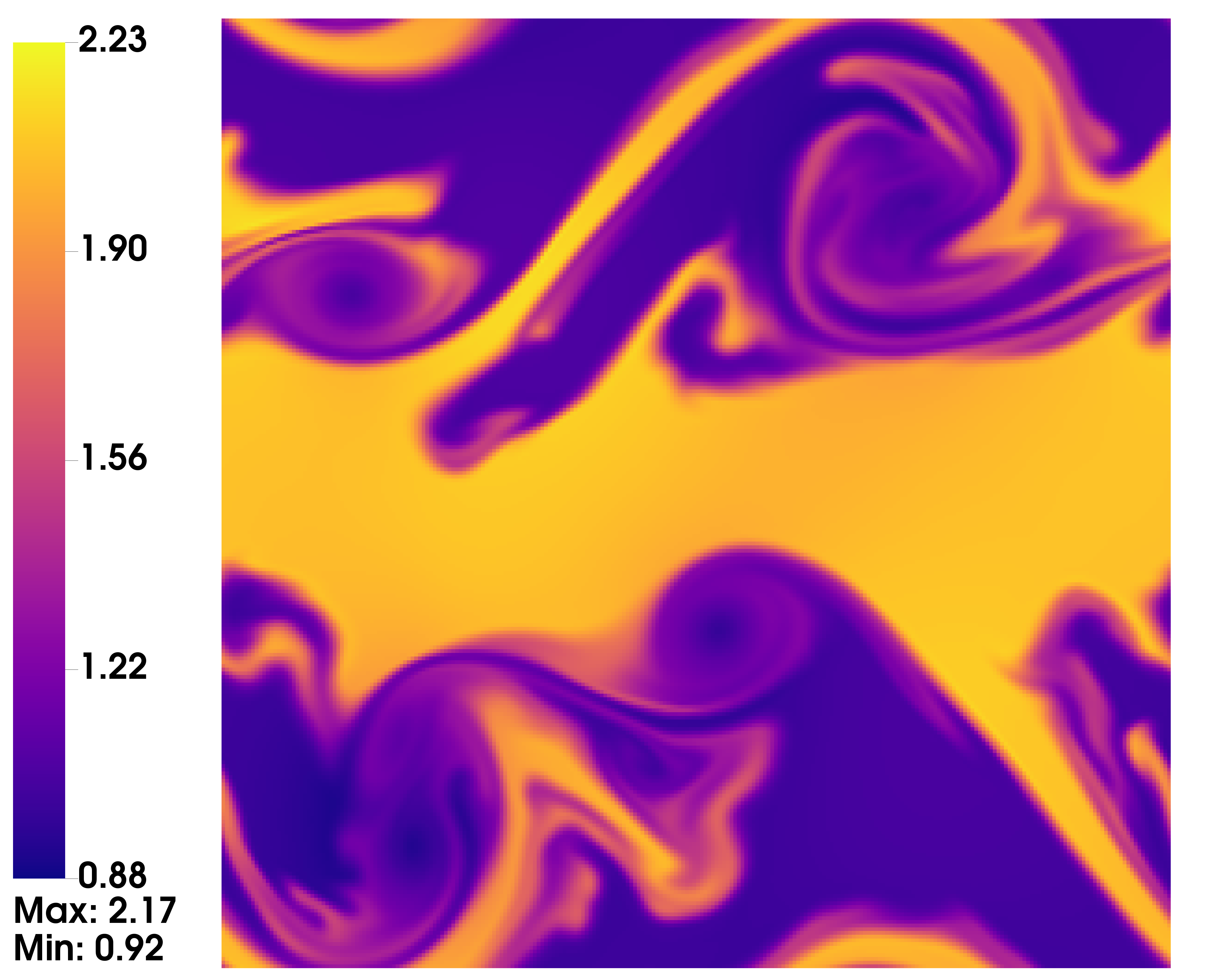}}
    \subfigure[$512 \times 512$ mesh]{\includegraphics[width=2.5in]{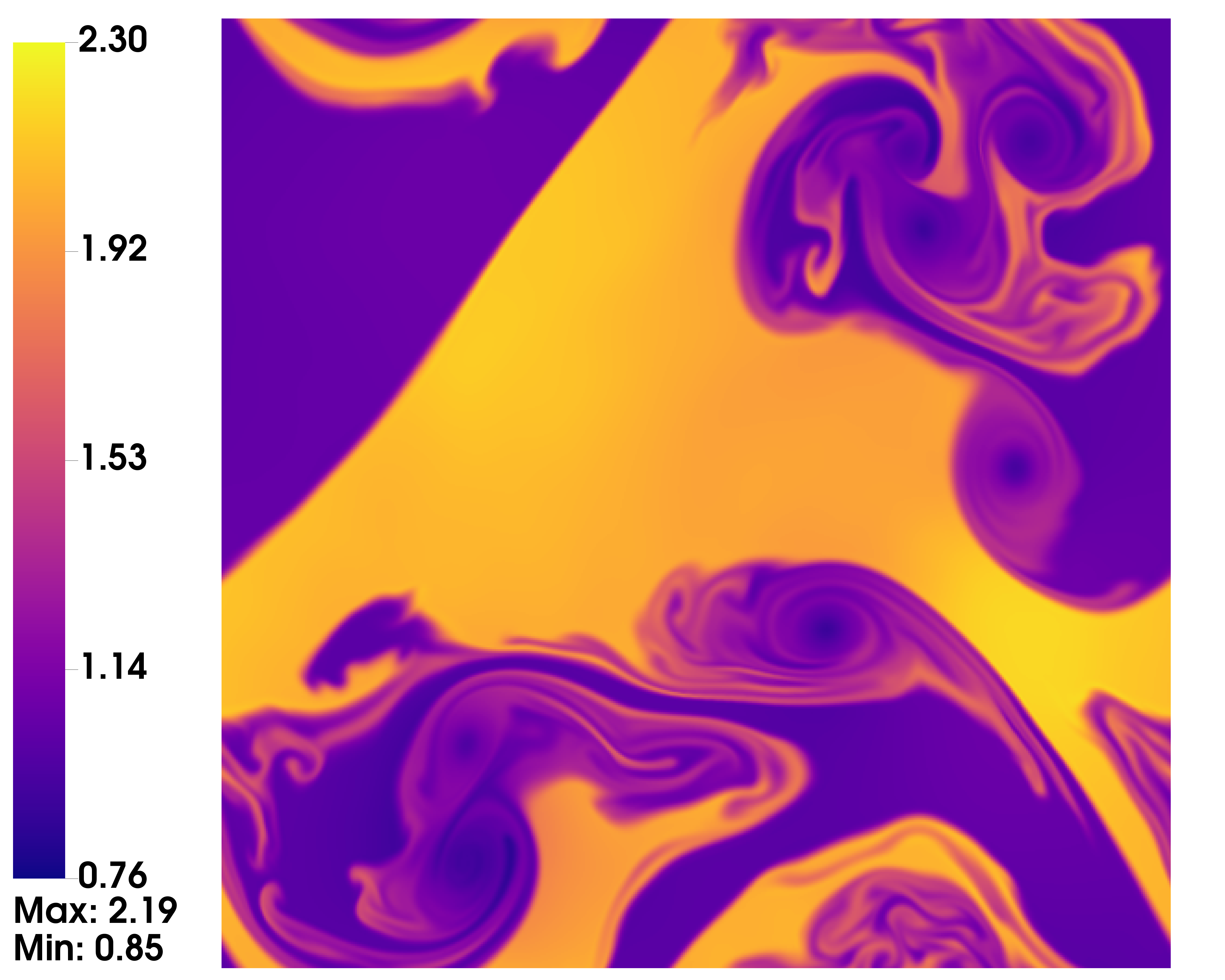}}
    \caption{Kelvin-Helmholtz Instability: Density profiles with \net at time $T=3$ on different mesh sizes.}
    \label{fig:kelvin-helmholtz}
\end{figure}

\section{Conclusion}\label{sec:conclusion}
In this work, we designed a novel neural network-based third-order WENO scheme called \net, which is guaranteed to satisfy the sign property. Thus, \net can be used within the TeCNO framework to obtain high-order entropy stable finite difference schemes. The motivation behind the proposed approach was two-fold: i) overcome the linear instability issues faced by ENO reconstruction (see Tables \ref{tab:lin_adv} and \ref{tab:isent_vortex}), and ii) have better shock-capturing capabilities compared to existing WENO algorithms, i.e., SP-WENO and SP-WENOc, that satisfy the sign property.

A key element in the proposed strategy was to decouple the constraints guaranteeing the sign property and third-order accuracy (in smooth regions) from the learning process. A strong imposition of these constraints led to a convex polygonal feasible region from which the WENO weights need to be selected. Then a network was trained to adaptively choose the weights from the feasible region to ensure good reconstruction properties near discontinuities. In contrast, we could impose these constraints weakly by adding a penalty term to the loss functions (analogous to physics informed neural networks \cite{raissi2019}). However, this would lead to the following challenges:
\begin{itemize}
    \item The weak imposition would not guarantee the satisfaction of the sign property in all situations, thus making it impossible to prove entropy stability in the TeCNO framework. 
    \item Neural networks trained on data extracted at a particular mesh resolution are rarely capable of demonstrating mesh convergence when tested on data from finer grids. In fact, the training (and test) errors typically plateau after a certain number of epochs, with the error values being several orders of magnitude larger than machine epsilon.
\end{itemize}     
Thus, there are major benefits of decoupling such constraints from the learning process and imposing them strongly.

The data used to train \net did not require solving conservation laws. Instead the data was generated from a library of functions with varying smoothness, which mimic the local structures typically arising in solutions to conservation laws. Thus, the cost of generating training data is insignificant, and the trained \net is model agnostic, i.e., does not depend on a specific conservation law. In other words, the \net needs to be trained once offline and can then be used for any conservation law. This strategy was based on similar ideas first considered in \cite{ray2018artificial} for designing troubled-cell detectors.

When comparing the numerical solutions obtained using various reconstruction methods satisfying the sign property, \net achieved third-order accuracy in smooth regions, while being more dissipative compared to SP-WENO and SP-WENOc. This was attributed to the fact that the reconstructed jump is mostly zero (or small) with SP-WENO and SP-WENOc. However, this had a negative impact near discontinuities where the solutions exhibited large spurious oscillations due to insufficient diffusion. \net significantly mitigated these spurious oscillations without compromising the order of accuracy in smooth regions. Further, \net did not suffer from linear instabilities faced by ENO3. It is important to note that the ENO3 stencil (corresponding to an interface) comprises six cells, while third-order SP-WENO and its variants (including \net) have access to the solution on a stencil with just four cells. 

The present work demonstrates that it is possible to use deep learning tools to learn adaptive reconstruction algorithms constrained to satisfy critical physical properties, such as the sign property. Further, it provides a framework to extend \net to high-order ($>3$) accurate reconstructions satisfying the sign property. This would involve formulating the sign property constraint on a larger stencil and transforming the constraints into corresponding convex polyhedral regions in high dimensional spaces for the WENO weights. Thus, instead of attempting to construct an explicit weight selection strategy by hand (which presents significant challenges in high dimensions), a neural network can learn the selection procedure from training data. Similar sign-preserving WENO-type reconstructions can also be designed on unstructured grids. These extensions will be explored in future work.

Finally, we recognize that the proposed \net has drawbacks. For one thing, we observe a carbuncle-like phenomenon in the 2D Riemann problem (configuration 3), which also appears with SP-WENO but surprisingly not SP-WENOc. We believe that this behavior is correlated by the ability of the schemes to resolve contact waves, which has also been numerically observed in the literature. Furthermore, since the improved performance of \net comes at the expense of additional computational cost, exploring hybridization techniques similar to \cite{chertock2022adaptive,clain2011high,farmakis2020weno} may allow us to dynamically switch between the reconstruction methods SP-WENO and \net so as to enjoy better performance in smooth regions while retaining the improved shock-capturing ability of \net.

\section*{Acknowledgements} This research did not receive any specific grant from funding agencies in the public, commercial, or not-for-profit sectors.

\printbibliography

\clearpage
\title{Supplementary Materials: Learning WENO for entropy stable schemes to solve conservation laws}
\emptythanks
\date{}

\setcounter{page}{1}
\setcounter{table}{0}
\setcounter{figure}{0}
\setcounter{equation}{0}
\setcounter{section}{0}
\setcounter{lem}{0}
\setcounter{thm}{0}
\setcounter{defn}{0}
\setcounter{remark}{0}

\let\thesectionWithoutSM\thesection 
\renewcommand\thesection{SM\thesectionWithoutSM}
\let\theequationWithoutSM\theequation 
\renewcommand\theequation{SM\theequationWithoutSM}
\let\thefigureWithoutSM\thefigure 
\renewcommand\thefigure{SM\thefigureWithoutSM}
\let\thetableWithoutSM\thetable 
\renewcommand\thetable{SM\thetableWithoutSM}
\let\thelemWithoutSM\thelem 
\renewcommand\thelem{SM\thelemWithoutSM}
\let\theremarkWithoutSM\theremark 
\renewcommand\theremark{SM\theremarkWithoutSM}
\let\thealgorithmWithoutSM\thealgorithm 
\renewcommand\thealgorithm{SM\thealgorithmWithoutSM}

\maketitle

\section{SP-WENO Cases}\label{sec:theta_cases}

In Table \ref{tab:feasible}, we list the various cases of interest along with the corresponding sign property and accuracy constraints for third-order SP-WENO formulations. Note that the various $\Omega_\Theta$ listed in the table (column 2) form a disjoint partition of $\Ro^2$. Further, $\psi^+_\iph$ is used to describe $\Omega_\Theta$ in cases 2 and 3, which is a function of $\theta^+_i,\theta^-_{i+1}$ by virtue of \eqref{eqn:psi_pm} in the main text. We reintroduce the following notation from \cite{fjordholm2016sign} (with the $i+\frac{1}{2}$ subscript dropped from $\psi^+,\psi^-$ terms)
\begin{flalign}\label{eqn:L}
\mathcal{L} := \begin{cases}
    \frac{C_1}{\frac{1}{8}(1+\psi^+)} + \frac{C_2}{\frac{1}{8}(1+\psi^-)}, &\text{ if } \psi^+ \neq -1 \\
    C_1 - C_2 + 1, &\text{ if } \psi^+=\psi^-=-1
\end{cases},
\end{flalign}
which is also used to define the sign property constraints in Table \ref{tab:feasible}. The expression of $\mathcal{L}$ can be obtained from the constraint \eqref{eqn:sign_prop_constraint} (in the main text) through simple algebraic manipulations. For full details, we refer interested readers to \cite{fjordholm2016sign}.

The SP-WENO and SP-WENOc formulations satisfy the constraints of each scenario. The DL-based \net formulation is also constructed to satisfy the feasibility constraints in each of these scenarios. Note that these constraints are strongly embedded into \net by selecting appropriate vertices of the convex feasible region.

\begin{remark}
    The consistency constraint $-\frac{3}{8}\leq C_1,C_2\leq\frac{1}{8}$ must always be satisfied. Thus, in cases where there are no constraints due to the sign property or third-order accuracy, consistency will be the only constraint governing the feasible region $\Omega_C$.
\end{remark}

\renewcommand{\arraystretch}{1.75}
\begin{table}[!htbp]
\centering
\begin{tabular}{|c|c|l|l|l|}
\hline
\textbf{Case} & $\bm {\Omega_\Theta}$                                                                                       & {\textbf{Sign Prop. Const.}}                                                                                          & {\textbf{Acc. Const.}}    & \textbf{Remarks}       \\ \hline
1             & $\theta_i^+,\theta_{i+1}^->1$                                                                              & \multicolumn{1}{c|}{$C_1 = C_2 = \frac{1}{8}$}                                                                                    & \multicolumn{1}{c|}{None}                            & \multicolumn{1}{c|}{$\tilde{w}_0 = w_1 = 0$} \\ \hline
2a            & \begin{tabular}[c]{@{}c@{}}$\theta_i^+<1,\theta_{i+1}^->1$\\ $-1\leq\psi_{i+\frac{1}{2}}^+<0$\end{tabular} & \multicolumn{1}{c|}{\begin{tabular}[c]{@{}c@{}}$\mathcal{L} \leq 1$\end{tabular}} & \multicolumn{1}{c|}{$C_1 , C_2 \sim \mathcal{O}(h)$} & \multicolumn{1}{c|}{C-region}                        \\ \hline
2b            & \begin{tabular}[c]{@{}c@{}}$\theta_i^+<1,\theta_{i+1}^->1$\\ $\psi_{i+\frac{1}{2}}^+<-1$\end{tabular} & \multicolumn{1}{c|}{\begin{tabular}[c]{@{}c@{}}$\mathcal{L} \geq 1$\end{tabular}} & \multicolumn{1}{c|}{$C_1 , C_2 \sim \mathcal{O}(h)$} & \multicolumn{1}{c|}{C-region}                        \\ \hline
3a            & \begin{tabular}[c]{@{}c@{}}$\theta_i^+>1,\theta_{i+1}^-<1$\\ $-1\leq\psi_{i+\frac{1}{2}}^+<0$\end{tabular} & \multicolumn{1}{c|}{\begin{tabular}[c]{@{}c@{}}$\mathcal{L} \geq 1$\end{tabular}} & \multicolumn{1}{c|}{$C_1 , C_2 \sim \mathcal{O}(h)$} & \multicolumn{1}{c|}{C-region}                        \\ \hline
3b            & \begin{tabular}[c]{@{}c@{}}$\theta_i^+>1,\theta_{i+1}^-<1$\\ $\psi_{i+\frac{1}{2}}^+<-1$\end{tabular} & \multicolumn{1}{c|}{\begin{tabular}[c]{@{}c@{}}$\mathcal{L} \leq 1$\end{tabular}} & \multicolumn{1}{c|}{$C_1 , C_2 \sim \mathcal{O}(h)$} & \multicolumn{1}{c|}{C-region}                        \\ \hline
4a            & \begin{tabular}[c]{@{}c@{}}$\theta_{i+1}^-=1,\theta_i^+>1$\end{tabular} & \multicolumn{1}{c|}{\begin{tabular}[c]{@{}c@{}}$C_1=\frac{1}{8}$\end{tabular}} & \multicolumn{1}{c|}{None} & \multicolumn{1}{c|}{$w_1 = 0$}                        \\ \hline
4b            & \begin{tabular}[c]{@{}c@{}}$\theta_{i+1}^-=1,\theta_i^+\leq1$\end{tabular} & \multicolumn{1}{c|}{None} & \multicolumn{1}{c|}{None} & \multicolumn{1}{c|}{}                        \\ \hline
5a            & \begin{tabular}[c]{@{}c@{}}$\theta_i^+=1,\theta_{i+1}^->1$\end{tabular} & \multicolumn{1}{c|}{\begin{tabular}[c]{@{}c@{}}$C_2=\frac{1}{8}$\end{tabular}} & \multicolumn{1}{c|}{None} & \multicolumn{1}{c|}{$\tilde{w}_0 = 0$}                        \\ \hline
5b            & \begin{tabular}[c]{@{}c@{}}$\theta_i^+=1,\theta_{i+1}^-\leq1$\end{tabular} & \multicolumn{1}{c|}{None} & \multicolumn{1}{c|}{None} & \multicolumn{1}{c|}{}                        \\ \hline
6            & \begin{tabular}[c]{@{}c@{}}$\theta_i^+,\theta_{i+1}^-<1$\end{tabular} & \multicolumn{1}{c|}{None} & \multicolumn{1}{c|}{None} & \multicolumn{1}{c|}{}                        \\ \hline
\end{tabular}
\caption{Summary of the possible $\Omega_\Theta$ along with the sign property and accuracy constraints defining the associated feasible region $\Omega_C$. Here $\mathcal{L}$ is given by \eqref{eqn:L}.}
\label{tab:feasible}
\end{table}
\renewcommand{\arraystretch}{1}

\section{Description of the Vertex Selection Algorithm}\label{sec:SM_vertex}
We provide a complete and in-depth description of how the vertices of the convex $\Omega_C$ are selected. The pseudocode for the algorithm can be found in Algorithm \ref{alg:vertex_selection}. Note that the vertex algorithm is applied for each four-cell stencil in the mesh. Further, the number of vertices needed to construct convex polygons described below across all the possible scenarios is at most five. Thus, to ensure that the neural network always outputs the same number of convex weights $\bm{\alpha}$ corresponding to the vertices, we constrain the output dimension of the network to be five. In the various scenarios where the number of vertices of $\Omega_C$ is less than five, we augment the set of vertices to a set of five vertices by including either redundant vertices or some averaged state of the figure's vertices.

With the exception of cases 2 and 3 (see Table \ref{tab:feasible}), the vertex selection for most scenarios is quite straight-forward. This is due to the fact that all cases besides 2 and 3 do not have an additional accuracy constraint. In particular, for
\begin{itemize}
    \item \textbf{Case 1:} The feasible region $\Omega_C$ reduces to a single node $\left(\frac{1}{8},\frac{1}{8}\right)$. Thus, all five vertices are taken as this node.
    \item \textbf{Case 4a:} $\Omega_C$ reduces the line segment $C_1 = \frac{1}{8}, -\frac{3}{8} \leq C_2  \leq \frac{1}{8}$. Thus we take two vertices to be the right end $\left(\frac{1}{8},\frac{1}{8}\right)$, two vertices to be the left end $\left(-\frac{3}{8},\frac{1}{8}\right)$, and the final vertex to be the line segment's mid-point $\left(-\frac{1}{8},\frac{1}{8}\right)$.
    \item \textbf{Case 5a:} Is identical to Case 4a, but with the role of $C_1$ and $C_2$ interchanged.
    \item \textbf{Case 4b, 5b and 6:} There are not constraints (except for consistency) and $\Omega_C = \left[ -\frac{3}{8},\frac{1}{8}\right] \times \left[ -\frac{3}{8},\frac{1}{8}\right]$. Thus four vertices are chosen as the four corners of this square region, while the fifth vertex is taken to be the centroid of this square. 
\end{itemize}
Cases 2 and 3 are more involved since we require that $C_1,C_2=\mathcal{O}(h)$ in smooth regions. To address this, we construct an $\mathcal{O}(h)$ region in $\mathbb{R}^2$ as depicted in Figure \ref{fig:case_2_3_feasible_region}. The black dashed lines enclose a box representing the consistency requirements for the perturbations, i.e., $-3/8 \leq C_1,C_2 \leq 1/8$. The $\mathcal{L}=1$ line shown in orange separates the regions satisfying the sign property for Cases 2 and 3, represented as solid dark and light grey regions respectively. Within these grey regions, we construct smaller sub-regions that also satisfy the accuracy constraints. These sub-regions are shown in blue and green, and form the required feasible regions $\Omega_C$ for cases 2 and 3. To construct these sub-regions, we first define the box $\mathcal{H}=[\gamma_2,\gamma_1]\times[\gamma_2,\gamma_1] \subset [-3/8,1/8]\times [-3/8,1/8]$ where 
\begin{equation}\label{eqn:gammas}
\begin{aligned}
 \gamma_1 &= \min{\left( \max{\left(|\Delta z_{i-\frac{1}{2}}^*|,|\Delta z_{i+\frac{1}{2}}^*|,|\Delta z_{i+\frac{3}{2}}^*|\right)},\frac{1}{8}\right)},\\
 \gamma_2 &= -\min{\left( \max{\left(|\Delta z_{i-\frac{1}{2}}^*|,|\Delta z_{i+\frac{1}{2}}^*|,|\Delta z_{i+\frac{3}{2}}^*|\right)},\frac{3}{8}\right)}.   
\end{aligned}
\end{equation}
The region $\mathcal{H}$ is shown as a red dotted box in Figure \ref{fig:case_2_3_feasible_region}. Note that when the $z$ on the local stencil is smooth, $|\Delta z_{i-\frac{1}{2}}^*|,|\Delta z_{i+\frac{1}{2}}^*|,|\Delta z_{i+\frac{3}{2}}^*|=\mathcal{O}(h)$. Thus, we ensure that $\gamma_1,\gamma_2=\mathcal{O}(h)$, and searching for $(C_1,C_2)$ inside $\mathcal{H}$ would constrain $C_1,C_2=\mathcal{O}(h)$, which is the order constraint for cases 2 and 3 assuming smoothness for $z$. When a discontinuity is present, $\mathcal{H}$ can potentially expand to the whole consistency region for $C_1,C_2$. With the line $\mathcal{L}=1$ splitting $\mathcal{H}$, we have that the feasible region for Case 2 consists of the region above $\mathcal{L}=1$ and that of Case 3 consists of the region below $\mathcal{L}=1$. We further note that the closer the perturbations are to zero, the more accurate the reconstruction. In smooth regions, this is desirable, but near discontinuities, capturing shocks becomes the priority, so searching a wider search space becomes important.

As shown in Figure \ref{fig:case_2_3_feasible_region}, the line $\mathcal{L}=1$ splits $\mathcal{H}$ into a triangular and pentagonal $\Omega_C$ region, with two out of three and two out of five vertices lying on the line $\mathcal{L}=1$, respectively. A peculiarity of choosing $(C_1,C_2)$ on this line is that it ensures the left and right reconstructed states at the interface are equal, i.e, $\llbracket z\rrbracket_{i+\frac{1}{2}}=0$. As discussed in Section \ref{sec:sp_wenoc} in the main text, a zero reconstruction jump is undesirable in the vicinity of a discontinuity in the TeCNO framework. When choosing $(C_1,C_2)$ from a triangular region, two of whose vertices are on this line, which biases the selection to be on this line (this is what we also observed in practice). Thus, when constructing a triangular $\Omega_C$ above or below $\mathcal{L}=1$, we do not use the triangle that is a part of $\mathcal{H}$. Instead, we replace one of the vertices lying on $\mathcal{L}=1$ with another vertex pulled away from this line (without violating the consistency and sign property constraints) to form a flipped triangular $\Omega_C$ that is outside of $\mathcal{H}$. This flipped $\Omega_C$ still satisfies $C_1,C_2=\mathcal{O}(h)$, and thus any point selected in this triangle will also satisfy this constraint. For instance, let us consider Case 2 when $\psi^+<-1$, where $\Omega_C$ is triangular (see Figure \ref{fig:case_2_3_feasible_region}(c),(d)). In this scenario, according to Algorithm \ref{alg:vertex_selection}, the vertex outside $\mathcal{H}$ is $(\gamma_2,\frac{1}{8}(1+\psi^-)-\gamma_1\psi^-)$. For smooth regions, $C_1=\gamma_2=\mathcal{O}(h)$. Further, we can deduce that $(1+\psi^-)=\mathcal{O}(h)$. Since $\gamma_1=\mathcal{O}(h)$, it is clear that $C_2=\frac{1}{8}(1+\psi^-)-\gamma_1\psi^-=\mathcal{O}(h)$. Similar arguments can be made for Case 3 when $-1<\psi^+<0$ where $\Omega_C$ is once again triangular (see Figure \ref{fig:case_2_3_feasible_region}(a),(b)). To ensure that we are always working with five vertices, when the feasible region is a triangle, we augment the set of three vertices with the centroid of the triangle twice.

\begin{remark}
    When constructing the triangular $\Omega_C$ for Cases 2 and 3, we do not discard all vertices lying on $\mathcal{L}=1$. This is because while $\llbracket z\rrbracket_{i+\frac{1}{2}}=0$ is not desirable near discontinuities, it does lead to better accuracy near smooth regions. Thus, we retain at least one of these vertices to maintain a balance in performance between smooth and discontinuous reconstructions. 
\end{remark}

\begin{remark}
    We do not modify the vertices of the pentagonal $\Omega_C$ obtained in Cases 2 and 3, as only two of the five vertices are on $\mathcal{L}=1$. Thus, the pentagonal regions do not suffer from the biasing issue faced by triangular $\Omega_C$.
\end{remark}

\begin{remark}
    The vertex modification algorithm for the triangular $\Omega_C$ is not unique. After experimenting with a few configurations of choosing the shifted vertex, we found the one used in Algorithm \ref{alg:vertex_selection} yields the best performance. 
\end{remark}

\begin{figure}
    \centering
    \subfigure[]{
    \includegraphics[width=2.5in]{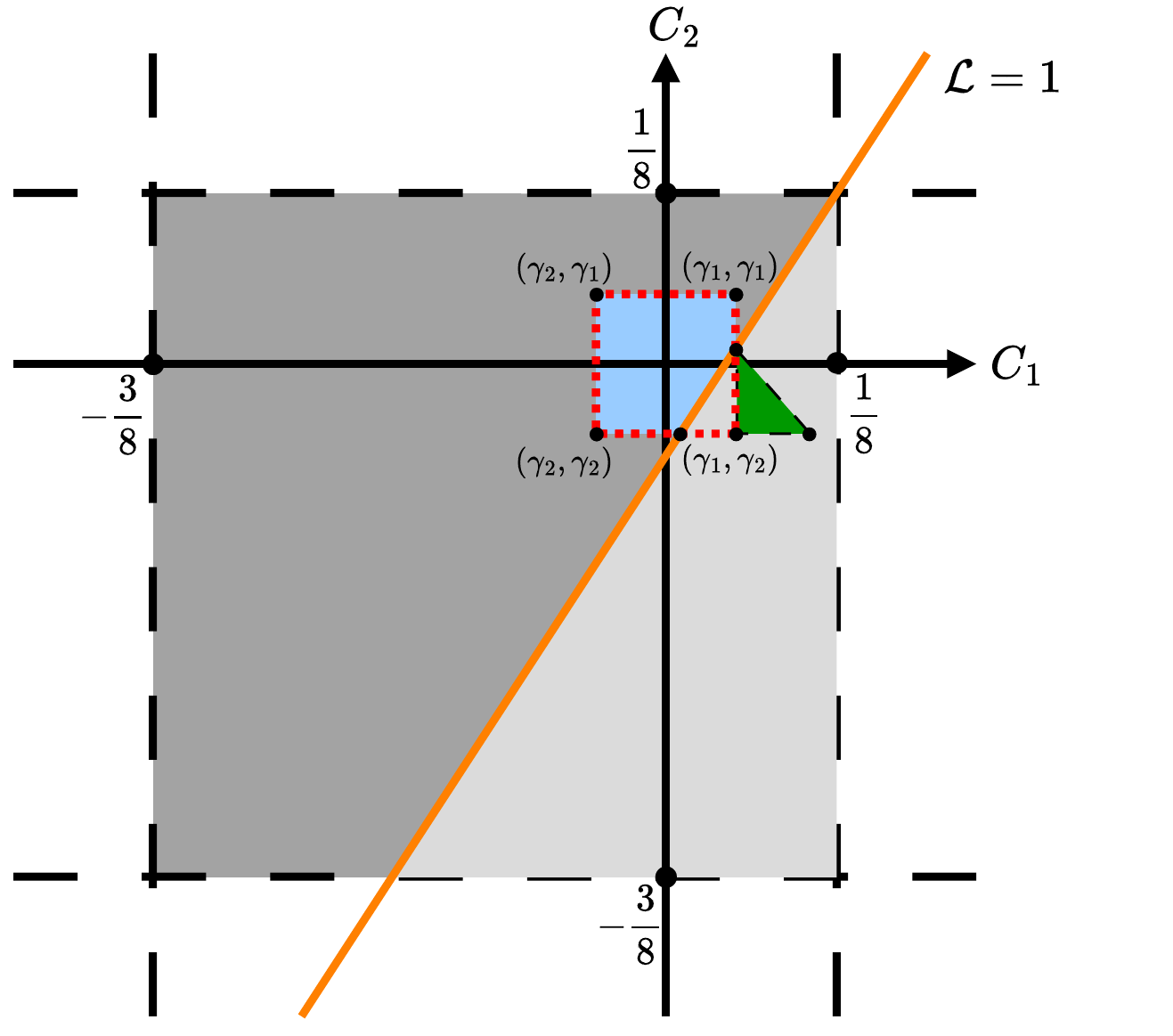}}
    \subfigure[]{\includegraphics[width=2in]{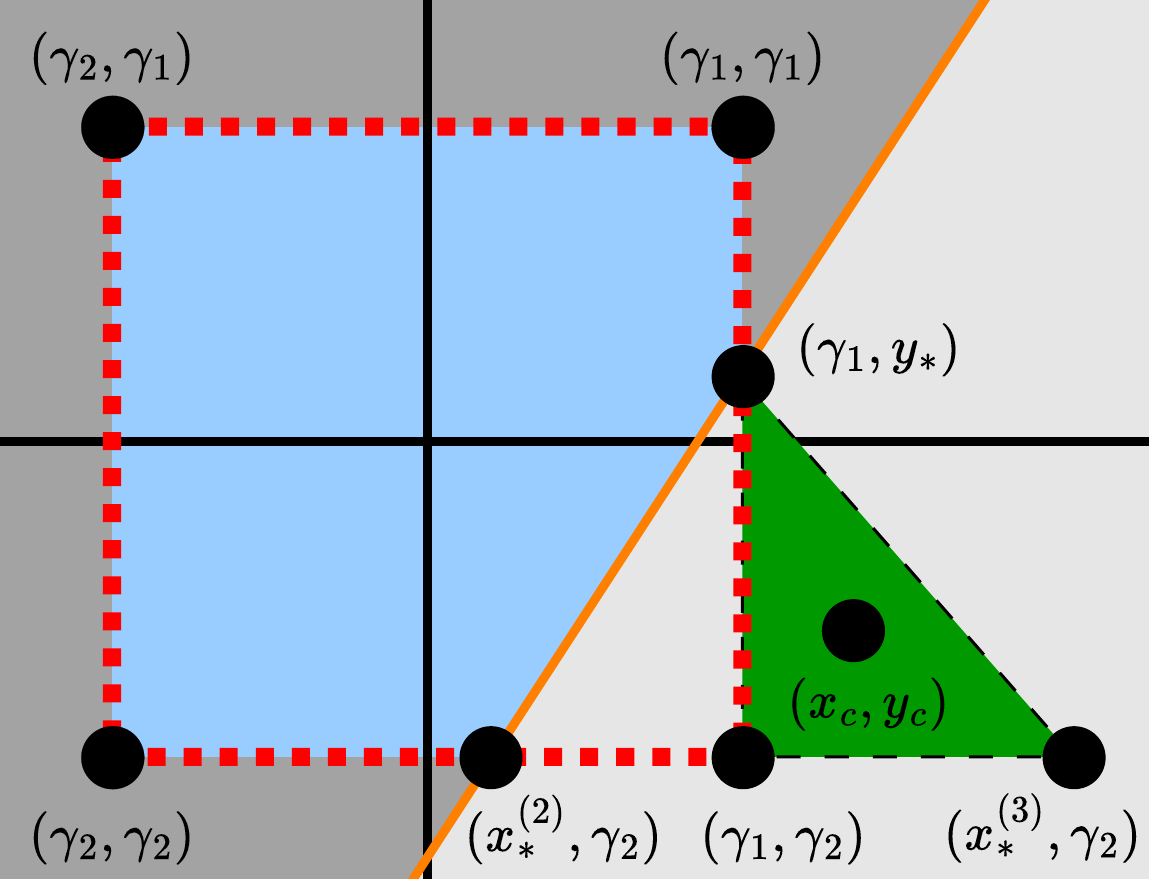}}
    \subfigure[]{\includegraphics[width=2.5in]{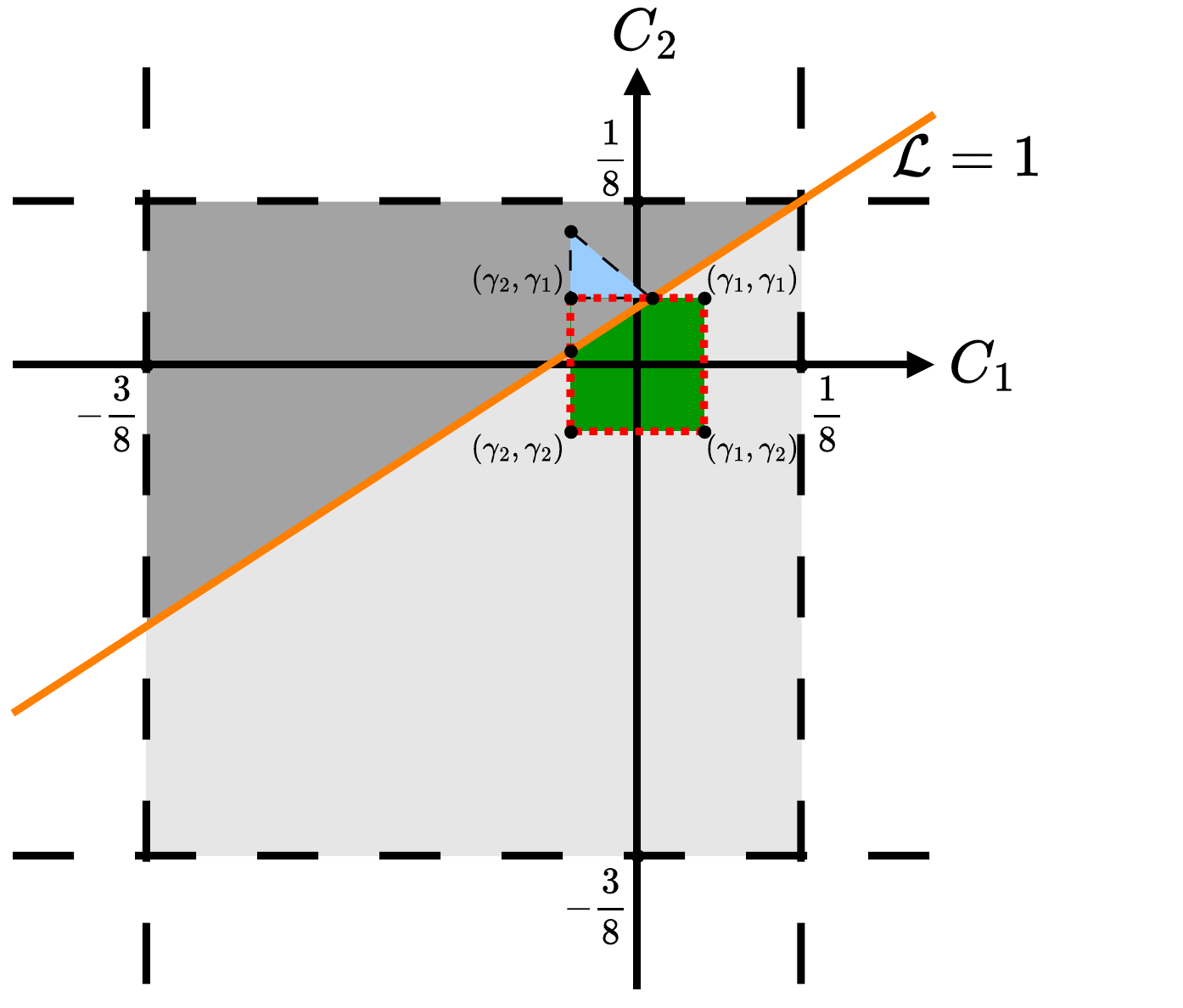}}
    \subfigure[]{\includegraphics[width=2in]{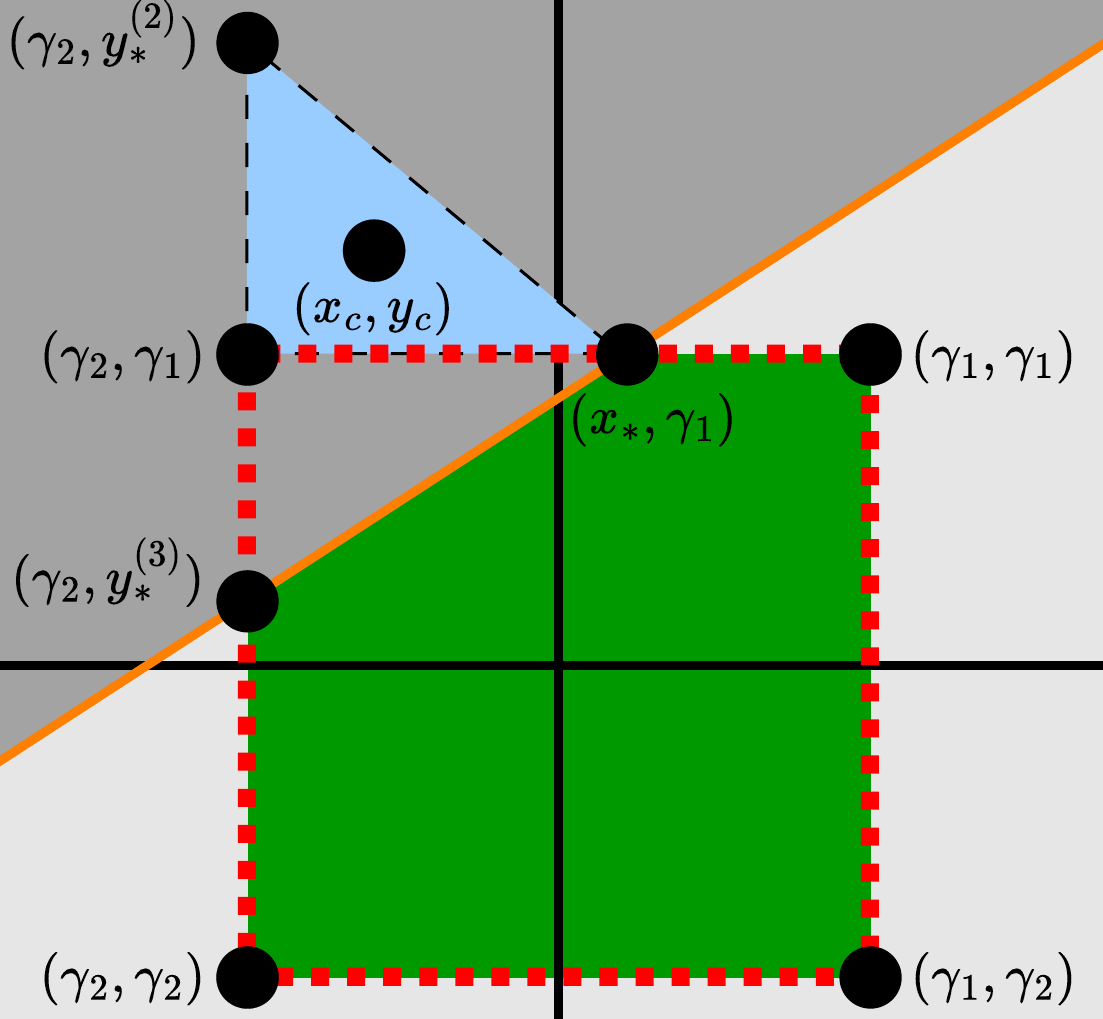}}
    \caption{Feasible region for Case 2 (\textit{dark grey}) and Case 3 (\textit{light grey}). The $\mathcal{O}(h)$ box $\mathcal{H}$ (outlined by red dots) is centered about the origin. The vertex selection algorithm further restricts the search space to the blue and green regions for Case 2 and Case 3, respectively. \textbf{(a)} $\psi^+\in(-1,0)$ with a zoomed in view in \textbf{(b)}, \textbf{(c)} $\psi^+\in(-\infty,-1)$ with a zoomed in view in \textbf{(d)}.}
    \label{fig:case_2_3_feasible_region}
\end{figure}

\begin{algorithm}
\small
\caption{Vertex Selection Algorithm}\label{alg:vertex_selection}
\textbf{Input:} Jump ratios $\theta_{i+1}^-$, $\theta_i^+$ and scaled absolute jumps $|\Delta z_{i-\frac{1}{2}}^*|$, $|\Delta z_{i+\frac{1}{2}}^*|$, $|\Delta z_{i+\frac{3}{2}}^*|$ \\
\textbf{Output:} Set of vertices $\bm{\nu} = \left[\bm{\nu}_1, \bm{\nu}_2,\bm{\nu}_3, \bm{\nu}_4, \bm{\nu}_5\right]$
\begin{algorithmic}[1]
\State Evaluate $\psi^\pm$ according to \eqref{eqn:psi_pm} in the main text, and $\gamma_1,\gamma_2$ according to \eqref{eqn:gammas}
\If{$\theta_{i+1}^-,\theta_{i}^+>1$} \tikzmark{top1} 
    \State $\bm{\nu} = \left[(\frac{1}{8},\frac{1}{8}),(\frac{1}{8},\frac{1}{8}),(\frac{1}{8},\frac{1}{8}),(\frac{1}{8},\frac{1}{8}),(\frac{1}{8},\frac{1}{8})\right]$\tikzmark{bottom1}
\ElsIf{$\theta_{i+1}^->1$, $\theta_{i}^+<1$}\tikzmark{top2} 
    \If{$\psi^+<-1$}
        \State $x_* = \frac{1}{8}(1+\psi^+) - \gamma_1\psi^+$
        \If{$x_*<\gamma_2$}
            \State $\hat{x} = \frac{(\psi^-)^2+\psi^-}{8((\psi^-)^2+1)}$, \hspace{2mm} $\hat{y} = \frac{1}{8}(1+\psi^-) - \hat{x}\psi^-$
            \State $\bm{\nu} = \left[(\hat{x},\hat{y}),(\hat{x},\hat{y}),(\hat{x},\hat{y}),(\hat{x},\hat{y}),(\hat{x},\hat{y})\right]$
        \Else
            \State $y_*^{(2)} = \frac{1}{8}(1+\psi^-) - \gamma_1\psi^-$, \hspace{2mm}$x_c = \frac{1}{3}(2\gamma_2+x_*)$, \hspace{2mm} $y_c = \frac{1}{3}(2\gamma_1+y_*^{(2)})$
            \State $\bm{\nu} = \left[ (\gamma_2,\gamma_1),(x_*,\gamma_1),(\gamma_2,y_*^{(2)}),(x_c,y_c),(x_c,y_c)\right]$
        \EndIf
    \Else
        \State $y_* = \frac{1}{8}(1+\psi^-) - \gamma_1\psi^-$, \hspace{2mm} $x_*^{(2)} = \frac{1}{8}(1+\psi^+) - \gamma_2\psi^+$
        \If{$y_*<\gamma_2$}
            \State $\bm{\nu} = \left[(\gamma_2,\gamma_1),(\gamma_1,\gamma_1),(\gamma_2,\gamma_2),(\gamma_1,\gamma_2),(0,0)\right]$
        \Else
            \State $\bm{\nu} = \left[(\gamma_2,\gamma_1),(\gamma_1,\gamma_1),(\gamma_2,\gamma_2),(\gamma_1,y_*),(x_*^{(2)},\gamma_2)\right]$
        \EndIf
    \EndIf\tikzmark{bottom2}
\ElsIf{$\theta_{i+1}^-<1$, $\theta_{i}^+>1$}\tikzmark{top3}
    \If{$\psi^+<-1$}
        \State $x_* = \frac{1}{8}(1+\psi^+) - \gamma_1\psi^+$, \hspace{2mm} $y_*^{(3)} = \frac{1}{8}(1+\psi^-) - \gamma_2\psi^-$
        \If{$x_*<\gamma_2$}
            \State $\bm{\nu} = \left[(\gamma_2,\gamma_1),(\gamma_1,\gamma_1),(\gamma_2,\gamma_2),(\gamma_1,\gamma_2),(0,0)\right]$
        \Else
            \State $\bm{\nu} = \left[(\gamma_1,\gamma_2),(\gamma_1,\gamma_1),(\gamma_2,\gamma_2),(x_*,\gamma_1),(\gamma_2,y_*^{(3)})\right]$
        \EndIf
    \Else
        \State $y_* = \frac{1}{8}(1+\psi^-) - \gamma_1\psi^-$
        \If{$y_*<\gamma_2$}
            \State $\hat{x} = \frac{(\psi^-)^2+\psi^-}{8((\psi^-)^2+1)}$, \hspace{2mm} $\hat{y} = \frac{1}{8}(1+\psi^-) - \hat{x}\psi^-$
            \State $\bm{\nu} = \left[(\hat{x},\hat{y}),(\hat{x},\hat{y}),(\hat{x},\hat{y}),(\hat{x},\hat{y}),(\hat{x},\hat{y})\right]$
        \Else
            \State $x_*^{(3)} = \frac{1}{8}(1+\psi^+) - \gamma_1\psi^+$, \hspace{2mm} $x_c = \frac{1}{3}(2\gamma_1+x_*^{(3)})$, \hspace{2mm} $y_c = \frac{1}{3}(2\gamma_2+y_*)$\tikzmark{right}
            \State $\bm{\nu} = \left[(\gamma_1,\gamma_2),(x_*^{(3)},\gamma_2),(\gamma_1,y_*),(x_c,y_c),(x_c,y_c)\right]$
        \EndIf
    \EndIf\tikzmark{bottom3}
\ElsIf{$\theta_{i+1}^-==1$, $\theta_{i}^+>1$}\tikzmark{top4}
    \State $\bm{\nu} = \left[(\frac{1}{8},-\frac{3}{8}),(\frac{1}{8},\frac{1}{8}),(\frac{1}{8},\frac{1}{8}),(\frac{1}{8},-\frac{3}{8}),(\frac{1}{8},-\frac{1}{8})\right]$\tikzmark{bottom4}
\ElsIf{$\theta_{i}^+==1$, $\theta_{i+1}^->1$}\tikzmark{top5}
    \State $\bm{\nu} = \left[(\frac{1}{8},\frac{1}{8}),(\frac{1}{8},\frac{1}{8}),(-\frac{3}{8},\frac{1}{8}),(-\frac{3}{8},\frac{1}{8}),(-\frac{1}{8},\frac{1}{8})\right]$\tikzmark{bottom5}
\Else\tikzmark{top6}
    \State $\bm{\nu} = \left[(\frac{1}{8},\frac{1}{8}),(\frac{1}{8},-\frac{3}{8}),(-\frac{3}{8},-\frac{3}{8}),(-\frac{3}{8},\frac{1}{8}),(-\frac{1}{8},-\frac{1}{8})\right]$\tikzmark{bottom6}
\EndIf
\end{algorithmic}
\AddNote{top1}{bottom1}{right}{Case 1}
\AddNote{top2}{bottom2}{right}{Case 2}
\AddNote{top3}{bottom3}{right}{Case 3}
\AddNote{top4}{bottom4}{right}{Case 4a}
\AddNote{top5}{bottom5}{right}{Case 5a}
\AddNote{top6}{bottom6}{right}{Cases 4b, 5b, 6}
\end{algorithm}

\section{Stability bound for \net}\label{sec:sm_stability}
\begin{lem}\label{lem:stability}
\textnormal{(Bounds on jumps)} The \net reconstructed jump satisfies the following estimate:
\begin{equation}\label{eqn:stability}
\left|\llbracket z\rrbracket_{i+\frac{1}{2}}\right| \leq \frac{1}{2}\left|\Delta z_{i-\frac{1}{2}}\right| + \left|\Delta z_{i+\frac{1}{2}}\right| + \frac{1}{2}\left|\Delta z_{i+\frac{3}{2}}\right| \hspace{3mm} \forall\hspace{1mm} i\in\mathbb{Z}.
\end{equation}
\end{lem}

\begin{proof}
If $\Delta z_{i+\frac{1}{2}}=0$, the bound will clearly be satisfied since $|\llbracket z\rrbracket_{i+\frac{1}{2}}|=0$ in this case by construction of \net. Hence, we assume $\Delta z_{i+\frac{1}{2}}\neq0$. We start by
using \eqref{eqn:spweno_wts} to rewrite the expression of $\llbracket z\rrbracket_{i+\frac{1}{2}}$ in \eqref{recon_jump} as  
\begin{flalign*}
\llbracket z\rrbracket_{i+\frac{1}{2}} = \left(C_1-\frac{1}{8}\right)\Delta z_{i-\frac{1}{2}} + \left(\frac{1}{4}-C_1-C_2\right)\Delta z_{i+\frac{1}{2}} + \left(C_2-\frac{1}{8}\right)\Delta z_{i+\frac{3}{2}}.
\end{flalign*}
By consistency, we have $-3/8 \leq C_1,C_2 \leq 1/8$. Thus, the absolute value of the jump can be bounded as
\begin{flalign*}
\left|\llbracket z\rrbracket_{i+\frac{1}{2}}\right| &\leq \left|C_1-\frac{1}{8}\right|\left|\Delta z_{i-\frac{1}{2}}\right| + \left|\frac{1}{4}-C_1-C_2\right|\left|\Delta z_{i+\frac{1}{2}}\right| + \left|C_2-\frac{1}{8}\right|\left|\Delta z_{i+\frac{3}{2}}\right| \\
 &\leq \frac{1}{2}\left|\Delta z_{i-\frac{1}{2}}\right| + \left|\Delta z_{i+\frac{1}{2}}\right| + \frac{1}{2}\left|\Delta z_{i+\frac{3}{2}}\right|.
\end{flalign*}
\end{proof}

\begin{remark}
    The estimate \eqref{eqn:stability} is not unique to \net and is a consequence of the sign property constraint \eqref{eqn:sign_prop_constraint}. Thus, it is also satisfied by every variant of SP-WENO. However, the original SP-WENO satisfies a sharper bound which is attributed to the fact that the reconstructed jump is zero in most cases \cite{fjordholm2016sign}.
\end{remark}

\section{Additional Numerical Results and Details}

\subsection{Empirical study of reconstruction jumps}\label{sec:rec_jump}
We consider a continuous function with a sharp connection between two linear components to mimic a discontinuity:

\begin{flalign*}
u(x) = \begin{cases} \frac{1}{2}x &\text{ if } -0.5\leq x<0\\ \frac{x}{\epsilon} &\text{ if } 0\leq x\leq\epsilon \\ x-\epsilon+1 &\text{ if } \epsilon<x\leq0.5 \end{cases},
\end{flalign*}
where $\epsilon$ is chosen as 0.001. We reconstruct on a mesh consisting of 1000 cells, ensuring that a single cell is sitting inside the jump. In other words, we are simulating a scenario where the discontinuity is not perfectly resolved. Figure \ref{fig:recon_test} shows the left and right reconstructed interface values for ENO3, SP-WENO, SP-WENOc, and \net, zoomed into the region close to the simulated discontinuity. Figure \ref{fig:recon_test} also features bar plots that show the actual jump in cell center values and the reconstructed jump at the interface for each method. The left and right reconstructed values are almost identical outside of the $0\leq x\leq\epsilon$ region for all the methods, and thus not visible in the figure given the scale of the jump. However, inside the pseudo-shock where a cell (let us call it $I_0$) lies, SP-WENO and SP-WENOc do not fully capture the non-zero jumps. At the left interface of $I_0$, SP-WENO results in a very small reconstructed jump, while SP-WENOc leads to a larger jump but which is still an order of magnitude smaller as compared to the jump at the right interface of $I_0$. On the other hand, ENO3 and \net give rise to larger reconstructed jumps at both interfaces of $I_0$, which is the desirable behavior in the vicinity of discontinuities. The impact of the reconstructed jump magnitude can be observed from the numerical results for conservation laws shown in Section \ref{sec:numerical} in the main text.

\begin{figure}
     \centering
     \includegraphics[width=5in]{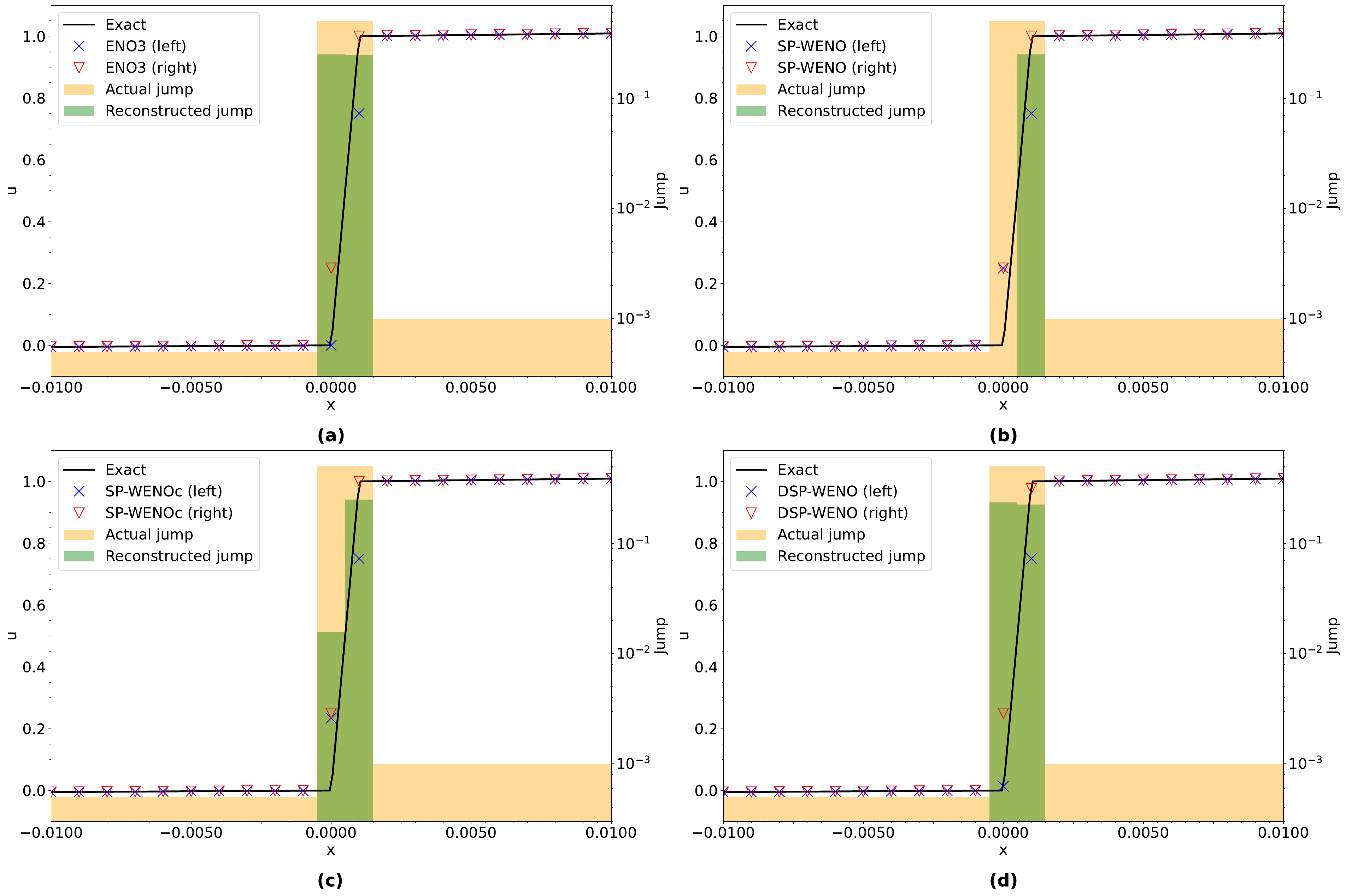}
     \caption{Reconstruction of a simulated discontinuity, which is not well-resolved, using: \textbf{(a)} ENO3, \textbf{(b)} SP-WENO, \textbf{(c)} SP-WENOc, \textbf{(d)} \net.}
     \label{fig:recon_test}
 \end{figure}

\subsection{Linear Advection}\label{sec:sm_LinAdv}

Tables \ref{tab:lin_adv_l2_SM} and \ref{tab:lin_adv_linf_SM} show the errors (measured in the discrete $L^2$ and $L^{\infty}$ norm, respectively) with various reconstructions for both smooth linear advection test cases. As with $L^1$, we observe third-order accuracy in $L^2$ in both test cases (except for ENO 3 in Test case 2). There seems to be slight deterioration in the $L^\infty$-order of convergence (which is much more severe for ENO3 in Test case 2).

\begin{table}[]
    \centering
    \begin{tabular}{|c|c|c c|c c|c c|c c|}
     \hline
     &N & \multicolumn{2}{c|}{ENO3} & \multicolumn{2}{c|}{SP-WENO} & \multicolumn{2}{c|}{SP-WENOc} & \multicolumn{2}{c|}{\net} \\
          & & Error & Rate & Error & Rate & Error & Rate & Error & Rate \\
     \hline
     \parbox[t]{2mm}{\multirow{6}{*}{\rotatebox[origin=c]{90}{Test 1}}} &100  & 1.43e-5 & -    & 6.08e-5 & -    & 5.69e-5 & -    & 5.77e-5 & - \\
     &200  & 1.79e-6 & 3.00 & 8.40e-6 & 2.86 & 7.91e-6 & 2.85 & 5.79e-6 & 3.32 \\
     &400  & 2.24e-7 & 3.00 & 1.17e-6 & 2.85 & 1.13e-6 & 2.81 & 6.89e-7 & 3.07 \\
     &600  & 6.63e-8 & 3.00 & 3.64e-7 & 2.87 & 3.60e-7 & 2.82 & 2.05e-7 & 2.99 \\
     &800  & 2.80e-8 & 3.00 & 1.59e-7 & 2.87 & 1.57e-7 & 2.88 & 8.74e-8 & 2.96 \\
     &1000 & 1.43e-8 & 3.00 & 8.40e-8 & 2.86 & 8.33e-8 & 2.84 & 4.47e-8 & 3.01 \\
     \hline
     \parbox[t]{2mm}{\multirow{6}{*}{\rotatebox[origin=c]{90}{Test 2}}} &100  & 7.02e-4 & -    & 9.88e-4 & -    & 9.71e-4 & -    & 1.14e-3 & - \\
     &200  & 9.44e-5 & 2.90 & 1.35e-4 & 2.87 & 1.34e-4 & 2.86 & 1.37e-4 & 3.05 \\
     &400  & 1.24e-5 & 2.93 & 1.83e-5 & 2.89 & 1.81e-5 & 2.89 & 1.68e-5 & 3.03 \\
     &600  & 4.04e-6 & 2.76 & 5.78e-6 & 2.84 & 5.76e-6 & 2.82 & 5.38e-6 & 2.80 \\
     &800  & 2.51e-6 & 1.66 & 2.50e-6 & 2.92 & 2.48e-6 & 2.94 & 2.49e-6 & 2.67 \\
     &1000 & 2.15e-6 & 0.68 & 1.33e-6 & 2.83 & 1.31e-6 & 2.84 & 1.32e-6 & 2.85 \\
     \hline
    \end{tabular}
    \caption{$L^2$ errors and convergence rates for linear advection smooth tests 1 and 2.}
    \label{tab:lin_adv_l2_SM}
\end{table}

\begin{table}[]
    \centering
    \begin{tabular}{|c|c|c c|c c|c c|c c|}
     \hline
     &N & \multicolumn{2}{c|}{ENO3} & \multicolumn{2}{c|}{SP-WENO} & \multicolumn{2}{c|}{SP-WENOc} & \multicolumn{2}{c|}{\net} \\
          & & Error & Rate & Error & Rate & Error & Rate & Error & Rate \\
     \hline
     \parbox[t]{2mm}{\multirow{6}{*}{\rotatebox[origin=c]{90}{Test 1}}} &100  & 9.20e-6 & -    & 1.05e-4 & -    & 9.77e-5 & -    & 1.13e-4 & - \\
     &200  & 1.10e-6 & 3.06 & 2.02e-5 & 2.38 & 1.91e-5 & 2.36 & 1.11e-5 & 3.35 \\
     &400  & 1.42e-7 & 2.96 & 3.56e-6 & 2.50 & 3.47e-6 & 2.46 & 1.71e-6 & 2.69 \\
     &600  & 4.21e-8 & 2.99 & 1.24e-6 & 2.60 & 1.23e-6 & 2.56 & 5.94e-7 & 2.61 \\
     &800  & 1.73e-8 & 3.10 & 6.42e-7 & 2.30 & 6.33e-7 & 2.31 & 2.98e-7 & 2.39 \\
     &1000 & 8.88e-9 & 2.98 & 3.83e-7 & 2.31 & 3.80e-7 & 2.28 & 1.70e-7 & 2.51 \\
     \hline
     \parbox[t]{2mm}{\multirow{6}{*}{\rotatebox[origin=c]{90}{Test 2}}} &100  & 6.38e-4 & -    & 1.33e-3 & -    & 1.30e-3 & -    & 1.75e-3 & - \\
     &200  & 9.22e-5 & 2.79 & 2.21e-4 & 2.59 & 2.19e-4 & 2.57 & 2.07e-4 & 3.08 \\
     &400  & 1.29e-5 & 2.84 & 4.16e-5 & 2.41 & 4.14e-5 & 2.40 & 2.82e-5 & 2.88 \\
     &600  & 4.44e-6 & 2.62 & 1.45e-5 & 2.59 & 1.45e-5 & 2.59 & 1.15e-5 & 2.21 \\
     &800  & 3.56e-6 & 0.76 & 7.62e-6 & 2.24 & 7.49e-6 & 2.29 & 6.08e-6 & 2.21 \\
     &1000 & 3.24e-6 & 0.43 & 4.45e-6 & 2.41 & 4.37e-6 & 2.41 & 3.84e-6 & 2.06 \\
     \hline
    \end{tabular}
    \caption{$L^{\infty}$ errors and convergence rates for linear advection smooth tests 1 and 2.}
    \label{tab:lin_adv_linf_SM}
\end{table}

\subsection{2D Riemann problem (conf. 12)}\label{sec:sm_2DRiemann_conf12}

Figure \ref{fig:2D_riemann12_SM} shows the density contour plots for the 2D Riemann problem (configuration 12) when solved using ENO3, SP-WENO, SP-WENOc, and \net.

\begin{figure}
    \centering
    \subfigure[ENO3]{\includegraphics[width=2.5in]{2DR_Conf12_eno3.png}}
    \subfigure[SP-WENO]{\includegraphics[width=2.5in]{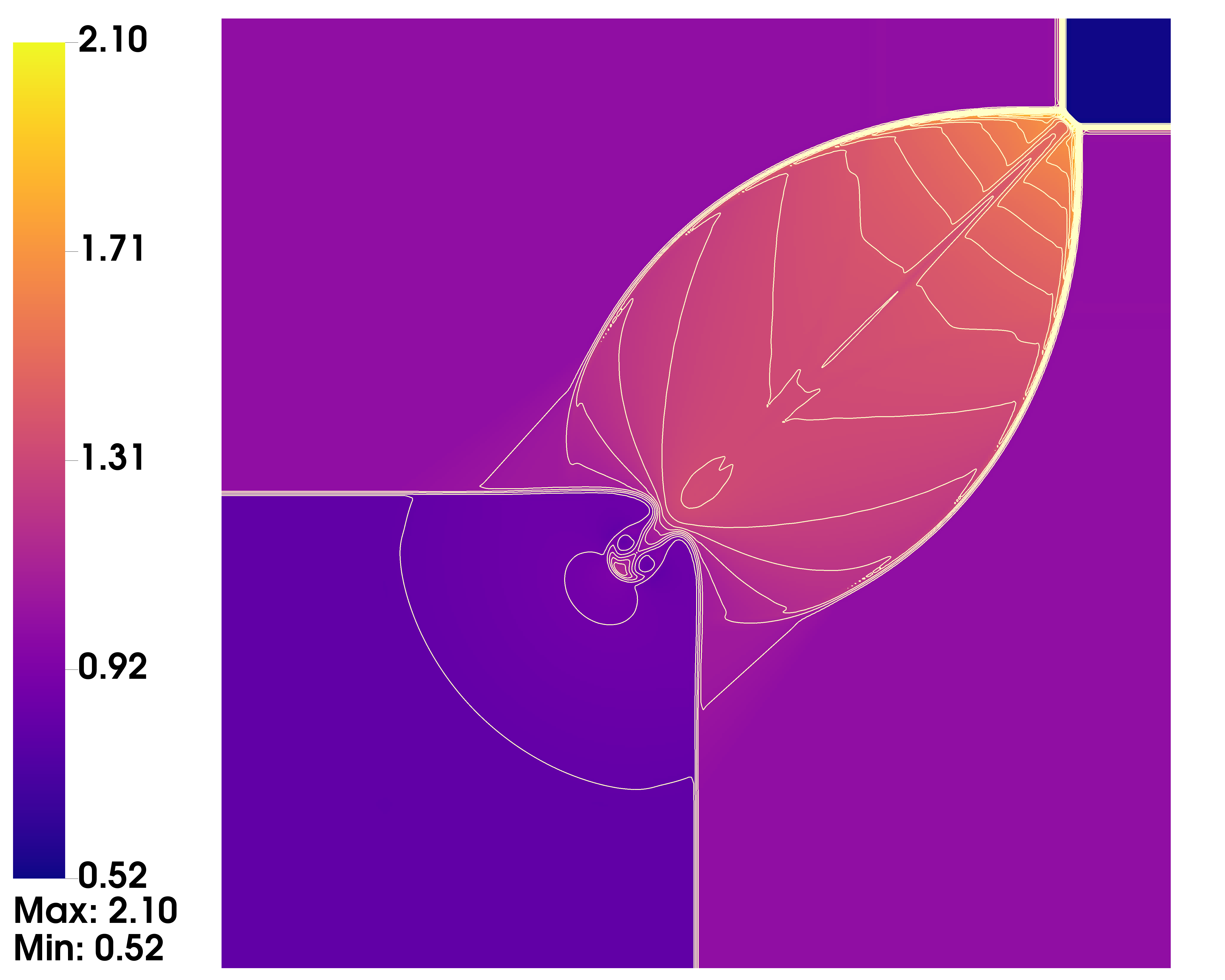}}
    \subfigure[SP-WENOc]{\includegraphics[width=2.5in]{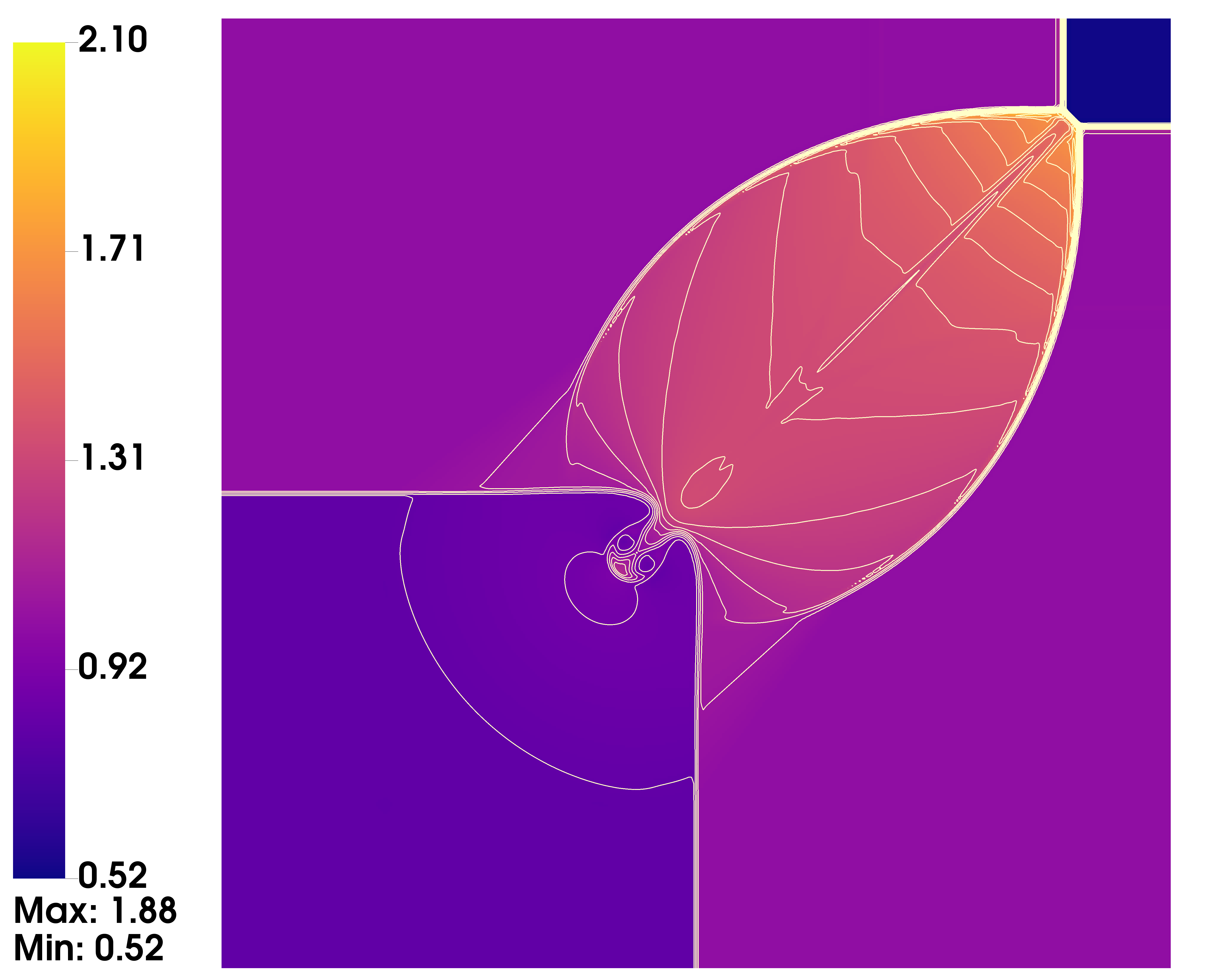}}
    \subfigure[\net]{\includegraphics[width=2.5in]{2DR_Conf12_sp_weno_dl_finalMSE.png}}
    \caption{2D Riemann problem (configuration 12): Density profiles at time $T=0.25$ with 30 contour lines between 0.52 and 2.2. Comparison of different reconstruction methods}
    \label{fig:2D_riemann12_SM}
\end{figure}

\subsection{2D Riemann problem (conf. 3)}\label{sec:sm_2DRiemann_conf3}

Figure \ref{fig:2D_riemann3_rus_SM} shows the density contour plots for the 2D Riemann problem (configuration 3) with Rusanov dissipation when using ENO3, SP-WENO, SP-WENOc, and \net. Figure \ref{fig:2D_riemann3_zoomed_SM} shows a zoomed-in perspective of the carbuncle-like artifact for SP-WENO and \net with the Roe and Rusanov dissipation. We observe that switching to the Rusanov dissipation mitigates the behavior.

\begin{figure}
    \centering
    \subfigure[ENO3]{\includegraphics[width=2.5in]{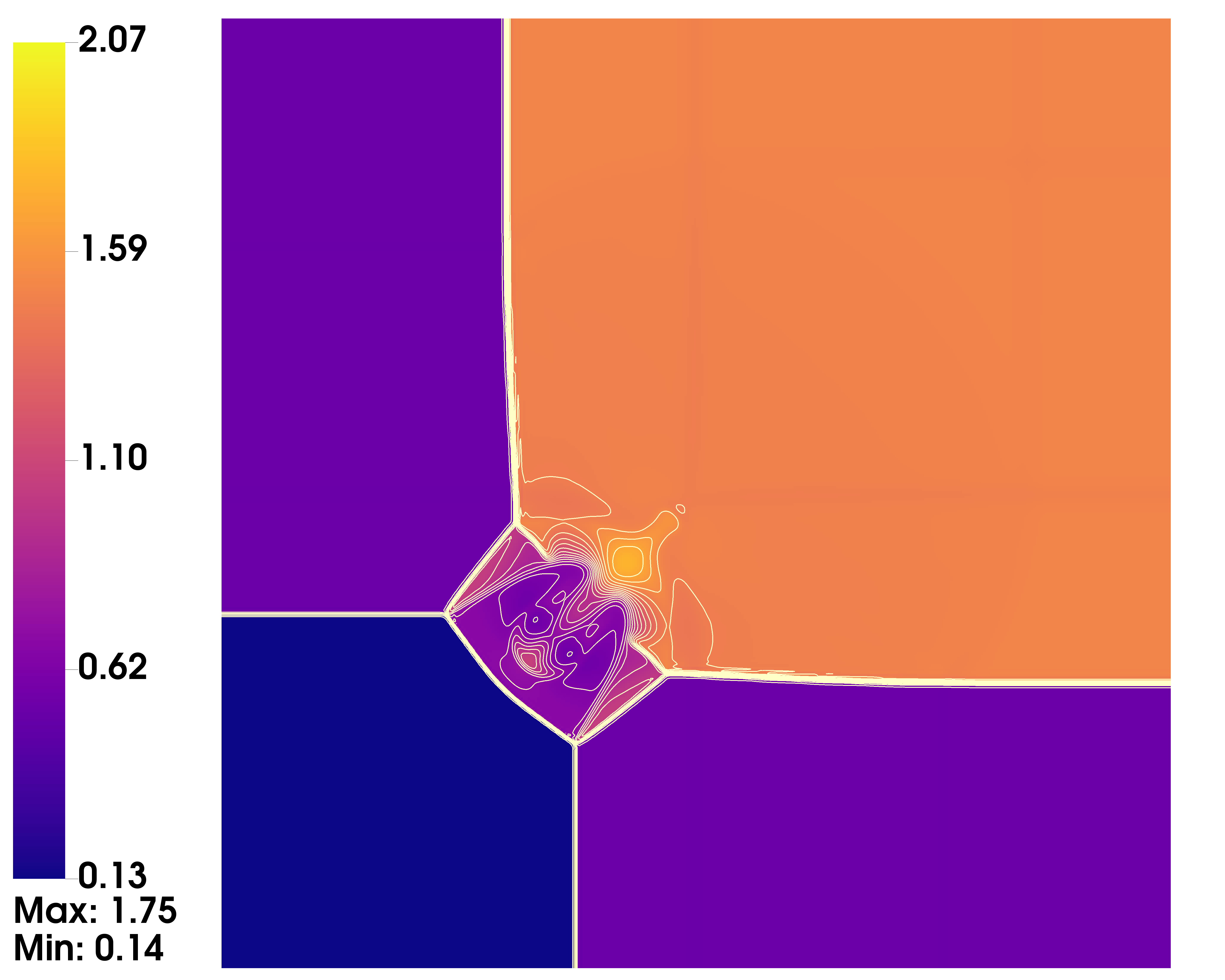}}
    \subfigure[SP-WENO]{\includegraphics[width=2.5in]{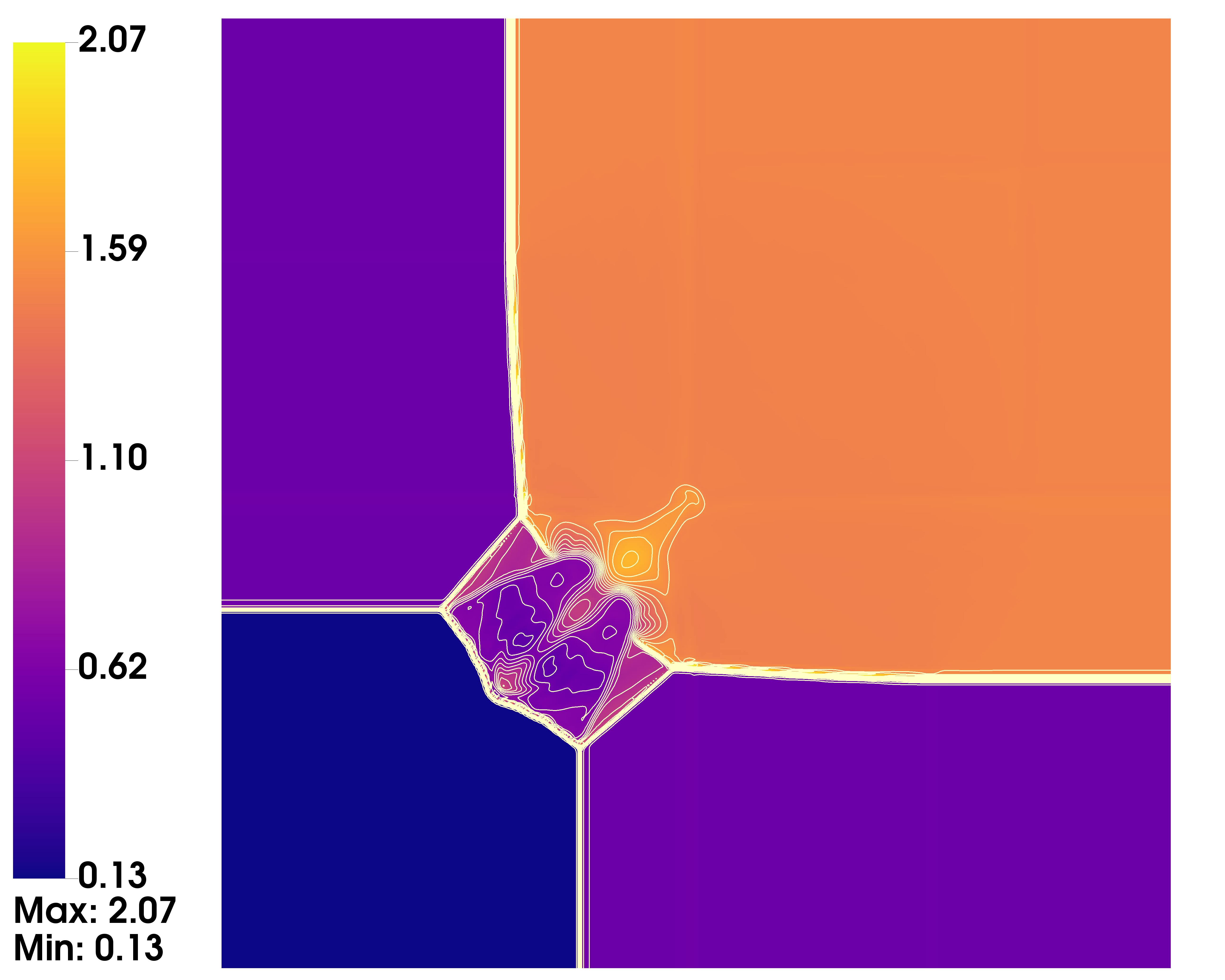}}
    \subfigure[SP-WENOc]{\includegraphics[width=2.5in]{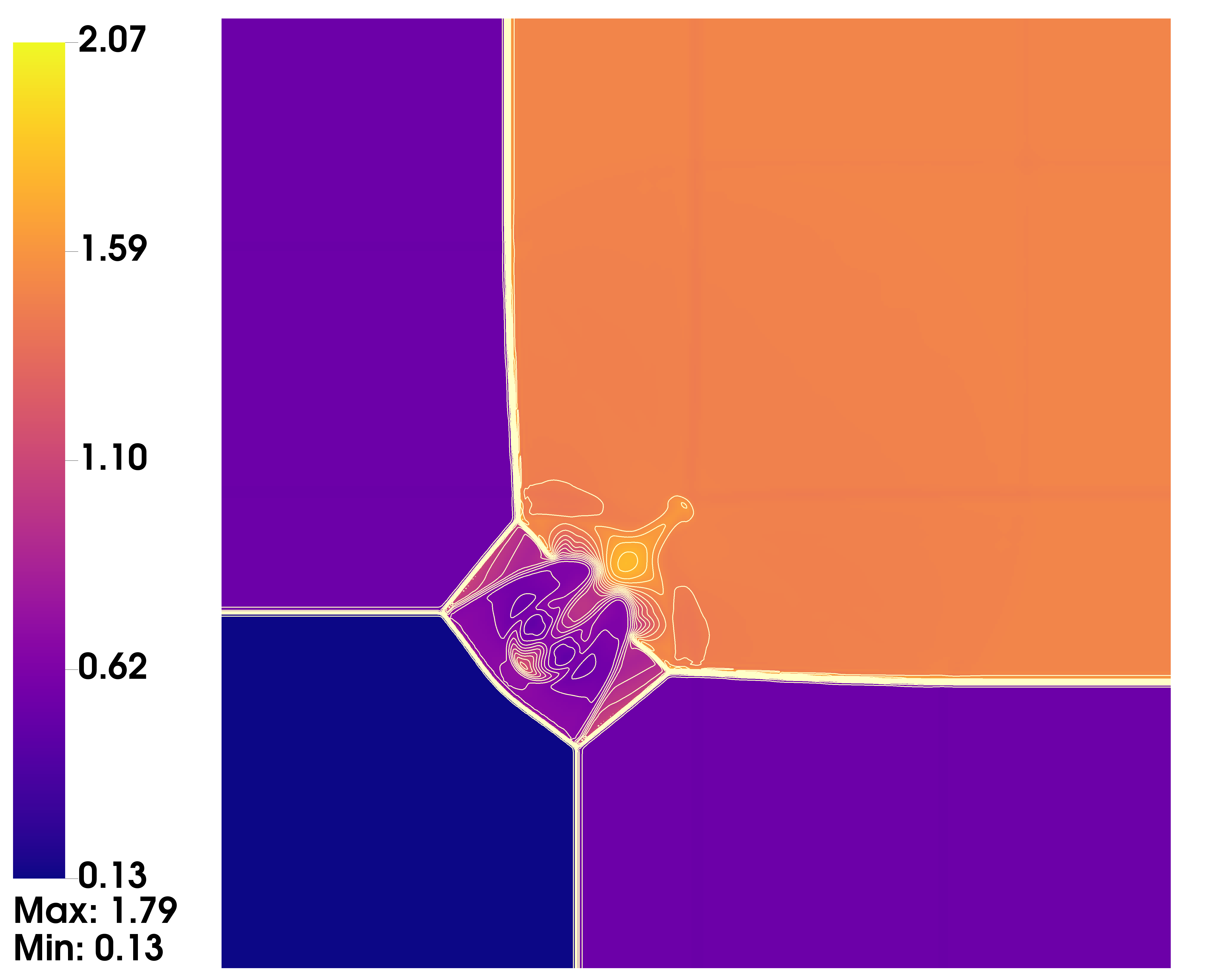}}
    \subfigure[\net]{\includegraphics[width=2.5in]{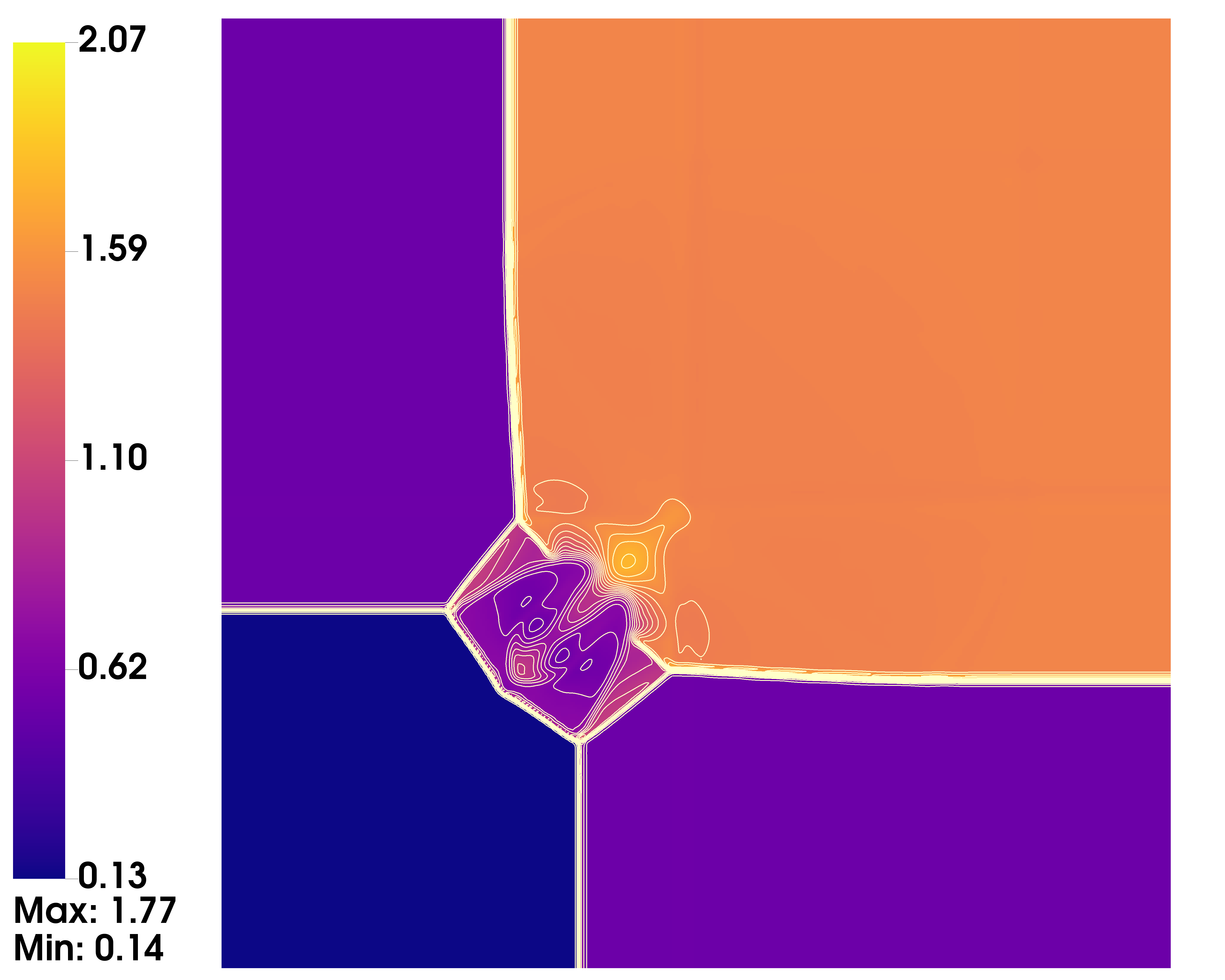}}
    \caption{2D Riemann problem (configuration 3): Density profiles at time $T=0.3$ with 30 contour lines between 0.1 and 2.28.  Comparison of different reconstruction methods}
    \label{fig:2D_riemann3_rus_SM}
\end{figure}

\begin{figure}
    \centering
    \subfigure[SP-WENO Roe]{\includegraphics[width=2.5in]{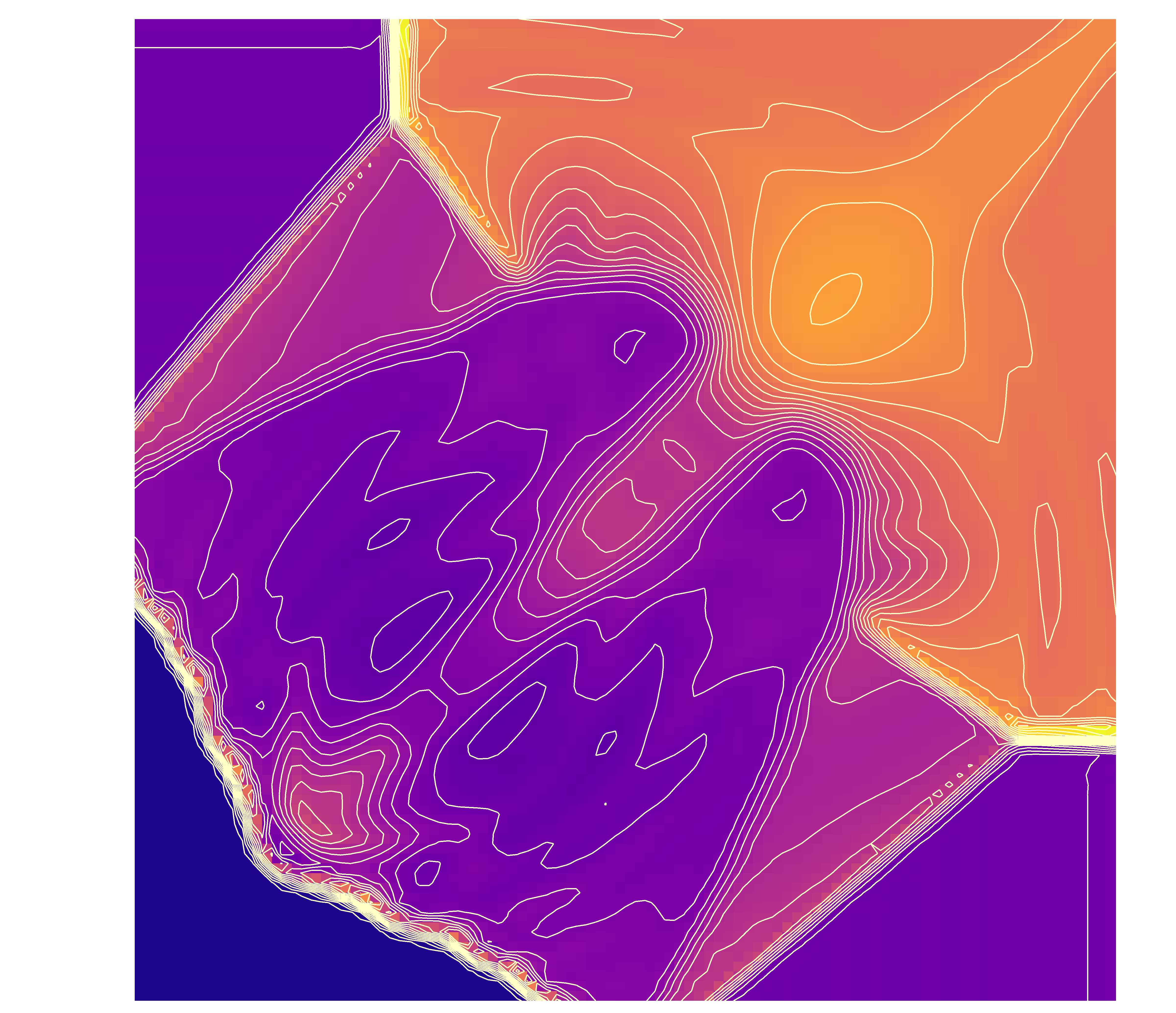}}
    \subfigure[\net Roe]{\includegraphics[width=2.5in]{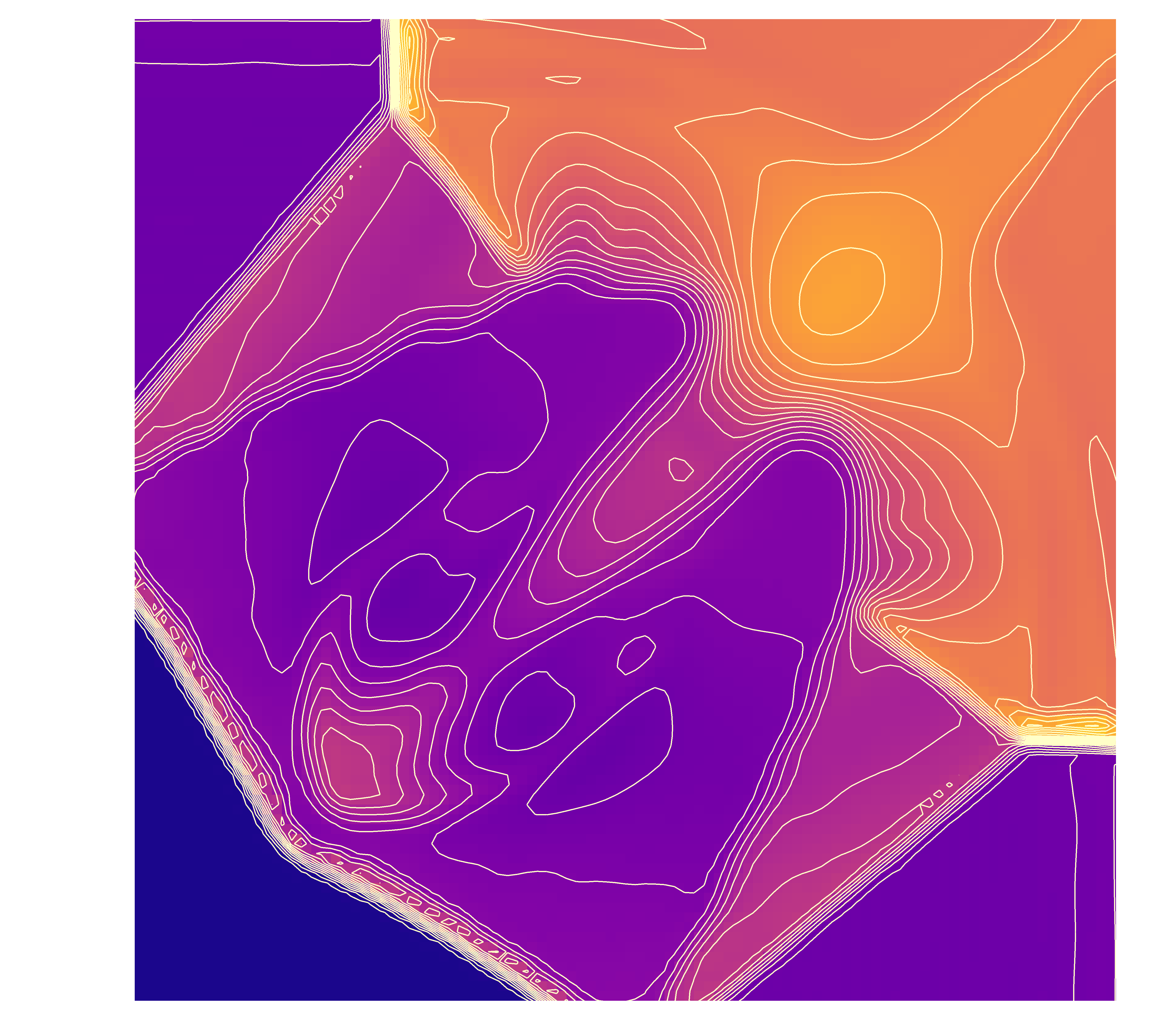}}
    \subfigure[SP-WENO Rusanov]{\includegraphics[width=2.5in]{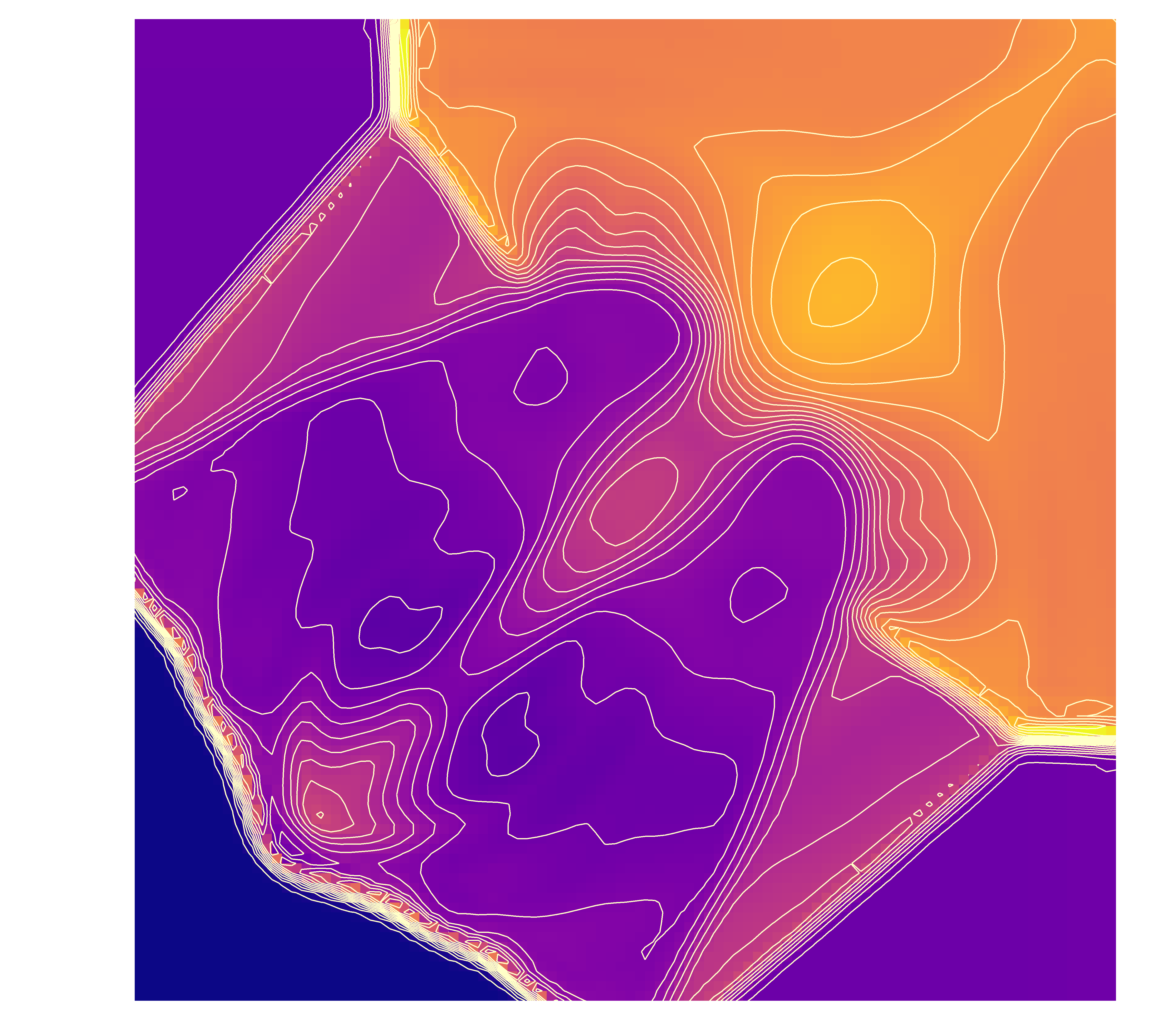}}
    \subfigure[\net Rusanov]{\includegraphics[width=2.5in]{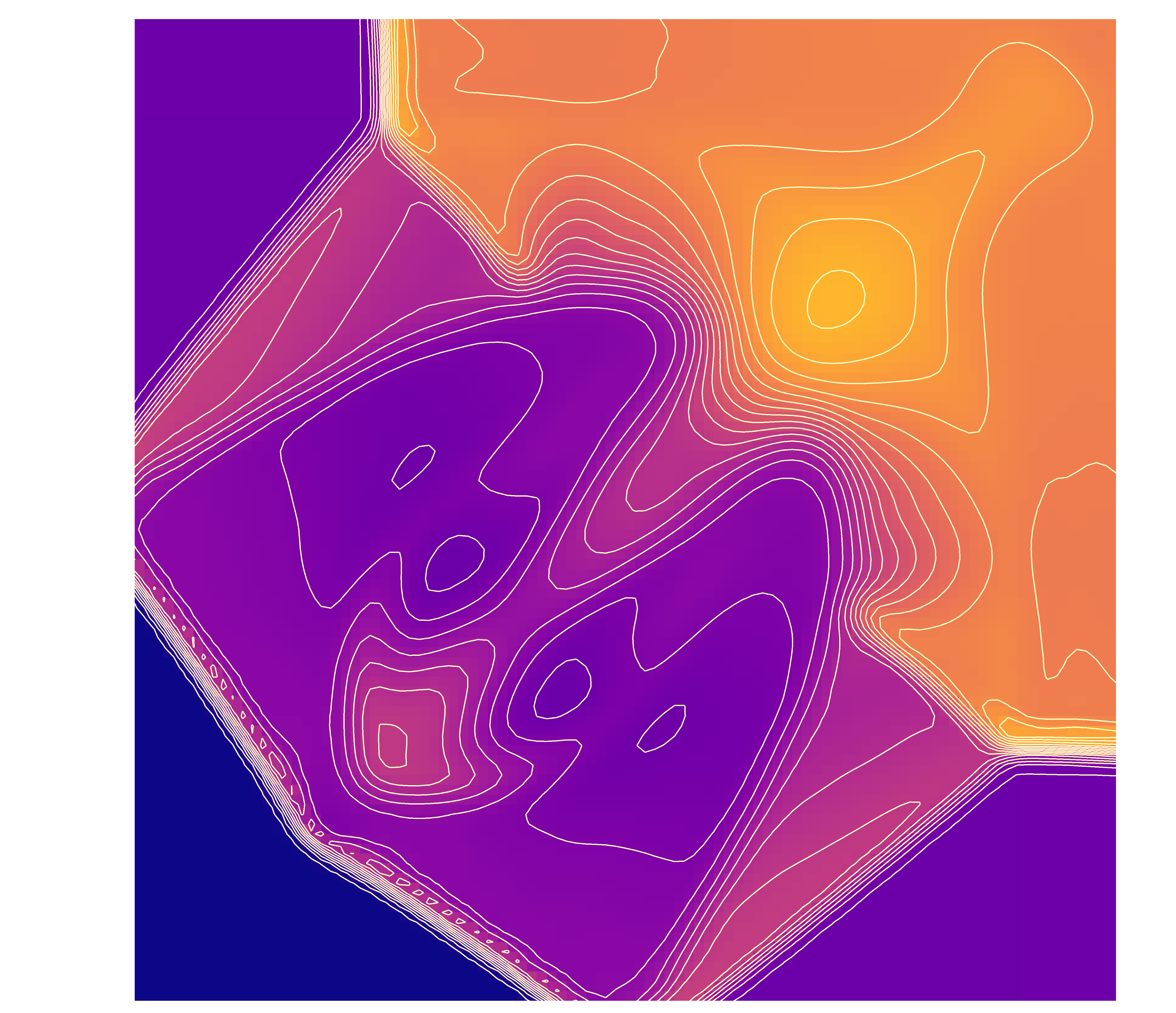}}
    \caption{2D Riemann problem (configuration 3): Zoom-in of the carbuncle-like phenomenon for SP-WENO and \net with Roe and Rusanov dissipation.}
    \label{fig:2D_riemann3_zoomed_SM}
\end{figure}

\subsection{Kelvin-Helmholtz Instability}\label{sec:sm_kelvin-helmholtz}

Figures \ref{fig:kelvin-helmholtz-256_SM} and \ref{fig:kelvin-helmholtz-512_SM} show the density profiles for the Kelvin-Helmholtz instability when solved using ENO3, SP-WENO, SP-WENOc, and \net for  meshes consisting of 256 $\times$ 256 cells and 512 $\times$ 512 cells.

\begin{figure}
    \centering
    \subfigure[ENO3]{\includegraphics[width=2.5in]{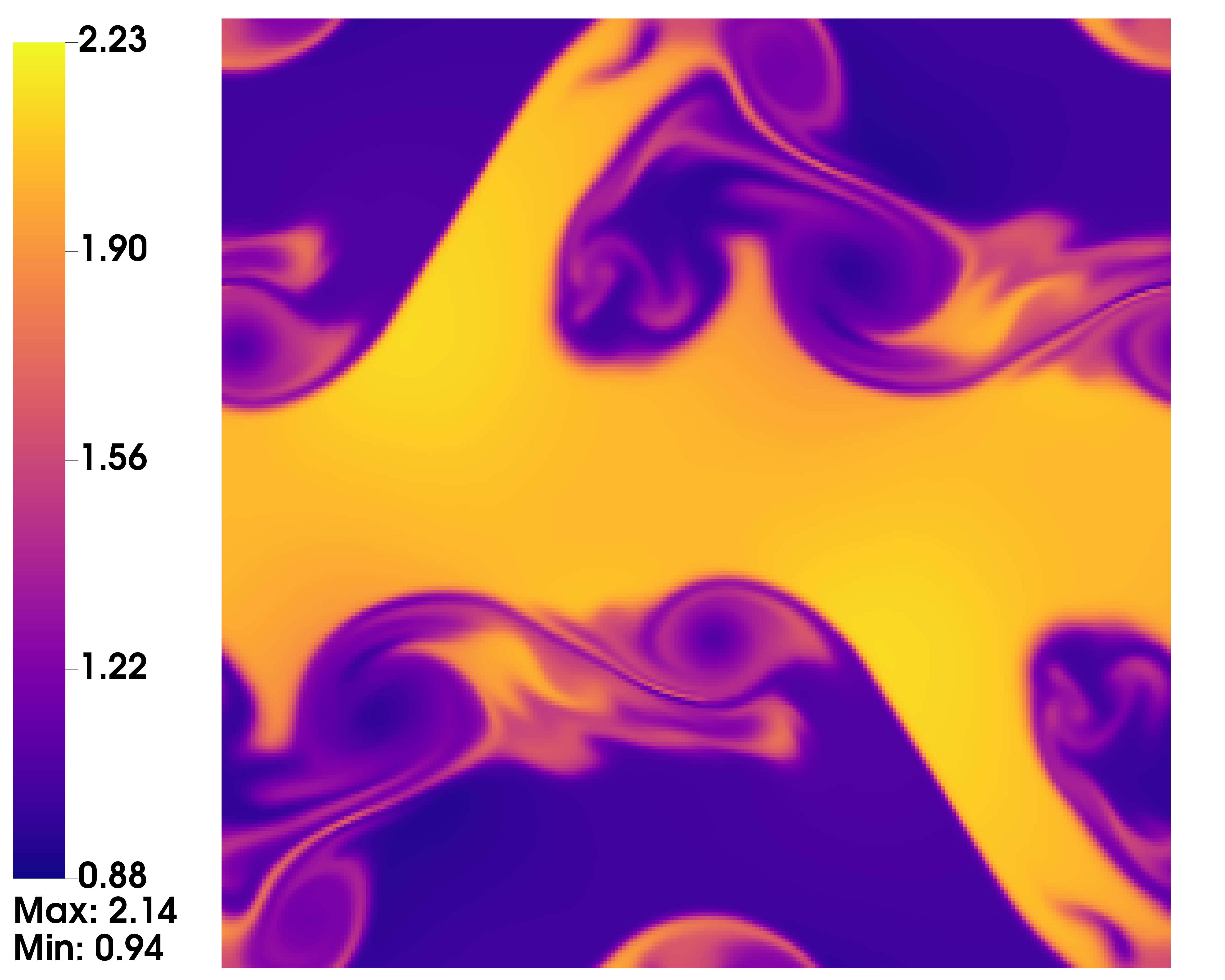}}
    \subfigure[SP-WENO]{\includegraphics[width=2.5in]{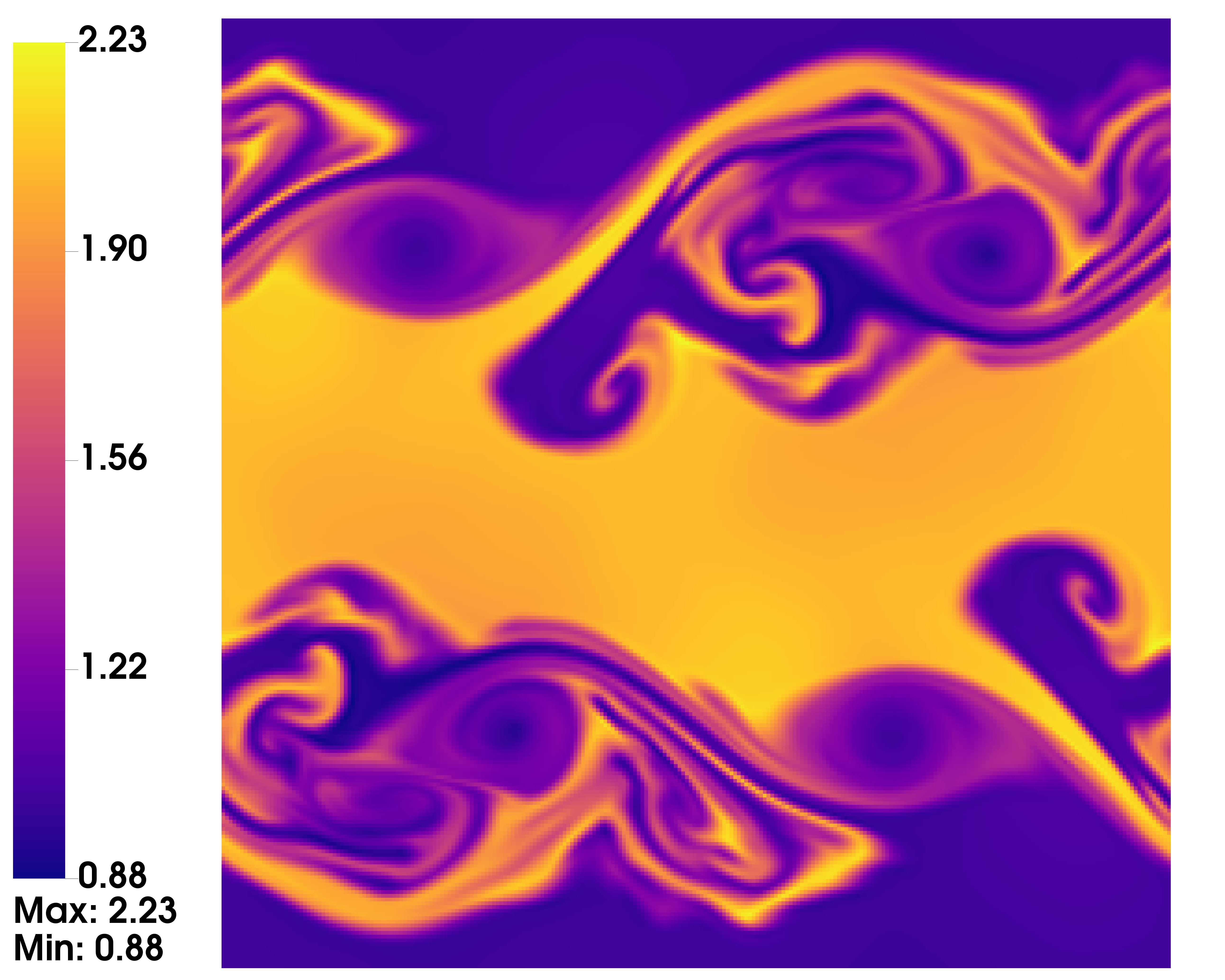}}
    \subfigure[SP-WENOc]{\includegraphics[width=2.5in]{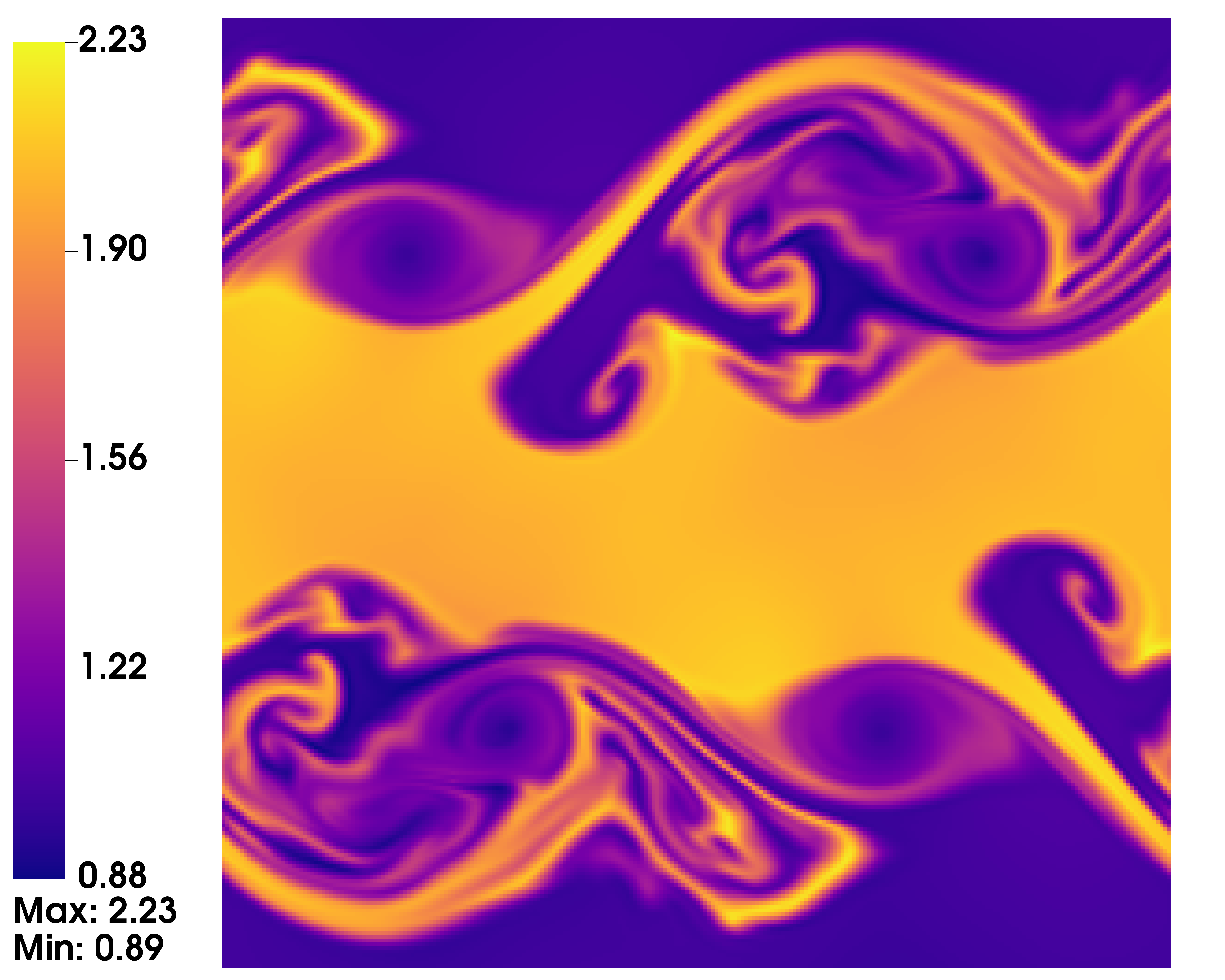}}
    \subfigure[\net]{\includegraphics[width=2.5in]{Kelvin-Helmholtz_256_sp_weno_dl_finalMSE.png}}
    \caption{Kelvin-Helmholtz Instability: Density profiles at time $T=3$. Comparison of different reconstruction methods for $256\times256$ mesh.}
    \label{fig:kelvin-helmholtz-256_SM}
\end{figure}

\begin{figure}
    \centering
    \subfigure[ENO3]{\includegraphics[width=2.5in]{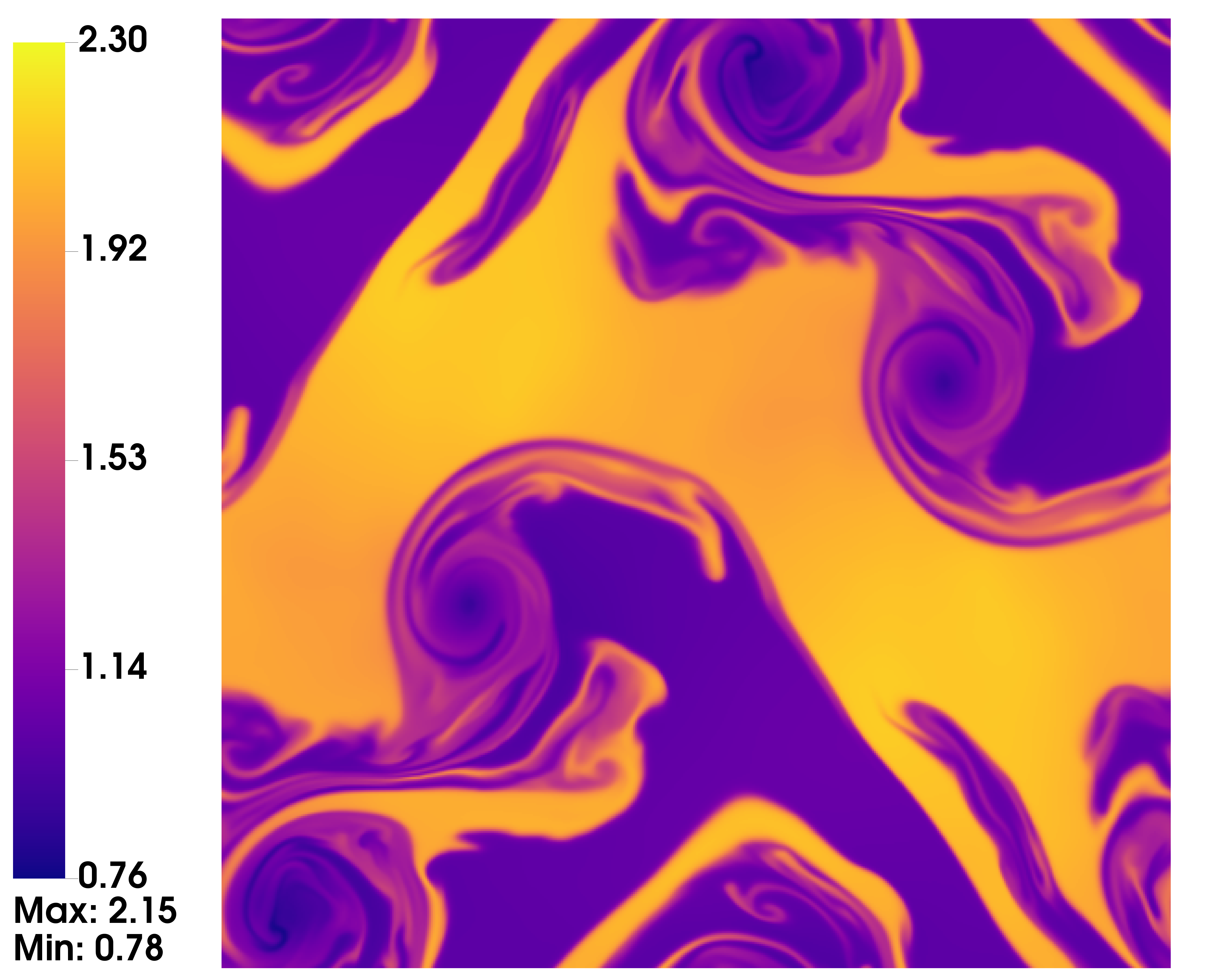}}
    \subfigure[SP-WENO]{\includegraphics[width=2.5in]{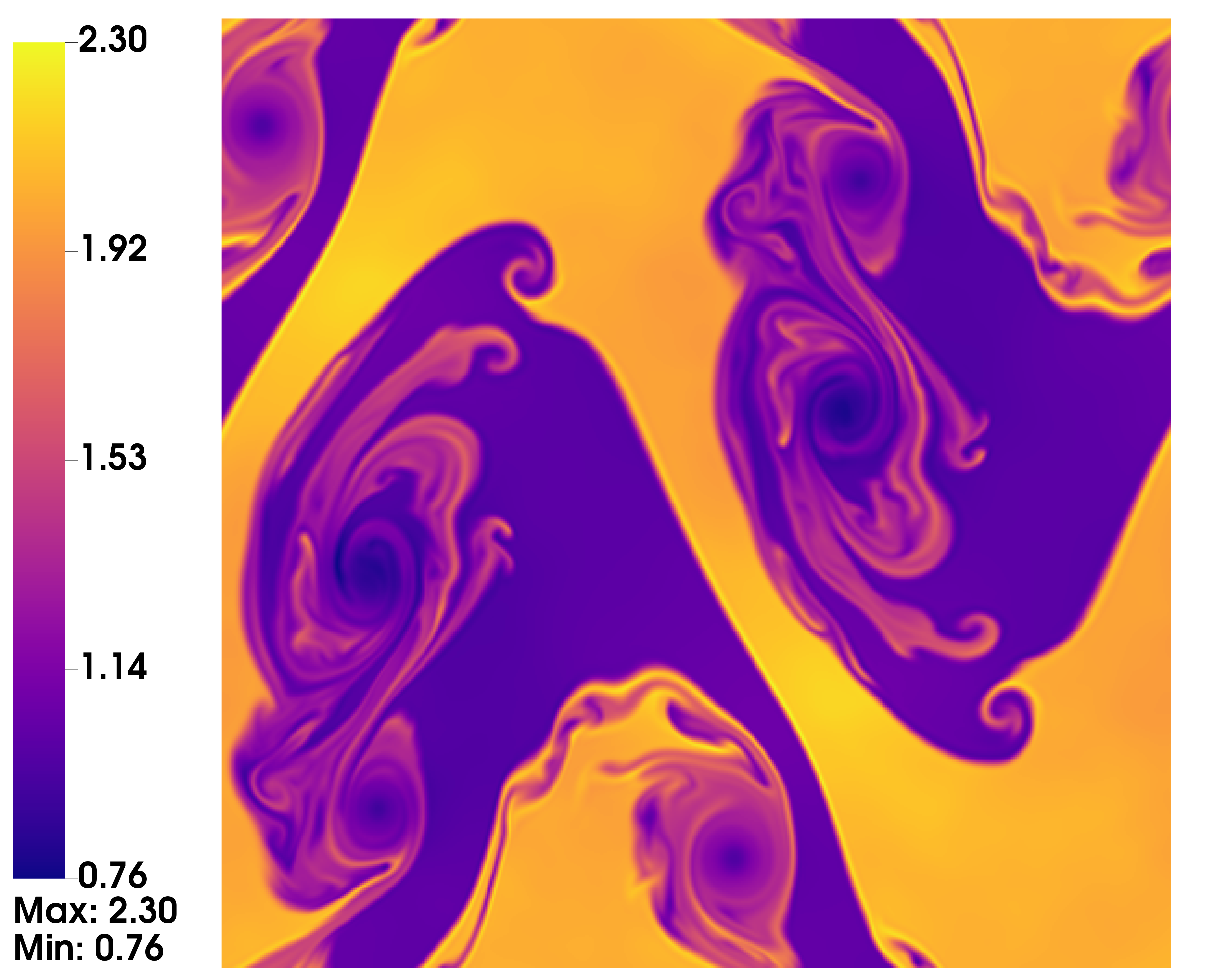}}
    \subfigure[SP-WENOc]{\includegraphics[width=2.5in]{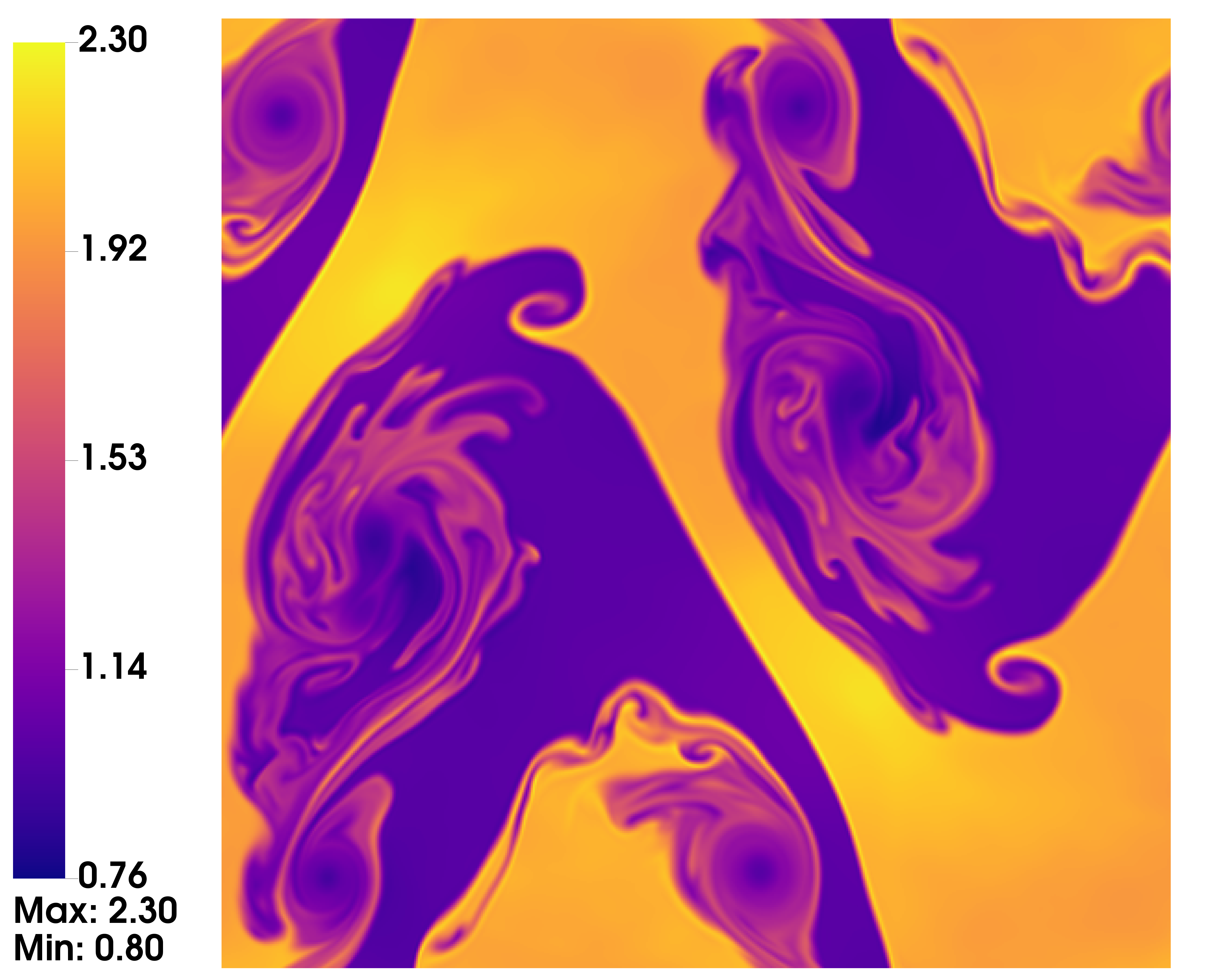}}
    \subfigure[\net]{\includegraphics[width=2.5in]{Kelvin-Helmholtz_512_sp_weno_dl_finalMSE.png}}
    \caption{Kelvin-Helmholtz Instability: Density profiles at time $T=3$. Comparison of different reconstruction methods for $512\times512$ mesh.}
    \label{fig:kelvin-helmholtz-512_SM}
\end{figure}

\end{document}